\documentclass[12pt]{amsart}
\usepackage{geometry}
\usepackage[all]{xy}
\usepackage{graphicx}
\usepackage{amssymb, extarrows}
\usepackage{epstopdf}
\usepackage{xr}

\makeatletter
\def\timenow{\@tempcnta\time
  \@tempcntb\@tempcnta
  \divide\@tempcntb60
  \ifnum10>\@tempcntb0\fi\number\@tempcntb
  \multiply\@tempcntb60
  \advance\@tempcnta-\@tempcntb
  :\ifnum10>\@tempcnta0\fi\number\@tempcnta}
\makeatother

\newtheorem{thm}{Theorem}[section]

\newtheorem{problem}{Problem}

\newtheorem{cor}[thm]{Corollary}
\newtheorem{sett}[thm]{Setting}
\newtheorem{lem}[thm]{Lemma}
\newtheorem{prop}[thm]{Proposition}
\newtheorem{fact}[thm]{Fact}
\newtheorem{example}[thm]{Example}
\theoremstyle{definition}
\newtheorem{defn}[thm]{Definition}
\theoremstyle{remark}
\newtheorem{rem}[thm]{Remark}
\numberwithin{figure}{section}
\numberwithin{table}{section}
\numberwithin{equation}{section}
\newcommand{\To}{\longrightarrow}

\mathchardef\mhyphen="2D

           \newcommand{\eq}[1][r]
   {\ar@<-3pt>@{-}[#1]
    \ar@<-1pt>@{}[#1]|<{}="gauche"
    \ar@<+0pt>@{}[#1]|-{}="milieu"
    \ar@<+1pt>@{}[#1]|>{}="droite"
    \ar@/^2pt/@{-}"gauche";"milieu"
    \ar@/_2pt/@{-}"milieu";"droite"}

\newcommand{\R}{\mathbb R}  
\newcommand{\Z}{\mathbb Z}  

\newcommand{\N}{\mathbb N}  
\newcommand{\C}{\mathbb C}  

\newcommand{\rightsetse}[1]{%
\hidewidth\rotatebox[origin=c]{-45}{$\xrightarrow{\kern2em}$}
     \rlap{\raisebox{1ex}
     {$\kern-.8em\scriptstyle #1$}}\hidewidth}
\newcommand{\rightsetsw}[1]{%
\hidewidth\rotatebox[origin=c]{45}{$\xleftarrow{\kern2em}$}
     \rlap{\raisebox{.1ex}
     {$\kern-.8em\scriptstyle #1$}}\hidewidth}
\newcommand{\leftsetsw}[1]{%
\hidewidth
     \llap{\raisebox{1ex}
     {$\scriptstyle #1$\kern-.8em}}
    \rotatebox[origin=c]{45}{$\xleftarrow{\kern2em}$}\hidewidth}
\newcommand{\rightsetnw}[1]{%
\hidewidth\rotatebox[origin=c]{135}{$\xrightarrow{\kern2em}$}
     \rlap{\raisebox{1ex}
     {$\kern-.8em\scriptstyle #1$}}\hidewidth}
\newcommand{\rightsetd}[1]{%
\hidewidth\rotatebox[origin=c]{-90}{$\xrightarrow{\kern2em}$}
     \rlap{{$\scriptstyle #1$}}\hidewidth}
\usepackage[breaklinks]{hyperref}
\subjclass[2010 MSC]{Primary 22E45;
Secondary
  22E46, 
  32M15, 
  33C45,  
  33C80, 
  43A85
 }
\title
{Inversion of Rankin--Cohen operators via Holographic Transform}
\author{Toshiyuki Kobayashi, Michael Pevzner}

\begin{document}
\begin{abstract}
The analysis of branching problems for restriction
of representations brings the concept of  \emph{symmetry breaking transform} and \emph{holographic transform}.
Symmetry breaking operators decrease the number of variables
 in geometric models,
 whereas holographic operators increase it.
Various expansions in classical analysis
 can be interpreted as particular occurrences of these transforms.
From this perspective we investigate two remarkable families of 
 differential operators:
 the Rankin--Cohen operators and the holomorphic Juhl conformally covariant operators.
Then we establish for 
the corresponding symmetry breaking transforms 
 the Parseval--Plancherel type theorems
 and find explicit inversion formul{\ae} with integral expression
 of holographic operators.

The proof uses the F-method 
 which provides a duality
 between symmetry breaking operators
 in the holomorphic model
 and holographic operators
 in the $L^2$-model, 
 leading us to  deep links
 between special orthogonal polynomials and branching laws for infinite-dimensional representations of real reductive
Lie groups.\vskip7pt
\noindent Keywords and phrases: \emph{Symmetry breaking, holographic transform, Rankin--Cohen operators, Juhl operators, orthogonal polynomials, branching rules, F-method}.
\end{abstract}

\maketitle
\tableofcontents
\section{Introduction}\label{sec:Intro}

Let $\pi$ be an irreducible representation of a group $G$
 on a vector space $V$, 
 and $G'$ a subgroup.  
The $G$-module $(\pi,V)$ may be seen as a $G'$-module by restriction, 
 for which we write $\pi|_{G'}$.  
For an irreducible representation $(\rho,W)$
 of the subgroup $G'$, 
 a {\it{symmetry breaking operator}} is
 a (continuous) linear map $V \to W$
 which intertwines $\pi|_{G'}$ and $\rho$.  
In recent years {\it{individual}} symmetry breaking operators
 have been studied intensively
 in different settings
 ranging from automorphic form theory
 to conformal geometry,
 see \cite{Clerc, Cohen, FG, Juhl, KKPAdS, KP16b, KS, Z} 
 and references therein.

In this article,
 we investigate a {\bf{collection}} of symmetry breaking operators,
 $$R_\ell\, \colon\, V\To W_\ell,\,\ell\in\Lambda,
 $$
 referred to as a \emph{symmetry breaking transform}, 
 for a family of irreducible representations $\rho_\ell$ of the subgroup $G'$
 on vector spaces $W_\ell$ with parameter $\ell\in\Lambda$.

Various expansions in classical analysis
 can be interpreted through this paradigm:

\begin{example}
[$GL_n \downarrow GL_{n-1}$]
Arranging homogeneous polynomials of $x=(x_1,\cdots,x_n)$
 in descending order with respect to the power of $x_n$ is an example of symmetry breaking transform for $(G,G')=(GL_n, GL_{n-1})$. In fact, taking the $\ell$-th component in the expansion
 $$
 f(x)=\sum_{\ell=0}^k f_\ell(x') x^{k-\ell}_n,\quad\mathrm{for}\,\, x'=(x_1,\cdots,x_{n-1})
 $$
 defines a $G'$-homomorphism from
 $V:=\mathrm{Pol}^k[x]$ to $W_\ell:=\mathrm{Pol}^\ell[x']$ on which $G$ and $G'$, respectively, 
act irreducibly.
\end{example}

Traditional representation-theoretic viewpoint tells
 that the Fourier series expansion or Fourier transform
 is the irreducible decomposition
 of the regular representation
 of the {\it{abelian}} group $G'=S^1$ or ${\mathbb{R}}$, 
 whereas we make use of a \emph{hidden symmetry} of the {\it{noncommutative}} group $G=SL(2,{\mathbb{R}})$ in the sense that $G$ contains $G'$ as a subgroup and that $G$
 acts on the space of functions on $S^1$ or $\R$. The latter viewpoint brings us a new
 interpretation of the (classical) Fourier series or Fourier transform in the framework of
 ``symmetry breaking" as follows.

\begin{example}
\label{ex:FourierSL}
A spherical principal series representation $\pi_{\lambda}$
 of $G=SL(2,{\mathbb{R}})$ is realized
 on the vector space of homogeneous functions
$$
V_\lambda:=\left\{ f \in C^\infty (\R^2\setminus\{(0,0)\}) : 
                   f(ax, ay)=|a|^{\lambda}f(x,y) \text { for all } a \in {\mathbb{R}}^{\times} \right\}.
$$
This representation is irreducible for all $\lambda\in\C\setminus\Z$.
\par\noindent
$\bullet$\enspace
{\rm{(Fourier series)}}\enspace
The representation $\pi_{\lambda}$ of $S L(2,{\mathbb{R}})$ can be realized
 in $C^{\infty}(S^1)$
 via the identification $V_{\lambda} \overset \sim \to C^{\infty}(S^1)$, 
 $f(x,y) \mapsto h(\theta):=f(\cos \theta, \sin \theta)$, 
 because any homogeneous function is determined by its restriction to the unit circle $S^1$.  
Since $S^1$ is preserved
 by the subgroup
$G':=SO(2)$, 
 the collection of the Fourier coefficients
\begin{eqnarray*}
        V_{\lambda}&\stackrel
{\sim}{\longrightarrow}& C^\infty(S^1)\To\operatorname{Map}(\Z,\C),
\\ 
       f \mapsto h&\mapsto& \widehat h(\ell):=\frac1{2\pi}\int_0^{2\pi} h(\theta) e^{-i\ell\theta} d\theta, \quad\ell\in\Z
\end{eqnarray*}
gives a symmetry breaking transform from the infinite-dimensional representation
 $(\pi_{\lambda}, V_{\lambda})$
 of $G=S L(2,{\mathbb{R}})$
 to the collection of one-dimensional representations $\chi_\ell$
 of the abelian subgroup $G'=SO(2) \simeq S^1$
 indexed by $\ell \in {\mathbb{Z}}$.

\par\noindent
$\bullet$\enspace
{\rm{(Fourier transform)}}\enspace
Similarly, any function $f(x,y)\in V_\lambda$ is determined by its restriction
to the real line $y=1$, 
which is preserved by the unipotent subgroup 
$
   G'':=\left\{\begin{pmatrix}1&\xi\\0&1\end{pmatrix} : \xi\in\R\right\}\,(\simeq\R)
$.  
Thus the Fourier transform 
$$
  L^1(\R)\To C(\R), 
\quad 
  F \mapsto (\mathcal F F)(\xi):=\int_\R F(x) e^{-i x \xi}d x,
$$
induces another symmetry breaking transform
 for the pair $(G, G'')=(SL(2,\R), \R)$. 
\end{example}

\vskip10pt
\begin{example}
[spherical harmonics]
Expansion of functions on $S^n$
 by eigenfunctions of the Laplacian $\Delta_{S^n}$
 corresponds to a symmetry breaking transform 
{}from a spherical principal series representation $\pi$
 of $G=SO(n+1,1)$
 to a collection
 of irreducible finite-dimensional representations
 of the compact subgroup $G'=O(n+1)$.  
\end{example}

Reversing the arrows in the definition of a symmetry breaking operator 
$R_\ell \colon V \to W_{\ell}$,
 we consider a $G'$-homomorphism $\Psi_\ell \colon W_{\ell} \to V$, 
 going from smaller to larger representation space, and thus
 referred to as a \emph{holographic operator}.  
As in the case of symmetry breaking,
 the collection of holographic operators $\{\Psi_\ell\}$
 is said to be a \emph{holographic transform.}
$$
 G \curvearrowright 
       \xymatrix
       {
          V\ar@<0.5ex>[r]^{R_\ell}&
            W_\ell
              \ar@<0.5ex>[l]^{\Psi_\ell}
       }
  \curvearrowleft G'.  
$$
To illustrate a holographic transform
 by an example 
 with both $V$ and $W_{\ell}$ being infinite-dimensional, 
 we recall that the classical Poisson integral (see \emph{e.g.} \cite[Sec. 0]{KKMOOT})
\begin{eqnarray*}
\mathcal P_\nu\colon C_c(\R)&\To& C^\infty(\Pi)
\\ 
h(t) &\mapsto&
(\mathcal P_\nu h)(x, y) = \int_{-\infty}^{\infty} \frac{y^\nu}{((x-t)^2+y^2)^{\nu}} h(t) d t 
\end{eqnarray*}
constructs eigenfunctions of the Laplace--Beltrami differential operator $\Delta= y^2\left(\frac{\partial^2}{\partial x^2}+\frac{\partial^2}{\partial y^2}\right)$ for the eigenvalue $\nu(\nu-2)$ on the upper-half plane $\Pi$
 endowed with Poincar\'e metric. 
The group $S L(2,{\mathbb{R}})$ acts isometrically on $\Pi$
 and conformally on its boundary.  
Traditionally,
 the Poisson integral was treated
 in the context of representations
 of $S L(2,{\mathbb{R}})$, 
 however,
 we highlight the fact
 that the totality of functions on $\Pi$ admits a larger symmetry
 because the group $S L(2,{\mathbb{C}})$
 acts on the conformal compactification of $\Pi$.  
Thus the Poisson integral can be interpreted as a particular occurrence
 of a holographic operator
 for the pair $(G,G')=(SL(2,\C), SL(2,\R))$
 as below.

\begin{example}[Poisson integral]
\label{ex:Poisson}
A generic symmetry breaking operator $A_{\lambda,\nu}$ from the spherical principal series representation $\pi_\lambda$
 of $G=S L(2,{\mathbb{C}})$ on $C^\infty(S^2)$
 to the one $\varpi_\nu$ of the subgroup $G'=S L(2,{\mathbb{R}})$ on $C^\infty(S^1)$ takes the following form $($see \cite[(7.2)]{KS}$)$:
\begin{eqnarray*}
A_{\lambda,\nu}:
C_c^\infty(\R^2)&\To& C^\infty(\R), 
\\
f(x,y) &\mapsto& (A_{\lambda,\nu} f)(x)=\int_{\R^2} f(t,y) K_{\lambda,\nu}(x-t,y)dtdy
\end{eqnarray*}
in the flat coordinates
 where $K_{\lambda,\nu}$ is a distributional kernel
 given by 
\begin{equation}
\label{eqn:Klmdnu}
K_{\lambda,\nu}(x,y)=(x^2+y^2)^{-\nu}\vert y\vert^{\lambda+\nu-2}.  
\end{equation}
Then the dual map of $A_{\lambda,\nu}$ yields a holographic operator
 $\Psi_{\lambda,\nu}$ with the formula
$$
g(t) \mapsto \left(\Psi_{\lambda,\nu}g\right)(x,y):=\int_{\R}g(t) K_{\lambda,\nu}(x-t,y)dt.
$$
Thus 
 the (classical) Poisson integral $\mathcal P_{\nu}$ can be viewed as the restriction
 of the holographic operator $\Psi_{\lambda,\nu}$
 with $\lambda=2$, 
 namely,  {defn}
$
\mathcal P_\nu=\mathrm{Rest}_\Pi\circ\Psi_{2,\nu}.  
$
\end{example}

With these interpretations of classical examples in mind, 
 we raise the following two general problems
 for a symmetry breaking transform
$\displaystyle R(v)=\{R_\ell(v)\}_{\ell\in \Lambda},$
where $R_\ell
 \colon V\To W_{\ell}$ ($\ell\in\Lambda$), 
are symmetry breaking operators:
\begin{problem}
Can we recover an element $v$ of $V$ from its symmetry breaking transform
$R(v)=\{R_\ell(v)\}_{\ell\in\Lambda}$?
\end{problem}

Problem A includes the following subproblems:
\begin{itemize}
   \item[{\bf A.0.}] Tell \emph{a priori} if $\Lambda$ is sufficiently large for $R$ to be injective.
      \item[{\bf A.1.}] Construct a \lq\lq{holographic transform}\rq\rq.
   \item[{\bf A.2.}] Find an explicit inversion of the symmetry breaking transform $R$.
\end{itemize}

When $V$ is a Hilbert space on which $G$ acts unitarily, 
 we also ask for a Parseval--Plancherel type theorem
 for the symmetry breaking transform:

\begin{problem}\label{prob:B}
  Find a closed formula for the norm of an element $v$ in $V$ in terms of its symmetry breaking transform  $\{R_\ell(v)\}_{\ell\in\Lambda}$.
\end{problem}

In this article, we investigate Problems A and B in the following two cases:
\begin{itemize}
  \item Rankin--Cohen transform (Section \ref{sec:RCT});
  \item Holomorphic Juhl transform (Section \ref{sec:HJT}).
\end{itemize}
In both cases,
 the transform is a collection of holomorphic differential operators
 between complex manifolds: 
 the first case is associated with the family of
 the Rankin--Cohen operators that appeared in the theory of holomorphic modular forms \cite{Cohen}, 
 whereas the second case originated from Juhl's conformally covariant operators \cite{Juhl}.

These transforms can be analyzed in the framework of infinite dimensional representations of Lie groups, namely, the decomposition of the tensor product of two holomorphic discrete series
representations of $SL(2,\R)$ in the first case, and the branching laws of holomorphic discrete series
representations of the conformal Lie group $G=SO_o(2,n)$
 when restricted to a subgroup $G'=SO_o(2,n-1)$, in the second case. 


The main goal here is to
give a solution
 to Problems A and \ref{prob:B} for the above two transforms.  
We provide two types of integral expressions
 as a solution to Problem A1, 
 see Theorems \ref{thm:psi} and \ref{thm:CIRM18}.
The main results are summarized as below.

{\small\begin{figure}[!h]
   \begin{center}
    \begin{tabular}{c|c|c}
       $G\supset G'$ &$SL_2\times SL_2\supset SL_2$  &$SO_o(2,n)\supset SO_o(2,n-1)$ \\
      &&\\
               \hline
         Problem A1 & &\\
         construction of & Theorem \ref{thm:psi}& Theorem \ref{thm:CIRM18}\\
         holographic transform &&\\
         \hline
         Problem A2 &&\\
         inversion of symmetry & Theorem \ref{thm:RCinv} &Theorem \ref{thm:170887}\\
         breaking transform&&\\
         \hline
         Problem B &&\\
         $L^2$-theory for &Theorem \ref{thm:PPTT4RCB}&Theorem \ref{thm:170887}\\
            &&
        \end{tabular}
   \end{center}
  \end{figure}%
}


The key idea of our approach is to 
introduce \lq\lq{special orthogonal polynomials}\rq\rq\ $\{P_\ell\}$ associated to
symmetry breaking operators. This can be done
  via the F-method, which we developed in \cite{KP16a,KP16b},
 that analyzes the representations through the Fourier transform
 of their geometric realizations. 
 In this article, we show for
 the Rankin--Cohen bidifferential operators
 $\{R_\ell\}$ 
that the polynomials $\{P_\ell\}$ are the Jacobi polynomials and that 
 the holographic operators are given by the Jacobi transforms
 along the transversal direction to a codimension-one foliation of the symmetric cone
 (Section \ref{sec:RCT});
for the holomorphic Juhl operators $\{R_\ell\}$, the holographic operators are associated to the Gegenbauer polynomials $\{P_\ell\}$
(Section \ref{sec:HJT}).  
Thus Problems A and \ref{prob:B}
 for symmetry breaking transforms
 can be studied
 as questions on special orthogonal polynomials via the F-method.

The table below shows some new links
 which the F-method provides
 between representations
 and special functions
 in this setting.

{\small
{\begin{figure}[!h]
   \begin{center}
    \begin{tabular}{c|c}
    Symmetry breaking operators $\{R_\ell\}$   
& Special orthogonal polynomials $\{P_\ell\}$
\\
&\\
\hline
&\\
   $G'$-intertwining property
&  hypergeometric differential equations
\\
&\\
\hline
&\\
  operator norm of  $R_\ell$     
&
 $L^2$-norm of $P_\ell$
\\
&\\
\hline
&\\
branching law $\pi|_{G'}$
&
$L^2$-completeness of $\{P_\ell\}$
\\
&\\
\hline
&\\
holographic transform
($L^2$-model)
&
integral transform
associated to $\{P_\ell\}$
\\

\end{tabular}
   \end{center}
  \end{figure}%
}
}
\vskip7pt

Analogously to the classical Poisson transform
 (Example \ref{ex:Poisson}), 
 the holographic transform
 provides an integral expression 
 of eigenfunctions
 of certain holomorphic differential operator.
We illustrate this idea with the example of the Rankin--Cohen operators, see Theorem \ref{thm:180297} which is proved as a byproduct of the main results.

In Section \ref{sec:conclusion} we discuss the background of Problems A and \ref{prob:B} from a viewpoint of the representation theory of real reductive Lie groups.


\vskip 10pt
Notation: $\N=\{0,1,2,\cdots\}$, $i=\sqrt{-1}$ (imaginary unit), $(x)_k=x(x+1)(x+2)\cdots(x+k-1)$ for $k\in\N$ (Pochhammer symbol), and $[x]$ is the largest integer that does not exceed $x\in\R$.

\vskip1pc

\emph{Acknowledgments}. The first author was partially supported by the JSPS under the 
Grant-in-Aid for Scientific Research (A) (JP18H03669).
Both authors were partially supported by the CNRS grant PICS-7270 and 
they
are grateful to Institut des Hautes \'Etudes Scientifiques 
(Bures-sur-Yvette, France), Institut Henri Poincar\'e (Paris, 
France)
and Centre International de Rencontres Math\'ematiques (Luminy, 
France) where an important part of this work was done.

\vskip10pt

\section{Rankin--Cohen transform and its holographic transform}\label{sec:RCT}

The Rankin--Cohen bidifferential operators map functions of two variables to those of one variable,
respecting twisted actions of $SL(2,\R)$. In this section, we solve Problems A and B stated in Section \ref{sec:Intro} for the \emph{Rankin--Cohen transform} (Definition
\ref{def:RCT}), a collection of such operators.

\subsection{Rankin--Cohen bidifferential operators}
We begin with a quick review of the Rankin--Cohen bidifferential operators.

\subsubsection{Holomorphic discrete series representations of ${SL(2,\R)}{\;}\widetilde{}$}\label{subsec:pilmd}
Let $\Pi=\{z=x+iy\in\C:\; x\in\R,\;y>0\}$ be the upper half-plane, and $\mathcal O(\Pi)$ the space
 of holomorphic functions on $\Pi$. 
For $\lambda\in\Z$
 we define a representation $\pi_\lambda$ of ${SL(2,\R)}$  on $\mathcal O(\Pi)$ by
$$
\pi_\lambda(g)f(z)=(cz+d)^{-\lambda}f\left(\frac{az+b}{cz+d}\right)
\qquad\mathrm{for}\quad
g^{-1}=\left(
\begin{array}{cc}
  a   & b  \\
  c   & d   
\end{array}
\right).
$$
Viewed as a representation of the universal covering group 
 ${SL(2,\R)}{\;}\widetilde{}$, 
 the representation $\pi_{\lambda}$ is well-defined
 for all $\lambda \in {\mathbb{C}}$.  
There is a canonical perfect pairing between
$(\pi_\lambda,\mathcal O(\Pi))$ and the Verma module
$$
M_{-\lambda}:= U(\mathfrak g_\C)\otimes_{U(\mathfrak b)}\C_{-\lambda},
$$
where $U(\mathfrak g_\C)$ denotes the universal enveloping algebra of 
$\mathfrak g_\C=\mathfrak{sl}(2,\C)$ and $\mathfrak b$ is a Borel subalgebra
 containing
$\mathfrak k_\C=\mathfrak{so}(2,\C)$.
Therefore,
 $(\pi_\lambda,\mathcal O(\Pi))$ is irreducible 
 if and only if $\lambda \in {\mathbb{C}} \setminus (-{\mathbb{N}})$ because
 the $\mathfrak g$-module $M_\nu$ is reducible if and only if $\nu\in\N$.
   
Let $p \colon {SL(2,{\mathbb{R}}}){\;}\widetilde{} \to SL(2,{\mathbb{R}})$
 be the covering homomorphism, 
 and set ${SO(2)}{\;}\widetilde{}=p^{-1}(SO(2))$.  
For every $\lambda\in\C$, we can form a homogeneous holomorphic line bundle $\mathcal L_\lambda$ over $\Pi\simeq
{SL(2,\R)}{\;}\widetilde{}/{SO(2)}{\;}\widetilde{}$ associated  to a character
$\C_\lambda$ of ${SO(2)}{\;}\widetilde{}$, and the multiplier representation $(\pi_\lambda, \mathcal O(\Pi))$ is equivalent to the natural action
 of ${SL(2,\R)}{\;}\widetilde{}\;$
 on the space $\mathcal O(\Pi, \mathcal L_\lambda)$ of holomorphic sections of  $\mathcal L_\lambda$.

\subsubsection{Holomorphic model ${\mathcal{H}}^2(\Pi)_{\lambda}$}
For $\lambda>1$ the weighted Bergman space $\mathcal H^2(\Pi)_\lambda:=(\mathcal O\cap L^2)(\Pi, y^{\lambda-2}d x d y)$ is nonzero, and the Hilbert space $\mathcal H^2(\Pi)_\lambda$ admits a reproducing kernel
$K_\lambda(z,w)=\frac{\lambda-1}{4\pi}\left(\frac{z-\bar w}{2i}\right)^{-\lambda}$, see \cite[Prop. XIII.1.2]{FK}.
The representation
$(\pi_\lambda,\mathcal O(\Pi))$ yields an irreducible unitary representation of ${SL(2,\R)}{\;}\widetilde{}{\;}$ on
$\mathcal H^2(\Pi)_\lambda$, 
which descends to $SL(2,\R)$ when $\lambda\in\Z$. The set of equivalence classes of
irreducible unitary representations (\emph{unitary dual}) of $SL(2,\R)$ contains a family of those with continuous parameter (\emph{e.g.} principal series representations, complementary series representations), whereas $\pi_\lambda$ ($\lambda= 2,3,\cdots$)
form a countable family of irreducible unitary representations realized in the kernel of the Cauchy--Riemann operator. Thus 
$\pi_\lambda$ ($\lambda= 2,3,\cdots$)
is referred to as a \emph{holomorphic discrete series representations} of $SL(2,\R)$,
and $\pi_\lambda$ ($\lambda>1$)
  as a relative holomorphic discrete series representation of the covering group
 ${SL(2,\R)}{\;}\widetilde{}$.
We call the realization on $\mathcal H^2(\Pi)_{\lambda}$ \emph{holomorphic model} of the representation $\pi_\lambda$.  
Similarly, the direct product group ${SL(2,\R)}{\;}\widetilde{}\times
{SL(2,\R)}{\;}\widetilde{}\;$ acts on $\mathcal H^2(\Pi\times\Pi)_{(\lambda',\lambda'')}\simeq
\mathcal H^2(\Pi)_{\lambda'}\widehat\otimes \mathcal H^2(\Pi)_{\lambda''}$ as an irreducible unitary representation if $\lambda',\lambda''>1$, where $\widehat\otimes$ stands for the completion of the algebraic tensor product.

We shall deal with another realization ($L^2$-model) of the same representation $\pi_\lambda$
in Section \ref{subsec:phi}. 
$$
$$
\subsubsection{Rankin--Cohen bidifferential operators}
~~~\par
Consider $\lambda',\lambda'',\lambda'''\in\C$ such that  $\ell:=\frac12(\lambda'''-\lambda'-\lambda'')\in\N$ and
define a differential operator 
$\mathcal{R}_{\lambda',\lambda''}^{\lambda'''}:\mathcal O(\Pi\times\Pi)\To \mathcal O(\Pi\times\Pi)$ by
\begin{equation}\label{rcb}
\mathcal{R}_{\lambda',\lambda''}^{\lambda'''}(f)(\zeta_1,\zeta_2):=
\sum_{j=0}^\ell (-1)^{j}
\frac{
\left(\lambda'+\ell-j\right)_j \left(\lambda''+j \right)_{\ell-j}}
{j!(\ell-j)!}           
\frac{\partial^\ell f}{\partial \zeta_1^{\ell-j}\partial \zeta_2^{j}}(\zeta_1,\zeta_2).
\end{equation} 

 The {\it{Rankin--Cohen bidifferential operator}} is a linear map
$$
\mathcal{RC}_{\lambda',\lambda''}^{\lambda'''}: \mathcal O(\Pi\times
\Pi)\To \mathcal O(\Pi),
$$
defined by $\mathcal{RC}_{\lambda',\lambda''}^{\lambda'''}:=
\mathrm{Rest}\circ \mathcal{R}_{\lambda',\lambda''}^{\lambda'''}$,
 where
 $\mathrm{Rest}$ stands for the restriction map $f(\zeta_1,\zeta_2)\mapsto f(\zeta,\zeta)$ to the diagonal.

The Rankin--Cohen bidifferential operator $\mathcal{RC}_{\lambda',\lambda''}^{\lambda'''}$
 is a symmetry breaking operator from the tensor product representation $\pi_{\lambda'}\widehat\otimes \pi_{\lambda''}$ to 
$\pi_{\lambda'''}$ with respect to the diagonal embedding
${SL(2,\R)}{\;}\widetilde{}\hookrightarrow {SL(2,\R)}{\;}\widetilde{}\times
{SL(2,\R)}{\;}\widetilde{},$ and such a symmetry breaking operator
is unique up to scalar multiplication for generic parameters (see \cite[Cor. 9.3]{KP16b} for the precise condition).
Moreover,
 $\mathcal{RC}_{\lambda',\lambda''}^{\lambda'''}$ induces a continuous map from the weighted Bergman space
  ${\mathcal{H}}^2(\Pi \times \Pi)_{(\lambda',\lambda'')}$
 to ${\mathcal{H}}^2(\Pi)_{\lambda'''}$
 if $\lambda', \lambda'' >1$
 (\cite[Thm.~5.13]{KP16a}, see also Proposition \ref{prop:RCopnorm} below
 for an explicit formula of its operator norm).  


\subsection{Notations and two constants $c_\ell(\lambda',\lambda'')$ and $r_\ell(\lambda',\lambda'')$}\label{sec:constants}
The parameter set in Section \ref{sec:RCT} is $(\lambda',\lambda'',\lambda''')\in\C^3$ with
$\lambda'''-\lambda'-\lambda''\in2\N$. Throughout this section, we use the following notation:
\begin{equation}\label{eqn:abl}
\alpha=\lambda'-1,\quad \beta=\lambda''-1,\quad 2\ell=\lambda'''-\lambda'-\lambda''.
\end{equation}
The main results involve the following two constants
\begin{eqnarray}\label{eqn:v-ell}
c\equiv c_\ell(\lambda',\lambda'') &:=& \nonumber
\frac1{2^{\alpha+\beta+1}}\int_{-1}^1\vert P_\ell^{\alpha,\beta}(v)\vert^2 (1-v)^\alpha (1+v)^\beta dv\\
&=& 
\frac{\Gamma(\lambda'+\ell) \Gamma(\lambda''+\ell)}{(\lambda'+\lambda''+2\ell-1) \Gamma(\lambda'+\lambda''+\ell-1)\ell!},
\\
\label{eqn:b-ell}
r\equiv r_\ell(\lambda',\lambda'') &:=& \nonumber
\frac{b(\lambda''')}{b(\lambda')b(\lambda'')}\\
&=& 
\frac{\Gamma(\lambda'+\lambda''+2\ell-1)}{2^{2\ell+2}\pi \Gamma(\lambda'-1)\Gamma(\lambda''-1)},
\end{eqnarray}
where $P_\ell^{\alpha,\beta}(v)$ is the Jacobi polynomial (see \eqref{eqn:Pnorm} in Appendix), and
$b(\lambda)=2^{2-\lambda}\pi\Gamma(\lambda-1)$ is a Plancherel density (see Fact \ref{fact:Laplace} below).
We note that $c_{\ell}(\lambda',\lambda'') \ne 0$
 if $\operatorname{Re}\lambda',\operatorname{Re}\lambda'' >0$
 and $\ell \in {\mathbb{N}}$.  

\subsection{Integral formula for holographic operators}
~~~

In this section,
 we introduce integral transforms $\Psi_{\lambda',\lambda''}^{\lambda'''}$
 ({\it{holographic operator}})
 that realize  irreducible summands in the tensor
product representations $\pi_{\lambda'}\widehat\otimes \pi_{\lambda''}$. 

\subsubsection{Construction of holographic operators for the tensor product}
~~~\par
\begin{defn}[holographic operators]\label{def:HOtensor}
For $\lambda',\lambda'',\lambda'''\in\C$ we set  $\ell:=\frac12(\lambda'''-\lambda'-\lambda'')$. Assume that
\begin{equation}\label{eqn:HOconverge}
\mathrm{Re}(\lambda'+\ell)>0,\quad\mathrm{Re}(\lambda''+\ell)>0,\quad \mathrm{and}\quad\ell\in\N.
\end{equation}
 For a holomorphic function $g$ on the upper half plane $\Pi$, we define a holomorphic function on $\Pi\times\Pi$ by
 the line integral:
\begin{align}\label{eqn:psi}
&\left(\Psi_{\lambda',\lambda''}^{\lambda'''} g\right)(\zeta_1,\zeta_2):=
\frac{(\zeta_1-\zeta_2)^\ell}{2^{\lambda'+\lambda''+2\ell-1}\ell!}\times\\
&\int_{-1}^1
g\left(\frac{(\zeta_2-\zeta_1)v+(\zeta_1+\zeta_2)}{2}\right)
(1-v)^{\lambda'+\ell-1}(1+v)^{\lambda''+\ell-1}
dv.\nonumber
\end{align}

\end{defn}
We note that the set $\left\{\frac{(\zeta_2-\zeta_1)v+(\zeta_1+\zeta_2)}{2} : -1\leq v\leq 1\right\}$ is
the line segment connecting the two points $\zeta_1$ and $\zeta_2$ in $\Pi$.

\subsubsection{Basic properties of $\Psi_{\lambda',\lambda''}^{\lambda'''}$}

The integral transform $\Psi_{\lambda',\lambda''}^{\lambda'''}$
in \eqref{eqn:psi} provides
 a holographic operator in the following sense:

\begin{thm}[holographic operator in the upper half plane]
\label{thm:psi}
Suppose ${\lambda',\lambda''},{\lambda'''}\in\C$ satisfy \eqref{eqn:HOconverge}.
\begin{itemize}
\item[(1)] The map
$\Psi_{\lambda',\lambda''}^{\lambda'''}:\mathcal O(\Pi)\To 
\mathcal O(\Pi\times\Pi)$
intertwines the action of ${SL(2,\R)}{\;}\widetilde{}$\,\, from $\pi_{\lambda'''}$ to the tensor product representation
$\pi_{\lambda'}\widehat\otimes\pi_{\lambda''}$.

\item[(2)] Moreover, if 
both $\lambda'$ and $\lambda''$ are real
 and greater than 1,
 then the linear map $\Psi_{\lambda',\lambda''}^{\lambda'''}$ induces an isometric embedding
 (up to rescaling)
 of the weighted Bergman space:
$$
\mathcal H^2(\Pi)_{\lambda'''}\To\mathcal H^2(\Pi\times\Pi)_{(\lambda',\lambda'')}.
$$
\end{itemize}
\end{thm}  
The image of the holographic operator $\Psi_{\lambda',\lambda''}^{\lambda'''}$ is characterized by a
differential equation of second order on $\Pi\times\Pi$
associated to the Casimir element under the diagonal action, 
see Theorem \ref{thm:180297}.
For ${\lambda',\lambda''}>1$, 
the operator
 $\Psi_{\lambda',\lambda''}^{\lambda'''}$ is a scalar multiple
 of the adjoint $\left(\mathcal{RC}_{\lambda',\lambda''}^{\lambda'''}\right)^*$
 of the Rankin--Cohen bidifferential operator $\mathcal{RC}_{\lambda',\lambda''}^{\lambda'''}$
 (see Proposition \ref{prop:PsiRC}), and its operator norm
 is given in Theorem \ref{thm:PPTT4RCB} (2).  

Theorem \ref{thm:psi} will be proved in Section \ref{subsec:pfthmpsi}.

\begin{rem}\label{rem:2integral}
In Section \ref{sec:HJT}, we introduce {\it{relative}} reproducing kernels to construct irreducible summands in the holomorphic model. The integral formula given there (see
Theorem \ref{thm:CIRM18}) is different from the one introduced in Definition \ref{def:HOtensor}.
The advantage of the definition \eqref{eqn:psi} is
 that the holographic operator $\Psi_{\lambda',\lambda''}^{\lambda'''}$
 is defined also for the nonunitary case, see Theorem \ref{thm:psi} (1).  
\end{rem}


\subsection{The Rankin--Cohen transform and its inversion}

In this section we introduce the Rankin--Cohen transform $\mathcal{RC}_{\lambda',\lambda''}$ as the collection of individual operators
$\mathcal{RC}_{\lambda',\lambda''}^{\lambda'''}$ for fixed $\lambda'$ and $\lambda''$. Its inversion formula
is proved in Theorem \ref{thm:RCinv} by using the holographic operators,  giving a solution to Problem A
in Section 1.

\begin{defn}[Rankin--Cohen transform]\label{def:RCT}
For $\lambda',\lambda''\in\C$, the \emph{Rankin--Cohen transform} $\mathcal{RC}_{\lambda',\lambda''}$ is
a linear map
\begin{equation}\label{eqn:RCtrans}
\mathcal{RC}_{\lambda',\lambda''}\colon \mathcal O(\Pi\times\Pi)\longrightarrow\operatorname{Map}(\mathbb N, \mathcal O(\Pi)), 
\quad
f \mapsto \left(\ell\mapsto\mathcal{RC}_{\lambda',\lambda''}(f)_{\ell}\right)
\end{equation}
defined by
$
\left(\mathcal{RC}_{\lambda',\lambda''}(f)\right)_\ell:=\mathcal{RC}_{\lambda',\lambda''}^{\lambda'+\lambda''+2\ell} f\qquad\mathrm{for}\;\ell\in\N.
$
\end{defn}

The Rankin--Cohen transform $\mathcal{RC}_{\lambda',\lambda''}$ intertwines $(\pi_{\lambda'}\widehat\otimes\pi_{\lambda''},\mathcal O(\Pi\times\Pi))$ with the formal direct
sum
$\widehat\bigoplus_{\ell\in\N}(\pi_{\lambda'+\lambda''+2\ell},\mathcal O(\Pi))$, and
can be inverted by using the integral operators $\Psi_{\lambda',\lambda''}^{\lambda'''}$ as follows.

\begin{thm}[inversion of the Rankin--Cohen transform]\label{thm:RCinv}
Suppose $\lambda',\lambda''>1.$ Then for any $f\in \mathcal H^2(\Pi)_{\lambda'}\widehat\otimes \mathcal H^2(\Pi)_{\lambda''}$ one has
$$
f=\sum_{\ell=0}^\infty \frac1{c_\ell(\lambda',\lambda'')}\,
\Psi_{\lambda',\lambda''}^{\lambda'+\lambda''+2\ell}\left(\mathcal{RC}_{\lambda',\lambda''}(f)\right)_\ell.
$$
\end{thm}
Theorem \ref{thm:RCinv} will be proved in Section \ref{subsec:pfthmRCinv}.


\subsection{Parseval--Plancherel type theorem for the Rankin--Cohen transform and its holographic transform}

In this section we develop an $L^2$-theory for the Rankin--Cohen transform (Definition \ref{def:RCT}) and for the
holographic transform (Theorem \ref{thm:PPTT4RCB} (2)), thus providing an answer to Problem B for these two transforms.


\subsubsection{Weighted Hilbert sums}

In order to formulate the Parseval--Plancherel type theorem, we fix some notations for the Hilbert direct sum.
\begin{defn}
[weighted Hilbert sum]
\label{def:compsum}
Let$\{V_\ell\}_{\ell\in\N}$ be a family of Hilbert spaces and $\{a_\ell\}_{\ell\in\N}$ a sequence  of positive numbers. 
The Hilbert sum $\displaystyle{\sum\limits_{\ell\in\N}}^\oplus V_\ell$  associated to the {\it{weights}} $\{a_\ell\}_{\ell\in\N}$ is the Hilbert
completion of the
algebraic direct sum $\displaystyle\bigoplus\limits_{\ell\in\N}V_\ell$ equipped with the pre-Hilbert structure given by
$$
(v,v'):=\sum_{\ell=0}^\infty a_\ell(v_\ell,v'_\ell)_{V_\ell}
\quad
\text{for $v=(v_\ell)_{\ell\in\N}$ and $v'=(v'_\ell)_{\ell\in\N}$}.
$$
\end{defn}


\subsubsection{Parseval--Plancherel type theorem}\label{subsec:PP-RC}

For
$\lambda',\lambda''>1$ the bidifferential 
operators $\mathcal{RC}_{\lambda',\lambda''}^{\lambda'''}$ 
extend to a continuous map between Hilbert spaces.
Now, we formulate a Parseval--Plancherel type theorem for the Rankin--Cohen transform as well as the ``holographic transform'', hence answer Problem B for these transforms.

\begin{thm}[Parseval--Plancherel theorem]\label{thm:PPTT4RCB}

Suppose $\lambda',\lambda''>1$. 
\begin{itemize}
\item[(1)] 
The Rankin--Cohen transform $\mathcal{RC}_{\lambda',\lambda''}$ (Definition \ref{def:RCT})
induces an ${SL(2,\R)}{\;}\widetilde{}$-equivariant unitary operator
$$
\mathcal H^2(\Pi)_{\lambda'}\widehat\otimes \mathcal H^2(\Pi)_{\lambda''}  \stackrel
{\sim}{\longrightarrow}
{\sum_{\ell\in\N}}^\oplus  \mathcal H^2(\Pi)_{\lambda'+\lambda''+2\ell}
$$
to the Hilbert sum associated to weights $\left
\{
\frac1{r_\ell(\lambda',\lambda'')c_\ell(\lambda',\lambda'')}\right\}_{\ell\in\N}$.
 Thus, for every $f\in \mathcal H^2(\Pi)_{\lambda'}\widehat\otimes \mathcal H^2(\Pi)_{\lambda''}$,
\begin{eqnarray*}
&&\Vert f\Vert_{\mathcal H^2(\Pi)_{\lambda'}\widehat\otimes \mathcal H^2(\Pi)_{\lambda''}}^2\\
&=&
\sum_{\ell=0}^\infty \frac1{r_\ell(\lambda',\lambda'')c_\ell(\lambda',\lambda'')}
\Vert \left(\mathcal{RC}_{\lambda',\lambda''}(f)\right)_\ell\Vert_{\mathcal H^2(\Pi)_{\lambda'+\lambda''+2\ell}}^2.
\end{eqnarray*}
\item[(2)]
Collecting the holographic operators
$\Psi_{\lambda',\lambda''}^{\lambda'''}$, we define
the holographic 
 transform 
 $$
 \Psi_{\lambda',\lambda''}:\bigoplus_{\ell\in\N} 
 \mathcal H^2(\Pi)_{\lambda'+\lambda''+2\ell}{\longrightarrow}
\mathcal H^2(\Pi)_{\lambda'}\widehat\otimes \mathcal H^2(\Pi)_{\lambda''}
$$
by
 $$
 \Psi_{\lambda',\lambda''}:=\bigoplus_{\ell=0}^\infty \Psi_{\lambda',\lambda''}^{\lambda'+\lambda''+2\ell}.
 $$
  Then $\Psi_{\lambda',\lambda''}$
induces an ${SL(2,\R)}{\;}\widetilde{}$-equivariant unitary operator
$$
\sum_{\ell=0}^\infty{}^{\oplus} \mathcal H^2(\Pi)_{\lambda'+\lambda''+2\ell}\stackrel
{\sim}{\longrightarrow}
\mathcal H^2(\Pi)_{\lambda'}\widehat\otimes \mathcal H^2(\Pi)_{\lambda''}
$$
from the Hilbert sum associated to the weights $\left\{\frac{c_\ell(\lambda',\lambda'')}{r_\ell(\lambda',\lambda'')}\right\}_{\ell\in\N}$. 
Thus,
$$
\Vert \Psi_{\lambda',\lambda''} g\Vert^2_{\mathcal H^2(\Pi)_{\lambda'}\widehat\otimes \mathcal H^2(\Pi)_{\lambda''}}=
\sum_{\ell=0}^\infty 
\frac{c_\ell(\lambda',\lambda'')}{r_\ell(\lambda',\lambda'')}\Vert g_\ell\Vert^2_{\mathcal H^2(\Pi)_{\lambda'+\lambda''+2\ell}}
$$
for $g=(g_\ell)_{\ell\in\N}$.
\end{itemize}
\end{thm}

Theorem \ref{thm:PPTT4RCB} will be proved in Section \ref{subsec:pfthmPPTT4RCB}. It
gives quantitative information on the classical
branching law (\emph{fusion rule})
of the tensor product of two holomorphic discrete series representations $\pi_{\lambda'}$ and $\pi_{\lambda''}$ 
that decomposes  into a multiplicity-free direct Hilbert sum of irreducible unitary representations 
 when $\lambda',\lambda''>1$ \cite{Mol,Repka}:
\begin{equation}\label{eqn:absbra}
\pi_{\lambda'}\widehat\otimes \pi_{\lambda''}
\simeq
{\sum_{\ell\in\N}}^{\oplus}\pi_{\lambda'+\lambda''+2\ell}. 
\end{equation}

The projection to each irreducible summand in the decomposition \eqref{eqn:absbra} is given as the composition of
the corresponding Rankin--Cohen operator and the holographic operator
in the holomorphic model. 
Thus Theorem \ref{thm:RCinv} (and Proposition
\ref{prop:PsiRC} below) shows the following corollary.
\begin{cor}
[projection operator]
\label{cor:20181103}
Suppose $\lambda', \lambda'', \lambda'''>1$
 and $\ell :=\frac 1 2(\lambda'''- \lambda'- \lambda'') \in {\mathbb{N}}$.  
Then 
\[\frac{1}{c_{\ell}(\lambda', \lambda'')}
\Psi_{\lambda',\lambda''}^{\lambda'''}
\circ
\mathcal{RC}_{\lambda',\lambda''}^{\lambda'''}
=
\frac{1}{r_{\ell}(\lambda', \lambda'')c_{\ell}(\lambda', \lambda'')}
(\mathcal{RC}_{\lambda',\lambda''}^{\lambda'''})^{\ast}
\circ
\mathcal{RC}_{\lambda',\lambda''}^{\lambda'''}
\]
is the projection operator of the Hilbert space
${\mathcal{H}}^2(\Pi)_{\lambda'}
\widehat{\otimes} 
{\mathcal{H}}^2(\Pi)_{\lambda''}$
onto the irreducible summand 
 which is isomorphic to 
${\mathcal{H}}^2(\Pi)_{\lambda'''}$, see \eqref{eqn:absbra}.  
\end{cor}


\subsection{Holographic transform in the $L^2$-model}

By the Fourier--Laplace transform, the weighted Bergman space $\mathcal H^2(\Pi)_\lambda$ realized in the space
of holomorphic functions on the upper half plane $\Pi$ is mapped into the space of functions supported on the
positive axis $\R_+$, more precisely, onto the Hilbert space $L^2(\R_+, x^{1-\lambda}d x)$, giving thus rise to an
$L^2$-\emph{model} of the same representation of ${SL(2,\R)}{\;}\widetilde{}$ (Fact \ref{fact:Laplace}). 
We shall find closed formul{\ae} for
the symmetry breaking transform and the holographic transform also in this model and give an answer to Problems A and B,
see Theorems \ref{thm:phi}, \ref{thm:FRCinv} and \ref{thm:CL2SL2}. The results in the $L^2$-model 
give a new interpretation of the classical theory of the Jacobi transform,
 and also play a key role in proving the theorems for the holomorphic model, see Section \ref{subsec:proofs}.


\subsubsection{$L^2$-model of holomorphic discrete series}

 For $\lambda>1$ we consider the Hilbert space $L^2(\R_+)_\lambda:=L^2(\R_+,x^{1-\lambda}d x)$. 

\begin{fact}\label{fact:Laplace}
Suppose $\lambda>1$.
The Fourier--Laplace transform
$$
\mathcal F\colon F\mapsto \mathcal F F(\zeta):= \int_0^\infty F(z)e^{i\zeta z} dz,
$$
is an isometry from $L^2(\R_+)_\lambda$ onto the weighted Bergman space $\mathcal H^2(\Pi)_\lambda$ up to scalar multiplication.
To be precise, we have 
 $$
\Vert\mathcal FF\Vert^2_{\mathcal H^2(\Pi)_\lambda}=b(\lambda)\Vert F\Vert^2_{L^2(\R_+)_\lambda}
$$
for all $F\in L^2(\R_+)_\lambda$ $($see \emph{e.g.} \cite[Thm.~XIII.1.1]{FK}$)$, where
$$
b(\lambda):=2^{2-\lambda}\pi\Gamma(\lambda-1).
$$
\end{fact}

For $\lambda>1$, via the unitary (up to scaling) map $\mathcal F \colon L^2(\R_+)_\lambda \stackrel
{\sim}{\longrightarrow}
\mathcal H^2(\Pi)_\lambda$, we define an irreducible unitary representation of
${SL(2,\R)}{\;}\widetilde{}$ on $L^2(\R_+)_\lambda$, which
 is referred to as the \emph {$L^2$-model} of the holomorphic discrete series representation $\pi_\lambda$.

 We shall write 
 $$\mathcal F_1\equiv\mathcal F\quad \rm{and}\quad\mathcal F_2:=\mathcal F
\otimes \mathcal F
$$ 
in order to distinguish the framework  of functions of one or two variables, respectively, 
and we write, by abuse of notations, 
\begin{equation}\label{eqn:L2two}
L^2(\R_+^2)_{\lambda',\lambda''}:
=L^2(\R_+\times\R_+,x^{1-\lambda'}y^{1-\lambda''}dx\;dy)\simeq
L^2(\R_+)_{\lambda'}\widehat\otimes L^2(\R_+)_{\lambda''}.
\end{equation}

\subsubsection{Construction of discrete summands in the $L^2$-model}\label{subsec:phi}

Via the Fourier--Laplace transform, we can define the counterpart for the $L^2$-model of the Rankin--Cohen
bidifferential operator $\mathcal{RC}_{\lambda',\lambda''}^{\lambda'''}$ and the holographic integral operator
$\Psi_{\lambda',\lambda''}^{\lambda'''}$ \eqref{eqn:psi} by
\begin{equation}\label{eqn:RCFourier}
   \widehat{\mathcal{RC}}_{\lambda',\lambda''}^{\lambda'''}:=\mathcal F_1^{-1}\circ
   {\mathcal{RC}_{\lambda',\lambda''}^{\lambda'''}}\circ \mathcal F_2.
\end{equation}
\begin{equation}\label{eqn:FPsi}
   \widehat{\Psi}_{\lambda',\lambda''}^{\lambda'''}:= \mathcal F_2^{-1}\circ
    {\Psi_{\lambda',\lambda''}^{\lambda'''}}\circ \mathcal F_1.
\end{equation}
We know from \cite{KP16a} that $\mathcal{RC}_{\lambda',\lambda''}^{\lambda'''}$
is continuous between the weighted Bergman spaces, and so is
$\widehat{\mathcal{RC}}_{\lambda',\lambda''}^{\lambda'''}$.
In turn, $\widehat{\Psi}_{\lambda',\lambda''}^{\lambda'''}$ is continuous between the Hilbert spaces by \eqref{eqn:3holograph} below, hence so is
${\Psi}_{\lambda',\lambda''}^{\lambda'''}$. Alternatively, the
continuity of ${\Psi}_{\lambda',\lambda''}^{\lambda'''}$ is also given by that of another
holographic operator
${\Phi}_{\lambda',\lambda''}^{\lambda'''}$ introduced in
Definition \ref{def:Phi} (Proposition \ref{prop:phinorm}).
The
 following commutative diagrams summarize these definitions: 
$$
\small{
\xymatrix{
 \ar@/_5pc/[dd]_{\widehat{\mathcal{RC}}_{\lambda',\lambda''}^{\lambda'''}}
  L^2(\R_+^2,x^{1-\lambda'} y^{1-\lambda''}dxdy)
 \ar[r]^-{\mathcal F_{2}}
 \ar @{} [dr] |{\circlearrowright}
  \ar[d]
 & \mathcal H^2_{\lambda'}(\Pi) \widehat\otimes \mathcal H^2_{\lambda''}(\Pi) 
 \ar[d]^{\mathcal{R}_{\lambda',\lambda''}^{\lambda'''}} 
  \ar@/^5pc/[dd]^{\mathcal{RC}_{\lambda',\lambda''}^{\lambda'''}}
 \\
 L^2(\R_+^2,x^{1-\lambda'}
y^{1-\lambda''}dx d y)
\ar[r]^-{\mathcal F_{2}}
 \ar @{} [dr] |{\circlearrowright}
 \ar[d]
&\mathcal H^2_{\lambda'}(\Pi) \widehat\otimes \mathcal H^2_{\lambda''}(\Pi) 
\ar[d]^{\mathrm{Rest}} 
\\
L^2(\R_+,z^{1-\lambda'''}dz)\ar[r]_-{\mathcal F_{1}} 
 &
 \mathcal H^2_{\lambda'''}(\Pi) 
 }
 }
$$
{\centerline{ Diagram 2.1. Symmetry breaking operators
$\pi_{\lambda'}\widehat\otimes\,\pi_{\lambda''}\twoheadrightarrow\pi_{\lambda'''}$}
{\centerline{ for holomorphic
 and $L^2$-models.}}

$$
\small{
\xymatrix{
L^2(\R_+^2,x^{1-\lambda'} y^{1-\lambda''}dx d y)
 \ar[r]^-{\mathcal F_{2}} 
\ar @{}[rd] |{\circlearrowright}
& \mathcal H^2_{\lambda'}(\Pi) \widehat\otimes \mathcal H^2_{\lambda''}(\Pi) 
 \\
 L^2(\R_+,z^{1-\lambda'''}dz)
 \ar[r]^-{\mathcal F_{1}} 
  \ar[u]^{\qquad \widehat{\Psi}_{\lambda',\lambda''}^{\lambda'''}=(c'\Phi_{\lambda',\lambda''}^{\lambda'''})}
 & \mathcal H^2_{\lambda'''}(\Pi)
\ar[u]_{\Psi_{\lambda',\lambda''}^{\lambda'''}(=c''(\mathcal{RC}_{\lambda',\lambda''}^{\lambda'''})^*)} 
}
}
$$

{\centerline{Diagram 2.2. Holographic operators $\pi_{\lambda'''}\hookrightarrow
\pi_{\lambda'}\widehat\otimes\pi_{\lambda''}$}}
{\centerline{
 for holomorphic and $L^2$-models.}}
\vskip10pt

We shall give an explicit integral formula of the symmetry breaking operator $\widehat{\mathcal{RC}}_{\lambda',\lambda''}^{\lambda'''}$
 in the $L^2$-model in Proposition \ref{lem:1}.
On the other hand, we observe the holographic operator in the $L^2$-model
 has the following three important characteristics: 
\begin{itemize}
    \item[(1)] the Fourier transform $\widehat \Psi_{\lambda',\lambda''}^{\lambda'''}$
    of the holographic operator $\Psi_{\lambda',\lambda''}^{\lambda'''}$, see \eqref{eqn:FPsi};
    \item[(2)] the adjoint of $\widehat{\mathcal{RC}}_{\lambda',\lambda''}^{\lambda'''}$ (Proposition \ref{lem:2});
    \item[(3)] the multiplication operator $\Phi_{\lambda',\lambda''}^{\lambda'''}$, see \eqref{eqn:defPhi} below for definition.
\end{itemize}

These three approaches may be summarized as the following identities:
\begin{equation}\label{eqn:3holograph}
\widehat\Psi_{\lambda',\lambda''}^{\lambda'''}
=\left(\widehat{\mathcal{RC}}_{\lambda',\lambda''}^{\lambda'''}\right)^*
=i^\ell\Phi_{\lambda',\lambda''}^{\lambda'''},
\end{equation}
see Propositions \ref{lem:2} and \ref{prop:PhiPsi}. 
The third characteristic is remarkable as it does not involve
any integration or differentiation.  
For this reason, 
 we adopt it as our definition  
of holographic transform in the $L^2$-model, see Definition \ref{def:Phi} below.

For $\alpha,\beta\in\C$ and $\ell\in\N$,  let $P_\ell^{\alpha,\beta}(x)$ be the Jacobi polynomial of degree $\ell$, see \eqref{eqn:Jacobi} in Appendix.
\begin{defn}\label{def:Phi}
Retain the setting 
 that $\lambda',\lambda'',\lambda'''\in\C$ with $\ell:=\frac12(\lambda'''-\lambda'-\lambda'')\in\N$.
  For a function $h(z)$ of one variable $z$ ($z>0$), we define a function of two variables $x,y$ ($x,y>0$) by
   \begin{equation}\label{eqn:defPhi}
      \left(\Phi_{\lambda',\lambda''}^{\lambda'''}h\right)(x,y):=
      \frac{x^{\lambda'-1}y^{\lambda''-1}}{(x+y)^{\lambda'+\lambda''+\ell-1}} P_\ell^{\lambda'-1,\lambda''-1}
      \left(\frac{y-x}{x+y}\right)\cdot h(x+y).
   \end{equation}
\end{defn}

\begin{thm}[holographic operator in the $L^2$-model]\label{thm:phi}
Suppose ${\lambda',\lambda''},{\lambda'''}>1$ such that $\ell:=\frac12(\lambda'''-\lambda'-\lambda'')\in\N$.
Then, $\Phi_{\lambda',\lambda''}^{\lambda'''}$ induces
an ${SL(2,\R)}{\;}\widetilde{}$-equivariant continuous homomorphism between the Hilbert spaces: 
\begin{eqnarray*}
\Phi_{\lambda',\lambda''}^{\lambda'''}\colon L^2(\R_+,z^{1-\lambda'''}dz)&\To&L^2(\R_+,x^{1-\lambda'}d x)\widehat\otimes L^2(\R_+,
y^{1-\lambda''}d y).
\end{eqnarray*}

\end{thm}
Theorem \ref{thm:phi} will be proved in Section \ref{subsec:pfthmphi}.
\begin{rem}
Using  the notation \eqref{eqn:Pxy} below, we may also write: 
$$
(\Phi_{\lambda',\lambda''}^{\lambda'''}h)(x,y)=(-1)^\ell\frac{x^{\lambda'-1}y^{\lambda''-1}}{(x+y)^{\lambda'''-1}} \widetilde P_\ell^{\lambda'-1,\lambda''-1}\left(x,y\right)\cdot
h(x+y).
$$
\end{rem}


\subsubsection{Symmetry breaking transform in the $L^2$-model and its inversion}

In this subsection, we give an inversion formula of the 
symmetry breaking operator
$\widehat{\mathcal{RC}}_{\lambda',\lambda''}^{\lambda'''}$ in the $L^2$-model by using the holographic operators 
$\Phi_{\lambda',\lambda''}^{\lambda'''}$ (Definition \ref{def:Phi}). The symmetry breaking operator
$\widehat{\mathcal{RC}}_{\lambda',\lambda''}^{\lambda'''}$ was defined originally as the Fourier transform of the 
Rankin--Cohen bidifferential operator ${\mathcal{RC}}_{\lambda',\lambda''}^{\lambda'''}$ (see \eqref{eqn:RCFourier}) but we give a simpler expression 
as an integral operator (\emph{Jacobi transform}).
\begin{prop}\label{lem:1}
Suppose $\lambda',\lambda''>1$ and $\ell\in\N$. Then for any $F\in C_c(\R_+\times\R_+)$, the following identity holds:
$$
\left(\widehat{\mathcal{RC}}_{\lambda',\lambda''}^{\lambda'+\lambda''+2\ell}F\right)(z)
= \frac{z^{\ell+1}}{2i^\ell}\int_{-1}^1 P_\ell^{\lambda'-1,\lambda''-1}(v) F\left(\frac z2(1-v),\frac z2(1+v)\right)dv.
$$
\end{prop}
See Section \ref{subsec:FTRC} for a proof.

Collecting the operators $\widehat{\mathcal{RC}}_{\lambda',\lambda''}^{\lambda'''}$, we define a
 symmetry breaking transform 
\begin{equation}\label{eqn:defRCTL2}
    \widehat{\mathcal{RC}}_{\lambda',\lambda''}\colon L^2(\R_+)_{\lambda'} \widehat\otimes 
    L^2(\R_+)_{\lambda''}
      \To
      \sum_{\ell\in\N}{}^{\oplus} L^2(\R_+)_{\lambda'+\lambda''+2\ell}
      \end{equation}
 by 
$
\left(  \widehat{\mathcal{RC}}_{\lambda',\lambda''} (F)\right)_\ell
         :=
         \widehat{\mathcal{RC}}_{\lambda',\lambda''}^{\lambda'+\lambda''+2\ell}(F)\quad\mathrm{for}\,\,\ell\in\N.
$
Then it can be inverted by using the holographic operators $\Phi_{\lambda',\lambda''}^{\lambda'''}$ (Definition \ref{def:Phi}) as follows.

\begin{thm}\label{thm:FRCinv}
    Suppose $\lambda', \lambda''>1$. Then for any element $F\in L^2(\R_+)_{\lambda'} \widehat\otimes 
    L^2(\R_+)_{\lambda''}$, one has
    $$
       F=\sum_{\ell=0}^\infty \frac{i^\ell }{c_\ell(\lambda',\lambda'')}
       \Phi_{\lambda',\lambda''}^{\lambda'+\lambda''+2\ell}
         \left(  \widehat{\mathcal{RC}}_{\lambda',\lambda''} (F)\right)_\ell,
     $$
     and
     $$
     \Vert F\Vert^2_{L^2(\R_+)_{\lambda'}\widehat\otimes L^2(\R_+)_{\lambda''}}
     =
    \sum_{\ell=0}^\infty
    \frac1{c_\ell(\lambda',\lambda'')}\left\Vert\left(\widehat{\mathcal{RC}}_{\lambda',\lambda''}(F)\right)_\ell\right\Vert^2_{L^2(\R_+)_{\lambda'+\lambda''+2\ell}}. 
     $$
\end{thm}

Theorem \ref{thm:FRCinv} will be proved in Section \ref{subsec:pfthmFRCinv}. 
It gives an answer to Problem A.2 and Problem B in the $L^2$-model.

\begin{rem}\label{rem:180716}
 The Jacobi transform (see \emph{e.g.} \cite[Chap. 15]{DB}) defined by
$$
H(v)\mapsto \left(J^{\alpha,\beta}(H)\right)_\ell:= \int_{-1}^1 H(v) P^{\alpha,\beta}_\ell(v)(1-v)^\alpha(1+v)^\beta dv
$$
is inverted by the following formula:
\begin{equation}\label{eqn:invJacobi}
H(v)=\sum_{\ell=0}^\infty d_\ell(\alpha,\beta)  \left(J^{\alpha,\beta}(H)\right)_\ell  P^{\alpha,\beta}_\ell(v),
\end{equation}
where we set
\begin{eqnarray*}
d_\ell(\alpha,\beta)&:=&\frac{\ell! (\alpha+\beta+2\ell+1)\Gamma(\alpha+\beta+\ell+1)}{2^{\alpha+\beta+1}\Gamma(\alpha+\ell+1)\Gamma(\beta+\ell+1)}
\\
&=&\frac1{2^{\alpha+\beta+1}c_\ell(\alpha+1,\beta+1)}.
\end{eqnarray*}
By change of variables, we can see that Theorem \ref{thm:FRCinv} is equivalent to \eqref{eqn:invJacobi}
applied to
$$
H(v)=(1-v)^{-\alpha}(1+v)^{-\beta} F\left(\frac z2(1-v),\frac z2(1+v)\right)
$$
with $\alpha=\lambda'-1$ and $\beta=\lambda''-1$.
\end{rem}


\subsubsection{Parseval--Plancherel type theorem for the holographic transform in the $L^2$-model}

Collecting the holographic operators $ \Phi_{\lambda',\lambda''}^{\lambda'''}$, we define the holographic
transform
\begin{equation}\label{eqn:defPhiT}
\Phi_{\lambda',\lambda''}\colon \bigoplus_{\ell\in\N}  L^2(\R_+)_{\lambda'+\lambda''+2\ell} \To
 L^2(\R_+)_{\lambda'} \widehat \otimes L^2(\R_+)_{\lambda''}  
\end{equation}
by
$$
\Phi_{\lambda',\lambda''}:=\bigoplus_{\ell=0}^\infty \Phi_{\lambda',\lambda''}^{\lambda'+\lambda''+2\ell}.
$$

This transform is the counterpart in the $L^2$-model of the holographic transform
$ \Psi_{\lambda',\lambda''}$ (Theorem \ref{thm:PPTT4RCB} (2)) defined in the holomorphic model.

\begin{thm}\label{thm:CL2SL2}
Suppose $\lambda',\lambda''>1$ and $\ell\in\N$. Then, the holographic transform
$
\Phi_{\lambda',\lambda''}
$
induces an ${SL(2,\R)}{\;}\widetilde{}$ -equivariant unitary operator
$$
\sum_{\ell=0}^\infty{}^{\oplus}L^2(\R_+)_{\lambda'+\lambda''+2\ell}\stackrel
{\sim}{\longrightarrow}
L^2(\R_+)_{\lambda'}\widehat\otimes L^2(\R_+)_{\lambda''}
$$
subject to the following Parseval--Plancherel type formula:

$$
\Vert \Phi_{\lambda',\lambda''} h\Vert^2_{L^2(\R_+^2)_{\lambda',\lambda''}}=
\sum_{\ell=0}^\infty 
c_\ell(\lambda',\lambda'')
\Vert h_\ell\Vert^2_{L^2(\R_+)_{\lambda'+\lambda''+2\ell}},
$$
for $h=(h_\ell)_{\ell\in\N}$ with $h_\ell\in L^2(\R_+)_{\lambda'+\lambda''+2\ell}$.
\end{thm}

Theorem \ref{thm:CL2SL2} will be proved in Section \ref{subsec:PPTPHI}.

\subsubsection{Representation theoretic interpretation of the Plancherel density}

The weights $c_\ell(\lambda',\lambda'')$ in the Plancherel formula (Theorem \ref{thm:CL2SL2}) are obviously positive when $\lambda',\lambda''>1$. We discuss the zeros of the meromorphic continuation of $c_\ell(\lambda',\lambda'')$ when we allow $\lambda'$ and
$\lambda''$ to wander outside the region $\lambda',\lambda''>1$, so that
$\pi_{\lambda'}$ and $\pi_{\lambda''}$ may not be (relative) holomorphic discrete series representations.

Assume furthermore that $\lambda',\lambda'',\lambda'''\in\Z$ such that
 $\ell:=\frac12(\lambda'''-\lambda'-\lambda'')\in\N$. Then the following four conditions on
$(\lambda',\lambda'',\lambda''')$ are equivalent (see \cite[Thm.~9.1]{KP16b}):
\begin{itemize}
\item[(i)] $c_\ell(\lambda',\lambda'')=0$;
\item[(ii)]
$2\geq\lambda'+\lambda''+\lambda'''$ and $\lambda'''\geq\vert\lambda'-\lambda''\vert+2$;
\item[(iii)] the Rankin--Cohen bilinear operator
$\mathcal{RC}_{\lambda',\lambda''}^{\lambda'''}$ vanishes;
\item[(iv)] $
\dim\mathrm{Hom}_{{SL(2,\R)}{\;}\widetilde{}}\left(
\mathcal O(\Pi\times\Pi,\mathcal L_{\lambda'}\,\boxtimes \,\mathcal L_{\lambda''}),
\mathcal O(\Pi,\mathcal L_{\lambda'''})\right)=2.
$
\end{itemize}

\subsection{Proof of Theorems \ref{thm:psi} and \ref{thm:phi}} \label{subsec:proofs}

In this section, we derive from the Rankin--Cohen bidifferential operators
$\mathcal{RC}_{\lambda',\lambda''}^{\lambda'''}$ the integral intertwining
operators that embed irreducible representations of ${SL(2,\R)}{\;}\widetilde{}$
into the tensor product representations, and give a proof of Theorems \ref{thm:psi} and \ref{thm:phi}.

The key idea is to use symmetry breaking operators $\widehat{\mathcal{RC}}_{\lambda',\lambda''}^{\lambda'''}$
in the $L^2$-model, which fits well into the F-method connecting the Rankin--Cohen operators with the
Jacobi polynomials. The scheme of the proof is summarized in the following diagram:
\begin{equation}\label{diagram:squig}
\xymatrix{
\mathcal{RC}_{\lambda',\lambda''}^{\lambda'''} 
\ar@{~>}[d]_{\mathrm{Proposition}\, \ref{lem:1}}
&& \Psi_{\lambda',\lambda''}^{\lambda'''} \ar@{~>}[ll]_{\mathrm{Proposition}\, \ref{prop:PsiRC}} 
& (\mathrm{Theorem}\,\ref{thm:psi})\\
\widehat{\mathcal{RC}}_{\lambda',\lambda''}^{\lambda'''}
\ar@{~>}[rr]_{\mathrm{Proposition}\, \ref{lem:2}}&&  \Phi_{\lambda',\lambda''}^{\lambda'''}
\ar@{~>}[u]_{\mathrm{Proposition}\, \ref{prop:PhiPsi}}
& (\mathrm{Theorem}\,
\ref{thm:phi})
}
\end{equation}

\subsubsection{Jacobi polynomials and Rankin--Cohen bidifferential operators}

We retain the notation and assumption 
that 
$
\ell:=\frac12(\lambda'''-\lambda'-\lambda'')\in\N.
$
{lem}

The nature of the bidifferential symmetry breaking operator $\mathcal{RC}_{\lambda',\lambda''}^{\lambda'''}$ is explained
 in \cite[Thm.\ 8.1]{KP16b} by the F-method, which we recall now.
We inflate  the Jacobi polynomial $P_\ell^{\alpha,\beta}(t)$ 
 (see \eqref{eqn:Jacobi}) into a homogeneous polynomial of degree $\ell$ 
 by 
\begin{eqnarray}\label{eqn:Pxy}
\nonumber\widetilde{P}_\ell^{\alpha,\beta}(x,y)&:=&(-1)^\ell(x+y)^\ell P_\ell^{\alpha,\beta}\left(\frac{y-x}{x+y}\right)\\
&=&
\sum_{j=0}^\ell
\frac{(-1)^{\ell-j}(\alpha+\beta+\ell+1)_j(\alpha+j+1)_{\ell-j}}{(\ell-j)!j!}(x+y)^{\ell-j}x^j.
\end{eqnarray}
Then we have the following
\begin{prop}\label{prop:RCPnew}
Suppose $\ell:=\frac12(\lambda'''-\lambda'-\lambda'')\in\N$. Then the Rankin--Cohen
bidifferential operator $\mathcal{RC}_{\lambda',\lambda''}^{\lambda'''}$ (see \eqref{rcb}) is given by
$
\mathcal{RC}_{\lambda',\lambda''}^{\lambda'''}=\mathrm{Rest}\circ
\mathcal{R}_{\lambda',\lambda''}^{\lambda'''}
$ 
with
\begin{equation}\label{eqn:RCPnew}
\mathcal{R}_{\lambda',\lambda''}^{\lambda'''}=\widetilde{P}_\ell^{\lambda'-1,\lambda''-1}\left(\frac\partial{\partial \zeta_1},\frac\partial{\partial \zeta_2}\right).
\end{equation}
\end{prop}

\begin{rem}
In  \cite[(9.9)]{KP16b}, we gave a similar formula
\begin{equation}\label{eqn:RCPsel}
\mathcal{R}_{\lambda',\lambda''}^{\lambda'''}={P}_\ell^{\lambda'-1,1-\lambda'''}\left(\frac\partial{\partial \zeta_1},\frac\partial{\partial \zeta_2}\right)
\end{equation}
by using another two-variable function 
$P_\ell^{\alpha,\beta}(x,y):=y^\ell P_\ell^{\alpha,\beta}\left(1+\frac{2x}y\right)$. Our expression \eqref{eqn:RCPnew} is symmetric with respect to the first and second variables.
\end{rem}
\begin{proof}
[Proof of Proposition \ref{prop:RCPnew}]
According to the first Kummer's relation for the hypergeometric function we get (see for instance \cite[8.962]{GR}):
$$
P_\ell^{\alpha,\beta}(x)=\left(\frac{1+x}2\right)^\ell P_\ell^{\alpha,-\alpha-\beta-2\ell-1}\left(\frac{3-x}{1+x}\right),
$$
and therefore
$$
P_\ell^{\lambda'-1,1-\lambda'''}(1-2s)=(1-s)^\ell P_\ell^{\lambda'-1,\lambda''-1}\left(\frac{1+s}{1-s}\right).
$$
Hence the right-hand sides of \eqref{eqn:RCPnew} and \eqref{eqn:RCPsel} are equal to each other.
\end{proof}

\subsubsection{Coordinate change in the $L^2$-model}\label{sec:coord-changeSL2}

For the study of symmetry breaking in the $L^2$-model, we introduce the following coordinates:
\begin{equation}\label{eqn:SLiota}
\iota \colon \R_+\times (-1,1) \stackrel{\sim}{\longrightarrow} \R_+^2, 
\quad
(z,v) \mapsto (x,y):=\left(\frac z2(1-v), \frac z2(1+v)\right).
\end{equation}

 Then, $\iota$ is a diffeomorphism with $d x d y=\frac z2 d z d v$. With the convention \eqref{eqn:abl}
in Section \ref{sec:constants}, we set
\begin{equation}\label{eqn:Azv}
M(z,v)\equiv M_{\lambda',\lambda'',\lambda'''}(z,v)
      := 2^{\alpha+\beta}z^{\ell+1}
(1-v)^{-\alpha}(1+v)^{-\beta}.
\end{equation}
If $(x,y)=\iota (z,v)$, then we have
\begin{equation}\label{eqn:mmAQ-SL}
\frac{z^{1-\lambda'''}}{x^{1-\lambda'}y^{1-\lambda''}}M(z,v)=z^{-\ell}, 
\end{equation}
\begin{equation}\label{eqn:imeasure}
x^{1-\lambda'}y^{1-\lambda''}d x d y=M(z,v)^2z^{-\alpha-\beta-\lambda'''}(1-v)^\alpha(1+v)^\beta dzd v,
\end{equation}
whereas the holographic operator $\Phi_{\lambda',\lambda''}^{\lambda'''}$ (see \eqref{eqn:defPhi})
takes the form
\begin{equation}\label{eqn:iholograph}
\left(\Phi_{\lambda',\lambda''}^{\lambda'''}h\right)\circ\iota(z,v)= M(z,v)^{-1} P^{\alpha,\beta}_\ell(v)h(z).
\end{equation}

\subsubsection{Fourier transform of the Rankin--Cohen bidifferential operators}\label{subsec:FTRC}

We are ready to prove Proposition \ref{lem:1} for an integral expression of the symmetry breaking operator
$\widehat{\mathcal{RC}}_{\lambda',\lambda''}^{\lambda'''}$
 (see \eqref{eqn:RCFourier}).

\begin{proof}[Proof of Proposition \ref{lem:1}]
For a function $F\in L^2(\R_+ ^2)_ {\lambda',\lambda''}$ we set 
$$
G(x,y):=\widetilde{P}_\ell^{\lambda'-1,\lambda''-1}(x,y) F(x,y).
$$ 
By Proposition \ref{prop:RCPnew} the F-method shows that the
 Rankin--Cohen bidifferential operator is induced from the multiplication 
 by the polynomial $\widetilde P_\ell^{\lambda'-1,\lambda''-1}(x,y)$, namely,
\begin{equation}\label{eqn:RCFG}
 (\mathcal{RC}_{\lambda',\lambda''}^{\lambda'''}\mathcal F_2F)(\zeta)=i^\ell(\mathrm{Rest}\circ
\mathcal F_2G)(\zeta).
\end{equation}
The left-hand side of \eqref{eqn:RCFG} equals $\left(\mathcal F_1 \widehat{\mathcal{RC}}_{\lambda',\lambda''}^{\lambda'''} F\right)(\zeta)$
by the definition \eqref{eqn:RCFourier}.
 We compute the right-hand side of \eqref{eqn:RCFG}.
Via the diffeomorphism \eqref{eqn:SLiota}, 
we have  $\widetilde{P}_\ell^{\alpha,\beta}\circ\iota(z,v)=(-1)^\ell z^\ell 
{P}_\ell^{\alpha,\beta}(v)$. 
Thus we get
\begin{eqnarray*}
\left(\mathrm{Rest}\circ \mathcal F_{2}\right)G(\zeta)&=&
\int_0^\infty\int_0^\infty G(x,y) e^{i(x+y)\zeta} d x d y\\
&=&\frac12\int_0^\infty\int_{-1}^1G\circ\iota( z,v)e^{i z\zeta}zdzdv\\
&=&\frac12\mathcal F_{1}(\mathcal JF)(\zeta),
\end{eqnarray*}
where
\begin{eqnarray*}
\mathcal JF(z)&:=&z\int_{-1}^1G\circ\iota(z,v)dv\\
&=&(-1)^\ell z^{\ell+1}\int_{-1}^1
P_\ell^{\lambda'-1,\lambda''-1}(v) F\circ\iota(z,v)dv.
\end{eqnarray*}
Hence Proposition \ref{lem:1} is proved.
\end{proof}


\subsubsection{Three characteristics of holographic operators in the $L^2$-model}

In Section \ref{subsec:phi}, we discussed the three characteristics (1), (2), and (3) of holographic operators in the $L^2$-model. These three
characteristics play a key role in the proof of main theorems. In this subsection we explicate the relationship between
\begin{itemize}
\item[] (2) and (3) in Proposition \ref{lem:2}, and
\item[] (1) and (3) in Proposition \ref{prop:PhiPsi},
\end{itemize}
and prove the formula \eqref {eqn:3holograph}.

\begin{prop}\label{lem:2}
The adjoint of the holographic operator
 $\Phi_{\lambda',\lambda''}^{\lambda'''}$
 (Definition \ref{def:Phi})
is proportional to the Fourier transform of the Rankin--Cohen operator ${\mathcal{RC}}_{\lambda',\lambda''}^{\lambda'''}$:
$$
\left(\Phi_{\lambda',\lambda''}^{\lambda'''}\right)^*=i^\ell\widehat{\mathcal{RC}}_{\lambda',\lambda''}^{\lambda'''}.
$$
\end{prop}
\begin{proof}
We have already seen in Section \ref{subsec:phi} that $\widehat{\mathcal{RC}}_{\lambda',\lambda''}^{\lambda'''}$ is a continuous map between the Hilbert spaces.
Hence we shall work with dense subspaces $C_c(\R_+)$ and $C_c(\R_+^2)$
in $L^2(\R_+)_{\lambda'''}$ and $L^2(\R_+^2)_{\lambda',\lambda''}$, respectively.
Take  $h\in C_c(\R_+)$ and $F\in C_c(\R_+^2)$. 
By the integral expression of $\widehat{\mathcal{RC}}_{\lambda',\lambda''}^{\lambda'''}$ given in Proposition \ref{lem:1}, we have
\begin{eqnarray*}
&&(h,\widehat{\mathcal{RC}}_{\lambda',\lambda''}^{\lambda'''}F)
_{L^2(\R_+,z^{1-\lambda'''}dz)}\\
&=&\frac{i^\ell}2\int_0^\infty
h(z)\overline{ z^{\ell+1}\int_{-1}^1 P_\ell^{\lambda'-1,\lambda''-1}(v) F\circ\iota(z,v)dv} z^{1-\lambda'''}dz
\\
&=&i^\ell\int_0^\infty\int_0^\infty (\Phi_{\lambda',\lambda''}^{\lambda'''}h)(x,y)\overline{F(x,y)} x^{1-\lambda'}y^{1-\lambda''}dxdy
\\
&=&i^\ell(\Phi_{\lambda',\lambda''}^{\lambda'''}h,F)_{L^2(\R_+^2,x^{1-\lambda'}y^{1-\lambda''}dxdy)}.
\end{eqnarray*}
Here, in the second equality we have used \eqref{eqn:imeasure}
and \eqref{eqn:iholograph}. Thus Proposition \ref{lem:2} is 
proved.
\end{proof}

\begin{prop}\label{prop:PhiPsi}
With the notation \eqref{eqn:FPsi}, we have 
$$
\widehat\Psi_{\lambda',\lambda''}^{\lambda'''}
=i ^\ell
\Phi_{\lambda',\lambda''}^{\lambda'''}. 
$$
\end{prop}
Before giving a proof of Proposition \ref{prop:PhiPsi}, we need the following.

\begin{lem}\label{lem:3}
For any $g\in \mathcal H^2(\Pi)_\lambda$, we have
$$
\left(\frac d{dt}\right)^\ell\int_0^\infty z^{-\ell}(\mathcal F_1^{-1}g)(z) e^{izt}dz=i^\ell g(t).
$$
\end{lem}

\begin{proof}
The statement follows from the (classical) Fourier inversion formula
 and the Paley--Wiener theorem for $g$.
\end{proof}

\begin{proof}[Proof of Proposition \ref{prop:PhiPsi}]
It suffices to show
$$
\mathcal F_2\circ\Phi_{\lambda',\lambda''}^{\lambda'''}\circ\mathcal F_1^{-1}
=(-i) ^\ell
\Psi_{\lambda',\lambda''}^{\lambda'''}. 
$$

We set
$$
t(v):=\frac12 \left((\zeta_2-\zeta_1)v+(\zeta_1+\zeta_2)\right).
$$
By the definitions \eqref{eqn:Azv}
 and \eqref{eqn:iholograph} of $M$ and $\Phi_{\lambda',\lambda''}^{\lambda'''}$, we have,

\begin{eqnarray*}
&&2^{\alpha+\beta+1}(\mathcal F_2\circ\Phi_{\lambda',\lambda''}^{\lambda'''}h)(\zeta_1,\zeta_2)\\
&=& 
\int_0^\infty \int_{-1}^1 z^{-\ell} (1-v)^\alpha (1+v)^\beta P^{\alpha,\beta}_\ell(v) h(z) e^{iz t(v)}dvdz
\\
&=&
\frac{(-1)^\ell}{2^{\ell}\ell!}
\int_0^\infty \int_{-1}^1
z^{-\ell}h(z) e^{izt(v)} \left(\frac d{dv}\right)^\ell \left((1-v)^{\alpha+\ell} (1+v)^{\beta+\ell}\right) dvdz
\\
&=&
\frac{i^\ell}{2^{\ell}\ell!}
 \int_{-1}^1
(1-v)^{\alpha+\ell} (1+v)^{\beta+\ell}\left(\frac{dt(v)}{dv}\right)^\ell
\left(\mathcal F_1 h\right)(t(v))
dv,
\end{eqnarray*}
where the second equality follows from the Rodrigues formula \eqref{eqn:JRodrigues} for the Jacobi polynomials, and
the third one from integration by parts and Lemma \ref{lem:3}.
Putting $g=\mathcal F_1 h$, we obtain
\begin{eqnarray*}
&&(\mathcal F_2\circ\Phi_{\lambda',\lambda''}^{\lambda'''}\circ\mathcal F_1^{-1}g)(\zeta_1,\zeta_2)\\
&=&
\frac{(\zeta_1-\zeta_2)^\ell (-i)^\ell}{2^{\alpha+\beta+2\ell+1}\ell!}
\int_{-1}^1 g(t(v))(1-v)^{\alpha+\ell} (1+v)^{\beta+\ell}dv\\
&=&
(-i)^\ell\left(\Psi_{\lambda',\lambda''}^{\lambda'''}g\right)(\zeta_1,\zeta_2).
\end{eqnarray*}

Hence Proposition 
 \ref{prop:PhiPsi} is proved.
\end{proof}

\subsubsection{Proof of Theorem \ref{thm:phi}}\label{subsec:pfthmphi}

In this subsection we give a proof of Theorem \ref{thm:phi}. 

\begin{proof}[Proof of Theorem \ref{thm:phi}]
As the Rankin--Cohen bidifferential operator ${\mathcal{RC}}_{\lambda',\lambda''}^{\lambda'''}$ intertwines the tensor product
$\pi_{\lambda'}\widehat\otimes\pi_{\lambda''}$ and $\pi_{\lambda'''}$, so does its Fourier transform $\widehat{\mathcal{RC}}_{\lambda',\lambda''}^{\lambda'''}$ (see \eqref{eqn:RCFourier}), and in turn its adjoint operator $\left(\widehat{\mathcal{RC}}_{\lambda',\lambda''}^{\lambda'''}\right)^*$
because $\pi_{\lambda'},\pi_{\lambda''}$, and $\pi_{\lambda'''}$ are unitary representations.
Hence Theorem \ref{thm:phi} follows from Proposition \ref{lem:2}.
\end{proof}

\subsubsection{Proof of Theorem \ref{thm:psi}}\label{subsec:pfthmpsi}
Theorem \ref{thm:phi} together with an argument of holomorphic continuation on parameters
completes the proof of Theorem \ref{thm:psi} as follows.

\begin{proof}[Proof of Theorem \ref{thm:psi}]

The second assertion follows from Proposition \ref{prop:PhiPsi} because $\Phi_{\lambda',\lambda''}^{\lambda'''}$
is an intertwining operator as it was shown in Theorem \ref{thm:phi}.

Let $\ell\in\N$ and $\lambda'''=\lambda'+\lambda''+2\ell$. If $(\lambda',\lambda'')\in\C^2$ satisfies \eqref{eqn:HOconverge} then the
 integral \eqref{eqn:psi} converges for all $g\in\mathcal O(\Pi)$, and $\Psi_{\lambda',\lambda''}^{\lambda'''}$ is continuous viewed
as a map from the Montel
space $\mathcal O(\Pi)$ to the one $\mathcal O(\Pi\times\Pi)$.

On the other hand, if furthermore $\lambda',\lambda''$ are real and $\lambda',\lambda''>1$,
then $\Psi_{\lambda',\lambda''}^{\lambda'''}$ is a $G$-homomorphism on $\mathcal H^2(\Pi)_{\lambda'''}$ by the second statement. 
Since $\mathcal H^2(\Pi)_{\lambda'''}$
 is dense in the Montel space $\mathcal O(\Pi)$ as its subspace of $K$-finite functions is already
 dense in
 $\mathcal O(\Pi)$, the continuous map 
$\Psi_{\lambda',\lambda''}^{\lambda'''} \colon \mathcal O(\Pi)\To\mathcal O(\Pi\times\Pi)$ intertwines
$\pi_{\lambda'''}$ and the tensor product representation $\pi_{\lambda'}\widehat\otimes\pi_{\lambda''}$
if $\lambda',\lambda''>1$. Since $\Psi_{\lambda',\lambda''}^{\lambda'''}g\in \mathcal O(\Pi\times\Pi)$ 
 depends holomorphically on $(\lambda',\lambda'')\in\C^2$ subject to \eqref{eqn:HOconverge} and since the actions
 $\pi_{\lambda'}, \pi_{\lambda''}$ and $\pi_{\lambda'+\lambda''+2\ell}$ of ${SL(2,\R)}{\;}\widetilde{}$ also depend holomorphically on  $(\lambda',\lambda'')\in\C^2$,
 the first statement
is shown.
\end{proof}

\subsubsection{Adjoint of the Rankin--Cohen operator}

As the last part of the diagram \eqref{diagram:squig}, we show that $\Psi_{\lambda',\lambda''}^{\lambda'''}$ is the adjoint of the Rankin--Cohen operator $\mathcal{RC}_{\lambda',\lambda''}^{\lambda'''}$ up to scalar multiplication.

Suppose $\lambda',\lambda'',\lambda'''>1$ and $\lambda'''-\lambda'-\lambda''\in2\N$. We regard 
the Rankin--Cohen operator $\mathcal{RC}_{\lambda',\lambda''}^{\lambda'''}$ as a continuous map between
Hilbert spaces
$$
\mathcal{RC}_{\lambda',\lambda''}^{\lambda'''}
\colon
\mathcal H^2(\Pi)_{\lambda'}\widehat\otimes \mathcal H^2(\Pi)_{\lambda''}\longrightarrow \mathcal H^2(\Pi)_{\lambda'''}.
$$
By \eqref{eqn:3holograph} and Lemma \ref{lem:opFourier} below, we obtain
\begin{prop}\label{prop:PsiRC}
Let $\ell:=\frac12(\lambda'''-\lambda'-\lambda'')\in\N$. The adjoint of $\mathcal{RC}_{\lambda',\lambda''}^{\lambda'''}$ is given by
$$
\left(\mathcal{RC}_{\lambda',\lambda''}^{\lambda'''}\right)^*=
r_\ell(\lambda',\lambda'')\Psi_{\lambda',\lambda''}^{\lambda'''}.
$$
\end{prop}

\subsection{Proof of the Parseval--Plancherel type theorem for the symmetry breaking transform and the holographic transform}

In this section we complete the proof  of Theorems \ref{thm:RCinv} and \ref{thm:PPTT4RCB} in the holomorphic model and Theorems
\ref{thm:FRCinv} and \ref{thm:CL2SL2} in the $L^2$-model for
the Parseval--Plancherel type results
for the symmetry breaking and holographic transforms.
Our strategy consists in applying the $F$-method,
 and then in reducing the proof of these theorems
 to the fact that the Jacobi polynomials $\left\{P_\ell^{\alpha,\beta}\right\}_{\ell\in\N}$ form an orthogonal basis of the Hilbert space $L^2([-1,1], (1-v)^{\alpha}(1+v)^{\beta}dv)$.

\subsubsection{Some properties of operators on Hilbert spaces}

We review a general fact on operators on Hilbert spaces. Suppose a Hilbert space $V$ is decomposed into a Hilbert direct sum of closed subspaces 
$\{V_\ell\}_{\ell\in\N}$, that is, $V\simeq\sum_{\ell\in\N}^{\oplus} V_\ell$, where the inner product on $V_\ell$ is induced from that of $V$.
Let $\mathrm{pr}_{V\to V_\ell} \colon V\To V_\ell$ be the projection operator.
 Let $\{W_\ell\}_{\ell\in\N}$ be another family of Hilbert spaces. Suppose that we are given a continuous map $R_\ell \colon V\To W_\ell$ such that the
restriction $R_\ell\vert_{V_\ell}:V_\ell\To W_\ell$ is a unitary operator up to scalar multiplication and $R_\ell\vert_{V_\ell^\perp}\equiv 0$ for every $\ell\in\N$. 
Then the adjoint operator
$R_\ell^* \colon W_\ell\To V$ is an isometry (up to scalar) onto $V_\ell$.
We write $\Vert R_\ell\Vert_{\mathrm{op}}$ for the operator norm of $R_\ell$ and set
$$
C_\ell:=\Vert R_\ell\Vert_{\mathrm{op}}^2.
$$
 The following two lemmas are elementary.
\begin{lem}\label{lem:Hilbertdeco}
{}
\begin{itemize}
\item[(1)] The linear map 
$\displaystyle R:=\bigoplus\limits_{\ell\in\N}R_\ell \colon V\to\bigoplus\limits_{\ell\in\N} W_\ell$ satisfies
\begin{eqnarray}
R_\ell^*R_\ell&=&C_\ell\mathrm{pr}_{V\to V_\ell}, \label{eqn:RR1}\\
\Vert F\Vert^2_V&=&\sum_{\ell\in\N}\frac1{C_\ell}\Vert R_\ell F\Vert^2_{W_\ell}
\quad\mathrm{for\,\, all}\quad F\in V. \label{eqn:Rexpansion}
\end{eqnarray}
In particular, we have the following inversion formula and the unitarity of the map $R$:
\begin{itemize}
\item[(inversion)] $\qquad F=\sum_{\ell\in\N}\frac1{C_\ell}R_\ell^*(R_\ell F)$,
\item[(unitarity)] $\qquad R$ extends to a unitary operator $\displaystyle V\stackrel
{\sim}{\longrightarrow}{\sum\limits_{\ell\in\N}}^\oplus W_\ell$,
where $\displaystyle{\sum_{\ell\in\N}}^\oplus W_\ell$ is the Hilbert sum associated to the weights $\{C_\ell^{-1}\}_{\ell\in\N}$ (see Definition \ref{def:compsum}).
\end{itemize}
\item[(2)] The linear map $R^*:=\bigoplus\limits_{\ell\in\N}R_\ell^* \colon \bigoplus\limits_{\ell\in\N} W_\ell\to V$ satisfies
\begin{eqnarray}
R_\ell R_\ell^*&=&C_\ell\mathrm{id}_{W_\ell}, \label{eqn:RR2}\\
\Vert R_\ell^*w_\ell\Vert^2_V&=&{C_\ell}\Vert w_\ell\Vert^2_{W_\ell}
\quad\mathrm{for\, \, all}\quad w_\ell\in W_\ell. \nonumber
\end{eqnarray}
In particular $R^*$ extends to a unitary operator $\displaystyle{\sum\limits_{\ell\in\N}}^\oplus W_\ell\stackrel
{\sim}{\longrightarrow} V$,
where $\displaystyle{\sum_{\ell\in\N}}^\oplus W_\ell$ is the Hilbert sum associated to the weights $\{C_\ell\}_{\ell\in\N}$.
\end{itemize}
\end{lem}

\begin{lem}\label{lem:opFourier}
Suppose that $H_j$ and $L_j$ $(j=1,2)$ are Hilbert spaces
 and that $\mathcal F_j \colon L_j\To H_j$ are
unitary operators up to scalar multiple. Let $b_j$ be positive numbers such that
$$
\Vert \mathcal F_j(F)\Vert^2_{H_j}=b_j\Vert F\Vert^2_{L_j}\qquad \mathrm{for\,\, all}\quad F\in L_j.
$$
Let $\Psi \colon H_1\To H_2$
 and $D \colon H_2\To H_1$ be continuous linear maps, and
  we define
 $\widehat \Psi\colon L_1\To L_2$ and $\widehat D\colon L_2\To L_1$ by
$$
\widehat \Psi:=\mathcal F_2^{-1}\circ \Psi\circ\mathcal F_1,\qquad
\widehat D:=\mathcal F_1^{-1}\circ D\circ\mathcal F_2.
$$
We set $r:=\frac{b_1}{b_2}$. Then,
\begin{itemize}
\item[(1)] the operator norms of these operators satisfy
$$
\Vert \widehat D\Vert^2_{\mathrm{op}}=\frac 1r \Vert D\Vert^2_{\mathrm{op}}, \qquad
\Vert \widehat \Psi\Vert^2_{\mathrm{op}}=r \Vert \Psi\Vert^2_{\mathrm{op}};
$$
\item[(2)] 
the adjoint operators of $D$ and $\widehat D$ are related as
$$
\widehat{D^*}= r\left(\widehat D\right)^*.
$$
\end{itemize}
\end{lem}


\subsubsection{Parseval--Plancherel type theorem for $\Phi_{\lambda',\lambda''}^{\lambda'''}$}\label{subsec:PPTPHI}

In this subsection, we prove Theorem \ref{thm:CL2SL2}. 
By the (abstract) branching law \eqref{eqn:absbra},
Theorem \ref{thm:CL2SL2} is deduced from the following proposition.

\begin{prop}\label{prop:phinorm}
Suppose $\lambda',\lambda'',\lambda'''>1$
 satisfy $\ell:=\frac12(\lambda'''-\lambda'-\lambda'')\in\N$. Then,
$$
\Vert \Phi_{\lambda',\lambda''}^{\lambda'''} h\Vert^2_{L^2(\R_+^2)_{\lambda',\lambda''}}=
c_\ell (\lambda',\lambda'')\Vert h\Vert^2_{L^2(\R_+)_{\lambda'''}},
$$
for all $h\in L^2(\R_+)_{\lambda'''}$.
Here we recall \eqref{eqn:v-ell}
 for the definition of $c_{\ell}(\lambda',\lambda'')$.  
\end{prop}

We reduce Proposition \ref{prop:phinorm} to
the fact that the Jacobi polynomials are orthogonal polynomials, see \eqref{eqn:Pnorm} in Appendix.

\begin{proof}[Proof of Proposition \ref{prop:phinorm}]
Via the diffeomorphism \eqref{eqn:SLiota}, we get
from the formul{\ae} \eqref{eqn:imeasure} for the measure and \eqref{eqn:iholograph}  for the holographic operator $\Phi_{\lambda',\lambda''}^{\lambda'''}$:
\begin{eqnarray*}
&&\Vert \Phi_{\lambda',\lambda''}^{\lambda'''} h\Vert^2_{L^2(\R_+^2)_{\lambda',\lambda''}}
=
\frac1{2^{\lambda'+\lambda''-1}} \times\\
&&
\int_0^\infty\int_{-1}^1\vert h(z)\vert^2
\left\vert P_\ell^{\lambda'-1,\lambda''-1}(v)\right\vert^2 z^{1-\lambda'''}
\left({1-v}\right)^{\lambda'-1} \left({1+v}\right)^{\lambda''-1}dvdz.
\end{eqnarray*}
By the $L^2$-norm \eqref{eqn:Pnorm} of the Jacobi polynomials, we conclude the proposition.
\end{proof}

\subsubsection{Operator norm of the holographic operator $\Psi_{\lambda',\lambda''}^{\lambda'''}$ in the holomorphic model}\label{subsec:PPTPSY}

\begin{prop}\label{prop:psinorm}
Suppose $\lambda',\lambda'',\lambda'''>1$ and $\ell:=\frac12(\lambda'''-\lambda'-\lambda'')\in\N$. Then
$$
\Vert \Psi_{\lambda',\lambda''}^{\lambda'''}g\Vert^2_{\mathcal H^2(\Pi)_{\lambda'}\widehat\otimes \mathcal H^2(\Pi)_{\lambda''}}
= \frac{c_\ell(\lambda',\lambda'')}
{r_\ell(\lambda',\lambda'')}
\Vert g\Vert^2_{\mathcal H^2(\Pi)_{\lambda'''}}\,\mathrm{for\,\,all}\,\, g\in\mathcal H^2(\Pi)_{\lambda'''}.
$$
\end{prop}

\begin{proof}
By Proposition \ref{prop:PhiPsi} and Fact \ref{fact:Laplace}, we have
$$
\Vert \Psi_{\lambda',\lambda''}^{\lambda'''}g\Vert^2_{\mathcal H^2(\Pi)_{\lambda'}\widehat\otimes \mathcal H^2(\Pi)_{\lambda''}}=
b(\lambda')b(\lambda'')\Vert
\Phi_{\lambda',\lambda''}^{\lambda'''}\mathcal F_1^{-1}g\Vert^2_{L^2(\R_+^2)_{\lambda',\lambda''}}.
$$

By Proposition \ref{prop:phinorm} and Fact \ref{fact:Laplace} again, the right-hand side of the above equality amounts to
$$
\frac{b(\lambda')b(\lambda'')}{b(\lambda''')}c_\ell(\lambda', \lambda'')
\Vert g\Vert^2_{\mathcal H^2(\Pi)_{\lambda'''}}.
$$
Now the proposition follows from the definition \eqref{eqn:b-ell} of $r_\ell(\lambda',\lambda'')$.
\end{proof}

\subsubsection{Norm of the Rankin--Cohen bidifferential operators}

 We find the operator norm of $\mathcal{RC}_{\lambda',\lambda''}^{\lambda'''}$
as below.
\begin{prop}\label{prop:RCopnorm}
Suppose that $\lambda',\lambda''> 1$ and  $\lambda'''=\lambda'+\lambda''+2\ell \,(\ell\in\N)$.
Then the operator norm of the Rankin--Cohen bidifferential operator 
$\mathcal{RC}_{\lambda',\lambda''}^{\lambda'''}$ seen as a map from the weighted Bergman
space $\mathcal H^2(\Pi)_{\lambda'}\widehat\otimes \mathcal H^2(\Pi)_{\lambda''}$ to $\mathcal H^2(\Pi)_{\lambda'''}$ is given by
$$
\Vert \mathcal{RC}_{\lambda',\lambda''}^{\lambda'''} \Vert_{\mathrm{op}}^2=r_\ell(\lambda',\lambda'')
c_\ell(\lambda',\lambda'').
$$
\end{prop}

\begin{proof}
By Lemma \ref{lem:opFourier} (1), we have
$$
\Vert \mathcal{RC}_{\lambda',\lambda''}^{\lambda'''} \Vert_{\mathrm{op}}^2=r_\ell(\lambda',\lambda'')
\Vert \widehat{\mathcal{RC}}_{\lambda',\lambda''}^{\lambda'''} \Vert_{\mathrm{op}}^2,
$$
which equals $r_\ell(\lambda',\lambda'')c_\ell(\lambda',\lambda'')$ by Propositions \ref{lem:2} and
\ref{prop:phinorm}.
\end{proof}

\subsubsection{Proof of Theorem \ref{thm:PPTT4RCB}}\label{subsec:pfthmPPTT4RCB}

Let us complete the proof of the Parseval--Plancherel type theorem for the Rankin--Cohen transform
$\mathcal{RC}_{\lambda',\lambda''}$
 and the holographic transform $\Psi_{\lambda',\lambda''}$.

\begin{proof}[Proof of Theorem \ref{thm:PPTT4RCB}]

(1) We apply Lemma \ref{lem:Hilbertdeco} with
$R_\ell={\mathcal{RC}}_{\lambda',\lambda''}^{\lambda'+\lambda''+2\ell}$.
By Proposition \ref{prop:RCopnorm}, we have
$$
\Vert R_\ell\Vert^2_{\mathrm{op}}=r_\ell({\lambda',\lambda''})c_\ell({\lambda',\lambda''}),
$$
hence the first statement follows from Lemma \ref{lem:Hilbertdeco} (1).

(2) We apply Lemma \ref{lem:Hilbertdeco} with
$R_\ell=\frac1{r_\ell({\lambda',\lambda''})}{\mathcal{RC}}_{\lambda',\lambda''}^{\lambda'+\lambda''+2\ell}$.
By Proposition \ref{prop:RCopnorm}, we have
$$
\Vert R_\ell\Vert^2_{\mathrm{op}}=\frac{c_\ell({\lambda',\lambda''})}{r_\ell({\lambda',\lambda''})}.
$$
Since $\Psi_{\lambda',\lambda''}^{\lambda'+\lambda''+2\ell}=R_\ell^*$ (see Proposition \ref{prop:PsiRC}), we get the second statement by 
Lemma \ref{lem:Hilbertdeco} (2).
\end{proof}


\subsubsection{Proof of Theorem \ref{thm:RCinv}}\label{subsec:pfthmRCinv}

We are ready to complete the proof of Theorem \ref{thm:RCinv}.
\begin{proof}[Proof of Theorem \ref{thm:RCinv}]

By Lemma \ref{lem:Hilbertdeco} (1)
applied to $R_\ell=\mathcal{RC}_{\lambda',\lambda''}^{\lambda'+\lambda''+2\ell}$, the above proof of Theorem \ref{thm:PPTT4RCB} (1)
implies
$$
f=\sum_{\ell=0}^\infty\frac1{r_\ell({\lambda',\lambda''})c_\ell({\lambda',\lambda''})} R_\ell^*R_\ell f
$$
for any $f\in \mathcal H^2(\Pi)_{\lambda'}\widehat\otimes \mathcal H^2(\Pi)_{\lambda''}$. Now Theorem \ref{thm:RCinv}
follows from the equation $R_\ell^*=r_\ell(\lambda',\lambda'')\Psi_{\lambda',\lambda''}^{\lambda'+\lambda''+2\ell}$ (see Proposition \ref{prop:PsiRC}).
\end{proof}


\subsubsection{Proof of Theorem \ref{thm:FRCinv}}\label{subsec:pfthmFRCinv}

Finally, we show Theorem \ref{thm:FRCinv}.
\begin{proof}[Proof of Theorem \ref{thm:FRCinv}]
We apply Lemma \ref{lem:Hilbertdeco} (1) with $R_\ell=\widehat{\mathcal{RC}}_{\lambda',\lambda''}^{\lambda'+\lambda''+2\ell}$.
By Lemma \ref{lem:opFourier} and Proposition \ref{prop:RCopnorm}, we obtain
$$
\Vert R_\ell\Vert^2_{\mathrm{op}}=\frac1{r_\ell(\lambda',\lambda'')}\left\Vert
\mathcal{RC}_{\lambda',\lambda''}^{\lambda'+\lambda''+2\ell}
\right\Vert^2_{\mathrm{op}}= c_\ell(\lambda',\lambda'').
$$
Since
$R_\ell^*=i^\ell\; \Phi_{\lambda',\lambda''}^{\lambda'+\lambda''+2\ell}$ by Proposition \ref{lem:2},
Theorem \ref{thm:FRCinv} follows from Lemma \ref{lem:Hilbertdeco} (1).
\end{proof}


\subsection{Some applications of symmetry breaking and holographic transforms}\label{sec:applications}

We point out two applications of the symmetry breaking and holographic transforms introduced in the previous 
section. First, we provide explicit description of the minimal $K$-types of the ${SL(2,\R)}{\;}\widetilde{}$-module 
$(\pi_\lambda,\mathcal O(\Pi))$ in both holomorphic
model $\pi_{\lambda'}\widehat\otimes\pi_{\lambda''}$
 and $L^2$-model $L^2(\R_+^2)_{\lambda',\lambda''}$ (see Propositions 
\ref{cor:Ktypeholom} and \ref{cor:KtypeL2}). Second, we find in Theorem \ref{thm:180297} an integral expression of any eigenfunction 
for a specific second-order holomorphic partial differential operator arising from the diagonal action of the Casimir in the enveloping algebra.

\subsubsection{Minimal $K$-types}

The minimal $K$-type of the ${SL(2,\R)}{\;}\widetilde{}$-module $(\pi_\lambda,\mathcal O(\Pi))$ 
is given by $\C (\zeta+i)^{-\lambda}$, see \eqref{eqn:piKtype} for the whole set of $K$-types.
 As an application of the integral formula \eqref{eqn:psi} we find an explicit expression for the minimal $K$-types of submodules in the tensor product $\pi_{\lambda'}\widehat\otimes\pi_{\lambda''}$ as follows.

\begin{prop}\label{cor:Ktypeholom}
Suppose ${\mathrm{Re}\;\lambda',\mathrm{Re}\;\lambda''}>0$ and $\lambda'''=\lambda'+\lambda''+2\ell\,(\ell\in\N)$. Then the holomorphic function
$$
(\zeta_1,\zeta_2)\mapsto (\zeta_1-\zeta_2)^\ell(\zeta_1+i)^{-\lambda'-\ell}(\zeta_2+i)^{-\lambda''-\ell}
$$
is a minimal $K$-type in the submodule $\Psi_{\lambda',\lambda''}^{\lambda'''}(\mathcal O(\Pi))$ in $\pi_{\lambda'} \widehat\otimes\pi_{\lambda''}$.
\end{prop}

\begin{proof}
We set $g(\zeta):= (\zeta+i)^{-\lambda'''}$. By the change of variables $t=\frac12(1+v)$, the definition \eqref{eqn:psi} shows
$$
\left(\Psi_{\lambda',\lambda''}^{\lambda'''} g\right) (\zeta_1,\zeta_2)
=\frac1{\ell!} (\zeta_1-\zeta_2)^\ell (\zeta_1+i)^{-\lambda'''}
\int_0^1t^{a-1}(1-t)^{c-a-1}(1-tz)^{-b}dt,
$$
where $a=\lambda''+\ell, b=c=\lambda'''$, and $z=\frac{\zeta_1-\zeta_2}{\zeta_1+i}$.

By the Euler integral  representation of the hypergeometric function ${}_2F_1$, and by the fact that 
${}_2F_1(a,b;b;z)=(1-z)^{-a}$, we obtain
\begin{equation}\label{eqn:181223}
\left(\Psi_{\lambda',\lambda''}^{\lambda'''} g\right) (\zeta_1,\zeta_2)
=\frac1{\ell!}
B(\lambda'+\ell,\lambda''+\ell) (\zeta_1-\zeta_2)^\ell
(\zeta_1+i)^{-\lambda'-\ell} (\zeta_2+i)^{-\lambda''-\ell},
\end{equation}
where $B(\cdot,\cdot)$ stands for the Euler beta function.
\end{proof}

\begin{prop}\label{cor:KtypeL2}
Suppose $\lambda',\lambda''>1$ and $\ell\in\N$. Then 
the function
$$
(x,y)\mapsto (x^{\lambda'-1}e^{-x})(y^{\lambda''-1}e^{-y})(x+y)^\ell P_\ell^{\lambda'-1,\lambda''-1}\left(\frac{y-x}{x+y}\right)
$$
belongs to $L^2(\R_+^2)_{\lambda',\lambda''}$, and gives a minimal $K$-type in the
image of the holographic operator $\Phi_{\lambda',\lambda''}^{\lambda'+\lambda''+2\ell}$.
\end{prop}
\begin{proof}
Since $z^{\lambda'''-1}e^{-z}$ belongs to the minimal $K$-type in the irreducible re\-presentation $L^2(\R_+)_{\lambda'''}$, so does
$\Phi_{\lambda',\lambda''}^{\lambda'''}(z^{\lambda'''-1}e^{-z})$ in the irreducible
representation $\Phi_{\lambda',\lambda''}^{\lambda'''}(L^2(\R_+^2)_{\lambda',\lambda''})$. Then
  the formula \eqref{eqn:defPhi}
of the holographic operator $\Phi_{\lambda',\lambda''}^{\lambda'''}$ with
$\lambda'''=\lambda'+\lambda''+2\ell$ shows Proposition \ref{cor:KtypeL2}.
\end{proof}


\subsubsection{An application of the integral formula}

Fix $\lambda',\lambda''\in\C$ and consider eigenfunctions of the following holomorphic differential operator on $\Pi\times\Pi$:
\begin{equation}\label{eqn:diagCa}
P_{\lambda',\lambda''}:=(\zeta_1-\zeta_2)^2\frac{\partial^2}{\partial\zeta_1\partial\zeta_2}+
(\lambda''\zeta_2+\lambda'-\lambda'')\zeta_1\frac{\partial}{\partial\zeta_1}+
(\lambda'\zeta_1-\lambda'+\lambda'')\zeta_2\frac{\partial}{\partial\zeta_2},
\end{equation}
and define for $\mu\in\C$
$$
\mathcal{S}o\ell(\Pi\times\Pi, \mathcal M_{\lambda',\lambda'',\;\mu}):=\left\{f\in\mathcal O(\Pi\times\Pi)\colon P_{\lambda',\lambda''}f=\mu f\right\}.
$$
The integral transform \eqref{eqn:psi} constructs all eigenfunctions of $P_{\lambda',\lambda''}$ as follows.

\begin{thm}\label{thm:180297}
Suppose $\lambda',\lambda''\in\C$. Then, the following hold.
\begin{itemize}
\item[(1)] $\mathcal{S}o\ell(\Pi\times\Pi, \mathcal M_{\lambda',\lambda'',\;\mu})\neq\{0\}$ if and only if $\mu$ is of the form
$$
\mu=-\ell(\lambda'+\lambda''+\ell-1) \quad\mathrm{for}\,\,\mathrm{some}\,\,\ell\in\N.
$$
\item[(2)] For any $\lambda', \lambda''\in\C$ and $\ell\in\N$,
$$
(\zeta_1-\zeta_2)^\ell (\zeta_1+i)^{-\lambda'-\ell}(\zeta_2+i)^{-\lambda''-\ell}
\in
\mathcal{S}o\ell (\Pi\times\Pi, \mathcal M_{\lambda',\lambda'',-\ell(\lambda'+\lambda''+\ell-1)}).
$$

\item[(3)] If ${\mathrm{Re}\;\lambda',\mathrm{Re}\;\lambda''}>0$ and $\ell\in\N$, then
 the integral transform \eqref{eqn:psi} gives a bijection
$$
\Psi_{\lambda',\lambda''}^{\lambda'+\lambda''+2\ell}
\colon
\mathcal O(\Pi)
\stackrel
{\sim}{\longrightarrow}
\mathcal{S}o\ell(\Pi\times\Pi, \mathcal M_{\lambda',\lambda'',-\ell(\lambda'+\lambda''+\ell-1)}).
$$
The inverse map is proportional to the Rankin--Cohen bidifferential operator, namely,
$$
\mathcal{RC}_{\lambda',\lambda''}^{\lambda'+\lambda''+2\ell}
\circ  \Psi_{\lambda',\lambda''}^{\lambda'+\lambda''+2\ell}= c_\ell(\lambda',\lambda'')\mathrm{id}\quad
\mathrm{on}\,\, \mathcal O(\Pi),
$$
where $c_\ell(\lambda',\lambda'')$ is defined as in \eqref{eqn:v-ell}.
\end{itemize}
\end{thm}

\subsubsection{Quick review of representations of the universal covering group $SL(2,\R){\;}\widetilde{}$}

In order to prove Theorem \ref{thm:180297} we recall some properties of representations of ${SL(2,\R)}{\;}\widetilde{}$.
The universal covering group ${SO(2)}{\;}\widetilde{}$ of the maximal compact subgroup $K=SO(2)$ is isomorphic to $\R$. 
We parametrize its characters $\chi_{\lambda}$
 by $\lambda \in \C$ 
 as an extension of the following group homomorphisms
 originally defined for $\lambda\in\Z$:
$$
\R\simeq{SO(2)}{\;}\widetilde{}\To SO(2)\To\C^\times,\quad \theta\mapsto
\left(\begin{matrix}
\cos\theta&-\sin\theta\\\sin\theta&\cos\theta
\end{matrix}\right)
\mapsto e^{i\lambda\theta}.
$$
The representation $\pi_\lambda$ on $\mathcal O(\Pi)$ given in Section \ref{subsec:pilmd} is a highest weight module with highest weight $-\lambda$ because it has the following $K$-types:
\begin{equation}\label{eqn:piKtype}
-\lambda,-\lambda-2,-\lambda-4,\ldots.
\end{equation}
Choose the standard basis of the Lie algebra $\mathfrak{sl}(2,\R)$:
$$
H:=
\left(\begin{matrix} 1&0\\0&-1\end{matrix}\right),\quad
X:=\left(\begin{matrix} 0&1\\0&0\end{matrix}\right),\quad
Y:=\left(\begin{matrix} 0&0\\1&0\end{matrix}\right).
$$
Then the Casimir element $C$ is expressed as
$
C=\frac18( H^2+2XY+2YX).
$

The infinitesimal action $d\pi_\lambda$ is given by holomorphic differential operators:
\begin{equation}\label{eqn:sl2vect}
d\pi_\lambda(H)=-\lambda-2z\frac{d}{dz},\quad
d\pi_\lambda(X)=-\frac{d}{dz},\quad
d\pi_\lambda(Y)=\lambda z +z^2\frac{d^2}{dz^2},
\end{equation}
and the Casimir element $C$ acts on $(d\pi_\lambda,\mathcal O(\Pi))$ as $d\pi_\lambda(C)=\frac18\lambda(\lambda-2)\mathrm{id}$.
In general, 
 if $\pi$ is a highest weight module of ${SL(2,\R)}{\;}\widetilde{}\,\,$ with highest weight $\nu$ ($\nu\in\C$), then the 
Casimir element is given via $d\pi$ as the scalar
multiplication $\frac18\nu(\nu+2)\mathrm{id}$.

\subsubsection{Proof of Theorem \ref{thm:180297}}\label{subsec:intsol}
\begin{lem}\label{lem:diagCasimir}
The Casimir element $C$ of $\mathfrak{sl}(2,\R)$ acts
 on $\mathcal O(\Pi \times \Pi)$ via $d\pi_{\lambda'}\otimes d\pi_{\lambda''}$
 by
$$
\left(d\pi_{\lambda'}\otimes d\pi_{\lambda''}\right)(\mathrm{diag(C)})=-\frac12 P_{\lambda',\lambda''}
+\frac18(\lambda'+\lambda'')(\lambda'+\lambda''-2),
$$
where the holomorphic differential operator $P_{\lambda',\lambda''}$ is defined in \eqref{eqn:diagCa}.
\end{lem}
\begin{proof}
By \eqref{eqn:sl2vect} and the Leibniz rule,
 we have
\begin{eqnarray*}
\left(d\pi_{\lambda'}\otimes d\pi_{\lambda''}\right)(\mathrm{diag(H)})
&=&
-\lambda'-\lambda''-2\left(\zeta_1\frac{\partial}{\partial\zeta_1}+\zeta_2\frac{\partial}{\partial\zeta_2}\right),\\
\left(d\pi_{\lambda'}\otimes d\pi_{\lambda''}\right)(\mathrm{diag(X)})
&=&
-\frac{\partial}{\partial\zeta_1}-\frac{\partial}{\partial\zeta_2},\\
\left(d\pi_{\lambda'}\otimes d\pi_{\lambda''}\right)(\mathrm{diag(Y)})
&=&
\lambda'\zeta_1+\lambda''\zeta_2+\left(\zeta_1^2\frac{\partial^2}{\partial\zeta_1^2}+\zeta_2^2\frac{\partial^2}{\partial\zeta_2^2}\right).
\end{eqnarray*}
Now the lemma follows by a direct computation.
\end{proof}
The tensor product representation $\pi_{\lambda'}\widehat\otimes \pi_{\lambda''}$ does not always split into a direct sum of irreducible representations in the nonunitary case when $\lambda',\lambda''\in\C$, see \cite{KP16b} for instance.
We determine the set of possible infinitesimal characters
 of subrepresentations of the tensor product $\pi_{\lambda'}\widehat\otimes \pi_{\lambda''}$ in this case.

\begin{lem}\label{lem:subtensor}
Let $\lambda',\lambda''\in\C$. 
Suppose $\pi$ is a subrepresentation  of $\pi_{\lambda'} \widehat\otimes \pi_{\lambda''}$ such that the Casimir element $C$ acts as scalar multiplication via $d\pi$. 
Then this scalar must be of the form
$$
\frac18 (\lambda'+\lambda''+2\ell)(\lambda'+\lambda''+2\ell-2)\quad\mathrm{for}\,\mathrm{some}\,\,\ell\in\N.
$$ 
\end{lem}

\begin{proof}
We use the general theory of discretely decomposable restrictions
 of (nonunitary) representations \cite{kdeco98}.  
First we observe from \eqref{eqn:piKtype}
 that the $K$-types of the tensor product representation $\pi_{\lambda'} \widehat\otimes \pi_{\lambda''}$ are of the form
$$
-\lambda'-\lambda''-2(\ell'+\ell'') \quad\mathrm{for}\,\,\mathrm{some}\,\,\ell',\ell''\in\N.
$$
Thus the tensor product representation $\pi_{\lambda'}\widehat\otimes \pi_{\lambda''}$ on $\mathcal O(\Pi\times\Pi)$
contains the direct sum of $K$-isotypic spaces
\begin{equation}\label{eqn:Ktensor}
\bigoplus_{\ell\in\N}(\ell+1)\chi_{-\lambda'-\lambda''-2\ell}
\end{equation}
as a dense subset, 
 where $(\ell +1)$ stands for the multiplicity.

In particular, each $K$-type occurs in $\pi_{\lambda'} \widehat\otimes \pi_{\lambda''}$ with at most finite multiplicities.
Hence, 
 any subrepresentation $\pi$ is admissible,
 and of highest weight $-\lambda'-\lambda''-2\ell$ for some $\ell\in\N$. 
Therefore, 
 if the Casimir element acts as a scalar via $d\pi$,
 then this scalar must coincide with 
 $\frac18 (\lambda'+\lambda''+2\ell)(\lambda'+\lambda''+2\ell-2)$.
\end{proof}

\begin{proof}[Proof of Theorem \ref{thm:180297}]
(1) By Lemma \ref{lem:diagCasimir}, $\mathcal So\ell(\Pi\times\Pi,\mathcal M_{\lambda',\lambda'',\mu})$ is characterized as the eigen\-space of the Casimir operator $C$ as follows:
\begin{multline}\label{eqn:SolCas}
\mathcal So\ell(\Pi\times\Pi,\mathcal M_{\lambda',\lambda'',\mu})\cr
=\{f\in\mathcal O(\Pi\times\Pi)\colon
\left(d\pi_{\lambda'}\otimes d\pi_{\lambda''}\right)(\mathrm{diag}(C))=e(\lambda',\lambda'',\mu) f\},
\end{multline}
where we set 
$e(\lambda',\lambda'',\mu)=-\frac12\mu+\frac18(\lambda'+\lambda'')(\lambda'+\lambda''-2)$.

On the other hand, Lemma \ref{lem:subtensor} tells that $e(\lambda',\lambda'',\mu)=
\frac18 (\lambda'+\lambda''+2\ell)(\lambda'+\lambda''+2\ell-2)$ for some $\ell\in\N$. 
This gives the desired formula for $\mu$, showing the ``if\," part of the first statement. The ``only if\," part
follows from the second statement.
\vskip5pt

(2) By the assumption $\mathrm{Re}\,\lambda', \mathrm{Re}\,\lambda''>0$, the integral \eqref{eqn:psi}
converges for any $\ell\in\N$.
Since the Casimir element $C$ acts on $\mathcal O(\Pi)$
 as the scalar 
$\frac18\lambda'''(\lambda'''-2)$ via $d\pi_{\lambda'''}$, 
 and since $\Psi_{\lambda',\lambda''}^{\lambda'''}$ is an intertwining operator,
the Casimir element $C$ acts on the image of $\Psi_{\lambda',\lambda''}^{\lambda'''}$ by the same scalar.
Therefore $\Psi_{\lambda',\lambda''}^{\lambda'+\lambda''+2\ell}(\mathcal O(\Pi))\subset\mathcal So\ell(\Pi\times\Pi,\mathcal M_{\lambda',\lambda'',-\ell(\lambda'+\lambda''+\ell-1)})$ by \eqref{eqn:SolCas}. In turn, Proposition \ref{cor:Ktypeholom} and Lemma \ref{lem:diagCasimir} imply
that
$$
\left(   
P_{\lambda',\lambda''} +\ell(\lambda'+\lambda''+\ell-1)
\right)
\left(
(\zeta_1-\zeta_2)^\ell (\zeta_1+i)^{-\lambda'-\ell}(\zeta_2+i)^{-\lambda''-\ell}
\right)
=0
$$
as far as $\mathrm{Re}\,\lambda', \mathrm{Re}\,\lambda''>0$. By the analytic continuation, the equation holds for all $\lambda',\lambda''\in\C$.

(3) We begin with the case where 
$\lambda'$ and $\lambda''$ are real and $\lambda',\lambda''>1$. Then it follows from Proposition
\ref{prop:RCopnorm}, Lemma \ref{lem:Hilbertdeco} (2) and Proposition \ref{prop:PsiRC} that
$$
\mathcal{RC}_{\lambda',\lambda''}^{\lambda'''}\circ\Psi_{\lambda',\lambda''}^{\lambda'''}=
c_\ell({\lambda',\lambda''})\mathrm{id}\quad\mathrm{on}\;{\mathcal H^2(\Pi)_{\lambda'''}}.  
$$
Since ${\mathcal{H}}^2(\Pi)_{\lambda'''}$ is dense in ${\mathcal{O}}(\Pi)$, 
 the equality holds on 
 the whole $\mathcal O(\Pi)$ by continuity.

Moreover,
 since the operators $\mathcal{RC}_{\lambda',\lambda''}^{\lambda'''}$ and $\Psi_{\lambda',\lambda''}^{\lambda'''}$ depend
holomorphically on $(\lambda',\lambda'')\in\C^2$ satisfying \eqref{eqn:HOconverge}, we conclude that
\begin{equation}\label{eqn:RCPsic}
\mathcal{RC}_{\lambda',\lambda''}^{\lambda'''}\circ\Psi_{\lambda',\lambda''}^{\lambda'''}=
c_\ell({\lambda',\lambda''})\mathrm{id}\quad\mathrm{on}\;{\mathcal O(\Pi)}
\end{equation}
for any $(\lambda',\lambda'',\lambda''')$ 
 subject to \eqref{eqn:HOconverge} 
 by analytic continuation.
In particular, $\Psi_{\lambda',\lambda''}^{\lambda'''}$ is injective and $\mathcal{RC}_{\lambda',\lambda''}^{\lambda'''}$ is surjective because $c_\ell(\lambda',\lambda'')\neq0$ in this case.

Let us prove that $ \Psi_{\lambda',\lambda''}^{\lambda'''}\colon \mathcal O(\Pi)\To
\mathcal So\ell(\Pi\times\Pi,\mathcal M_{\lambda',\lambda'',-\ell(\lambda'+\lambda''+\ell-1)})$
is surjective. We let ${SL(2,\R)}{\;}\widetilde{}$ act on $\mathcal O(\Pi\times\Pi)$ via $\pi_{\lambda'}\widehat\otimes
\pi_{\lambda''}$.
 Since the map $\N\To\C,\, \ell\mapsto -\ell(\lambda'+\lambda''+\ell-1)$ is injective
if  $\mathrm{Re}\,\lambda', \mathrm{Re}\,\lambda''>0$,
 we have the following inclusion of ${SL(2,\R)}{\;}\widetilde{}$ -submodules of $\mathcal O(\Pi\times\Pi)$:
$$
\bigoplus_{\ell\in\N} \Psi_{\lambda',\lambda''}^{\lambda'+\lambda''+2\ell}
\left(\mathcal O(\Pi) \right)\subset
\bigoplus_{\ell\in\N}\mathcal So\ell(\Pi\times\Pi,\mathcal M_{\lambda',\lambda'',-\ell(\lambda'+\lambda''+\ell-1)}).
$$

We observe that $K$-multiplicities coincide by \eqref{eqn:Ktensor} because the holographic operator $\Psi_{\lambda',\lambda''}^{\lambda'+\lambda''+2\ell}$ is injective for any $\ell\in\N$. 
Therefore, 
$$\Psi_{\lambda',\lambda''}^{\lambda'+\lambda''+2\ell} \colon \mathcal O(\Pi)\To
\mathcal So\ell(\Pi\times\Pi,\mathcal M_{\lambda',\lambda'',-\ell(\lambda'+\lambda''+\ell-1)})$$
is surjective on the level of $(\mathfrak g, K)$-modules.

Since $c_\ell(\lambda',\lambda'')\neq0$,  the symmetry breaking operator $\mathcal{RC}_{\lambda',\lambda''}^{\lambda'''}$
is injective on $\Psi_{\lambda',\lambda''}^{\lambda'+\lambda''+2\ell}\left(\mathcal O(\Pi)\right)$
by \eqref{eqn:RCPsic}, and therefore the underlying  $(\mathfrak g, K)$-module of
$\mathrm{Ker}\left(\mathcal{RC}_{\lambda',\lambda''}^{\lambda'''}\right)\cap
\mathcal So\ell(\Pi\times\Pi,\mathcal M_{\lambda',\lambda'',-\ell(\lambda'+\lambda''+\ell-1)})$ must be zero for any $\lambda'''=\lambda'+\lambda''+2\ell$
($\ell\in\N$). Hence $\mathcal{RC}_{\lambda',\lambda''}^{\lambda'''}$ is injective
when restricted to $\mathcal So\ell(\Pi\times\Pi,\mathcal M_{\lambda',\lambda'',-\ell(\lambda'+\lambda''+\ell-1)})$. Now we conclude that $\Psi_{\lambda',\lambda''}^{\lambda'''}$ is surjective. 
Thus the theorem is proved.
\end{proof}


\section{Holomorphic Juhl transform and its holographic transform}
\label{sec:HJT}

Another remarkable family of differential operators is provided by conformally covariant differential operators for the pair $S^n\supset S^{n-1}$ of Riemannian manifolds, introduced by Juhl \cite{Juhl}.
These operators are symmetry breaking operators from spherical principal series representations
of the Lorentz group $O(1,n+1)$ to those of the subgroup $O(1,n)$.
This setting is intimately related to the holographic or AdS/CFT correspondence in string theory (see \emph{e.g}. \cite{Ma,W1}).

The \emph{holomorphic Juhl operators} are the holomorphic continuation of Juhl's operators, which map holomorphic functions on the $n$-dimensional Lie ball to those on the $(n-1)$-dimensional Lie ball,
intertwining (relative) discrete series representations of $G=SO_o(2,n)$
with those of the subgroup $G'=SO_o(2,n-1)$, see \cite{KP16b}. 

In this section we solve Problems A and \ref{prob:B} stated in Section \ref{sec:Intro} 
 for the symmetry breaking transform
 associated with the holomorphic Juhl operators. 
We assume $n\geq 3$ throughout this section. The case
$n=2$ can be recovered from the previous section with an appropriate change of parameters.

\subsection{Holomorphic Juhl operators}
\subsubsection{Holomorphic discrete series of $SO_o(2,n)$}
Let $Q_{p,q}$ be the standard quadratic form of signature $(p,q)$ on $\R^{p+q}$.
The indefinite orthogonal group
$$
O(p,q):=\left\{ g\in GL(p+q,\R) : Q_{p,q}(gx)=Q_{p,q}(x)\,\mathrm{for\,all}\,x\in\R^{p+q}
             \right\}
$$
has four connected components when $p,q>0$.
Let $G=SO_o(2,n)$ be the identity component of $O(2,n)$ and $K$ a maximal compact subgroup of $G$.
We write $\mathfrak c(\mathfrak k)$ for the first factor of the Lie algebra
$\mathfrak k\simeq \R\oplus \mathfrak{so}(n)$, and
fix a characteristic element $H_0\in \mathfrak c(\mathfrak k)$ such that ad$(H_0)$ gives
the eigenspace decomposition of $\mathfrak g_\C= \mathrm {Lie}(G)\otimes_\R\C$ as
\begin{equation}\label{eqn:gknn}
\mathfrak g_\C=\mathfrak k_\C+\mathfrak n_++\mathfrak n_-
\end{equation}
for eigenvalues $0,-i$ and $i$, respectively. 
The complex structure
 of the homogeneous space $G/K$ is
 given by the $G$-translation of $\exp\left(\mathrm{ad}\left(\frac\pi2 H_0\right)\right)\in GL_\R(T_o(G/K))$, or
 equivalently, it is
  induced from the Borel embedding $G/K\subset
 G_\C/K_\C\exp\mathfrak n_+$.

Let $\widetilde G$ be the universal covering of $G=SO_o(2,n)$, $p:\widetilde G\To G$ the covering homomorphism, and set $\widetilde K:=p^{-1}(K)$.
For $\lambda\in\C$, we define a character of  $\mathfrak c(\mathfrak k)$ by
$tH_0\mapsto\lambda t$, which lifts  to a character $\C_\lambda$ of $\widetilde K$. 

 Then one can define a $\widetilde G$-equivariant holomorphic line bundle $\mathcal L_\lambda=\widetilde G\times_{\widetilde K}\C_\lambda$ over $X:=\widetilde G/\widetilde K\simeq G/K$ for all $\lambda\in\C$, and
obtain representations $\pi_\lambda^{(n)}$ of $\widetilde G$ on the space $\mathcal O(X,\mathcal L_\lambda)$ of holomorphic sections. The representation $\pi_\lambda^{(n)}$ descends to $G$ if $\lambda\in\Z$.

If $\lambda\in\R$, then the line bundle $\mathcal L_\lambda\To X$ carries a $\widetilde G$-invariant Hermitian metric, and therefore we can define a Hilbert space $\left( \mathcal O\cap L^2\right)(X,\mathcal L_\lambda)$. This Hilbert space  is nonzero if and only if $\lambda>n-1$, and the resulting unitary representation of $\widetilde G$, to be denoted by the same symbol $\pi_\lambda^{(n)}$, is called a \emph{(relative) holomorphic discrete series representation} of $\widetilde G$. For
actual computations we use its realization in the weighted Bergman space as below. 

We define the tube domain
$$
T_\Omega\equiv T_{\Omega(n)}:=\R^n+ i\Omega(n)
$$
by taking $\Omega(n)$ to be the time-like cone in the Minkowski space $\R^{1,n-1}$, namely,
$$
\Omega(n):=\left\{ \eta\in\R^n\;:\;Q_{1,n-1}(\eta)>0,\; \eta_1>0\right\}.
$$

Then the tube domain $T_\Omega$ is biholomorphic to the bounded symmetric domain  of type $\mathrm{IV}_n$, sometimes
referred to as the Lie ball. From a group-theoretic viewpoint, $T_\Omega$ is isomorphic
to the Hermitian symmetric space $X=G/K$, which is realized as an open subset of
$\mathfrak n_- (\simeq\C^n)$ via the following maps
\begin{equation}\label{eqn:BorelGK}
\mathfrak n_-\underset{\mathrm{open}}{\hookrightarrow}
G_\C/K_\C\exp(\mathfrak n_+)\underset{\mathrm{open}}{\supset}G/K.
\end{equation}

The homogeneous holomorphic line bundle $\mathcal L_\lambda\To X$ is trivialized via
the Bruhat decomposition, and the Hilbert space $(\mathcal O\cap L^2)(X,\mathcal L_\lambda)$
is then identified with the weighted Bergman space
$$
\mathcal H^2(T_{\Omega(n)})_\lambda:=\mathcal O(T_{\Omega(n)})\cap
L^2(T_{\Omega(n)}, Q_{1,n-1}(\eta)^{\lambda-n}d\xi d\eta),
$$
on which $G$ acts as a multiplier representation
by
$$
f(\zeta)\mapsto b_\lambda (g,\zeta)f(g^{-1}.\zeta)
$$
for $g\in G$ and $f(\zeta)\in\mathcal O(T_{\Omega(n)})$. Here the multiplier
$$
b_\lambda: G\times T_{\Omega(n)}\To \C^\times
$$
is a 1-cocycle defined by $b_\lambda(g,\zeta):=\chi_{-\lambda}(k(g,\zeta))$, where
$\chi_\lambda \colon K_\C\To \C^\times$ is the holomorphic extension of the character $\C_\lambda$ of $K$ and $k(g,\zeta)$ is an element of $K_\C$ determined by
$$
g^{-1}\exp(\zeta)\in \exp(g^{-1}.\zeta)k(g,\zeta)\exp (\mathfrak n_+),
$$
see \eqref{eqn:BorelGK}.

For $\lambda>n-1$ this Hilbert space admits a reproducing kernel $K_\lambda(\zeta,\tau)$ given
by:
\begin{equation}\label{eqn:Klmbd}
K_\lambda(\zeta,\tau)= k_{\lambda,n}Q_{1,n-1}\left({\zeta-\overline \tau}\right)^{-\lambda},
\end{equation}
where $Q_{1,n-1}(\zeta)$ stands for the
holomorphic extension of
$
Q_{1,n-1},$ and
 we set
$$
k_{\lambda,n}:=\frac{(2i)^{2\lambda}}{(4\pi)^{n}}\frac{\left(\lambda-\frac n2\right)\Gamma(\lambda)}{\Gamma(\lambda-n+1)},
$$
see \cite[Prop.~XIII.1.2]{FK}.
We note that $k_{\lambda,n}\neq0$ if $\lambda>n-1$.

We realize $O(2,n-1)$ as the subgroup of $O(2,n)$
 which fixes the $(n+2)$-th coordinate,
 and set $G'=SO_o(2,n-1)$
 as its identity component.
 By abuse of notation, we write $\widetilde G'$ for the connected
 subgroup of $\widetilde G=SO_o(2,n){\;}\widetilde{}$\,
 corresponding to $G'\subset G$. The subgroup $\widetilde G'$ is simply connected if $n\neq4$.
  Similarly, a 
 (relative) holomorphic discrete series representation $\pi_\nu^{(n-1)}$
of $\widetilde G'$
 is defined for $\nu >n-2$, 
 as an irreducible unitary representation
 on the weighted Bergman space $\mathcal H^2(T_{\Omega(n-1)})_\nu$.
 By abuse of notation, the same symbol $\pi_\nu^{(n-1)}$ will be used to denote a (nonunitary) representation
 on $\mathcal O(T_{\Omega(n-1)},\mathcal L_\nu)$
 for $\nu \in {\mathbb{C}}$.

\subsubsection{Holomorphic Juhl operators}

Let $\Delta_{\C^{1,n-2}}:=\frac{\partial^2}{\partial\zeta_1^2}-\frac{\partial^2}{\partial\zeta_2^2}
 -\cdots-\frac{\partial^2}{\partial\zeta_{n-1}^2}$ be the holomorphic Laplacian on $\C^{n-1}$
 associated to the complexified quadratic form $Q_{1,n-2}$.
For $\alpha \in {\mathbb{C}}$
 and $\ell, k \in {\mathbb{N}}$
 with $\ell \ge 2 k$, 
 we define a polynomial of $\alpha$
 of degree $\ell-k$ 
by 
\begin{equation}\label{eqn:alalpha}
a_k(\ell,\alpha):=
\frac{(-1)^k 2^{\ell-2k} \cdot \Gamma\left(\alpha+\ell-k\right)}
{\Gamma\left(\alpha\right) k! (\ell-2k)!}.
\end{equation}
 The coefficients $a_k(\ell,\alpha)$ appear in the definition \eqref{eqn:gegen} of the Gegenbauer
 polynomials $C_\ell^\alpha(x)$, see Appendix.  
We define a holomorphic differential operator
$\mathcal D_\ell^\alpha$ on $\C^n$ by
\begin{equation}\label{eqn:JuhlonCn}
\mathcal D_\ell^\alpha:=\sum_{k=0}^{\left[\frac\ell2\right]}
a_k\left(\ell,\alpha\right)
 \left(\frac{\partial}{\partial\zeta_n}\right)^{\ell-2k}\Delta^k_{\C^{1,n-2}}. 
\end{equation}

For $\lambda,\nu\in \C$ with $\ell:=\nu-\lambda\in\N$,  
 the holomorphic Juhl operator $D_{\lambda\to\nu} \colon \mathcal O(T_{\Omega(n)})\To
\mathcal O(T_{\Omega(n-1)})$ 
 is defined as the composition 
\begin{eqnarray}\label{eqn:holoJuhl}
D_{\lambda\to\nu}:=\mathrm{Rest}_{\zeta_n=0}\circ 
\mathcal D_\ell^{\lambda-\frac{n-1}2}.
 \end{eqnarray}

The operator $D_{\lambda\to\nu}$ may be viewed as the holomorphic extension
 of the original Juhl operator \cite{Juhl}, 
 which is a conformally covariant differential operator
 $C^{\infty}(S^{n})\to C^{\infty}(S^{n-1})$.  
In our setting, 
 the hyperbolic space $H^n$ is realized
 as a totally real submanifold of the tube domain $T_{\Omega(n)}$, and
likewise, $H^{n-1}$ is that of $T_{\Omega(n-1)}$. 
The restriction of the holomorphic Juhl operator to these real manifolds
 also yields
a conformally covariant operator
$C^\infty(H^n)\To C^\infty(H^{n-1})$, 
see \cite[Thm.~E]{KKPAdS}.

The holomorphic Juhl operator $D_{\lambda\to\nu}$ gives yet another symmetry breaking operator, 
intertwining the (relative) holomorphic discrete series representation $\pi_\lambda^{(n)}$ of $\widetilde G$ and the one $\pi_{\nu}^{(n-1)}$
of the subgroup $\widetilde G'$ \cite[Thm.~6.3]{KP16b}.
Moreover, the differential operator $D_{\lambda\to\nu}$ induces a \emph{continuous} map
between the Bergman spaces by the general theory \cite[Thm.~5.13]{KP16a}.
Its adjoint is denoted by $D_{\lambda\to\nu}^*$.  
We determine the operator norm of $D_{\lambda \to \nu}$
 in Proposition \ref{prop:normJuhl}.


\subsection{Two constants $c_\ell(\lambda)$ and $r_\ell(\lambda)$}\label{sec:const-conf}

Throughout Section \ref{sec:HJT}
the parameter set is $(\lambda,\nu)\in\C^2$ with $\nu-\lambda\in\N$ and $n\geq3$.
We use the following notation:
\begin{equation}
\label{eqn:alphal}
\alpha=\lambda-\frac{n-1}2,\qquad \ell=\nu-\lambda.
\end{equation}

As in Section \ref{sec:constants} devoted to the tensor product case, the main results in this section involve the following two constants:
\begin{eqnarray}\label{eqn:c-ell}
c\equiv c_\ell(\lambda)&:=& \int_{-1}^1\left\vert C_\ell^\alpha(v)\right\vert^2 (1-v^2)^{\alpha-\frac12}dv\nonumber\\
&=&  \frac{\pi 2^{n-2\lambda} \Gamma(2\lambda+\ell-n+1)}{\ell!\left(\lambda+\ell-\frac{n-1}2\right)
\Gamma\left(\lambda-\frac{n-1}2\right)^2},\\
\label{eqn:r-ell}
r\equiv r_\ell(\lambda)&:=&\frac{b_{n-1}(\nu)}{b_n(\lambda)}\nonumber\\
&=&
\frac{\Gamma\left(\lambda+\ell-\frac{n-1}2\right) \Gamma\left(\lambda+\ell-n+2\right)}
{(2\pi)^{\frac32}2^{2\ell+1} \Gamma\left(\lambda-\frac n2\right)\Gamma\left(\lambda-n+1\right)},
\end{eqnarray}
where
$b_n(\lambda)$ is a Plancherel density (see \eqref{eqn:bnlmd} below).


\subsection{Holomorphic Juhl transforms}

\begin{defn}[holomorphic Juhl transform]
For $\lambda\in\C$, the holomorphic Juhl transform $D_\lambda$
 is a collection of the holomorphic Juhl operators
$$
D_\lambda \colon \mathcal O(T_{\Omega(n)})\To
\operatorname{Map}(\mathbb N, \mathcal O(T_{\Omega(n-1)})), 
\quad
f \mapsto \{(D_{\lambda} f)_{\ell}\}_{\ell \in {\mathbb{N}}}, 
$$
where $ (D_\lambda f)_\ell:=D_{\lambda\to\lambda+\ell} f$. 
\end{defn}
The holomorphic Juhl transform $D_\lambda$ intertwines $(\pi_\lambda^{(n)},\mathcal O(T_{\Omega(n)}))$ with the formal direct sum $\widehat\bigoplus_{\ell\in\N}(\pi_{\lambda+\ell}^{(n-1)},\mathcal O(T_{\Omega(n-1)}))$; its
 inversion formula and the corresponding Parseval--Plancherel type theorems
 are given as follows.

\begin{thm}\label{thm:170887}
Suppose $\lambda>n-1$.
\begin{itemize}
\item[(1)] \emph{(inversion formula).} 
 Any $f\in \mathcal H^2(T_{\Omega(n)})_\lambda$ is recovered from 
 $D_{\lambda} f$ by
$$
f=\sum_{\ell=0}^\infty \frac1{r_\ell(\lambda)c_\ell(\lambda)}
D_{\lambda\to{\lambda+\ell}}^* \left(D_\lambda f\right)_\ell.
$$
\item[(2)]
\emph{(Parseval--Plancherel type theorem).}
For every $f\in \mathcal H^2(T_{\Omega(n)})_\lambda$, we have
\begin{equation}
\Vert f\Vert^2_{\mathcal H^2_\lambda(T_{\Omega(n)})}=
\sum_{\ell=0}^\infty \frac1{r_\ell(\lambda)c_\ell(\lambda)}
\Vert \left(D_{\lambda}f\right)_\ell\Vert^2_{\mathcal H^2_{\lambda+\ell}(T_{\Omega(n-1)})}.
\end{equation}
\end{itemize}
\end{thm}
Theorem \ref{thm:170887} is proved in Section \ref{subsec:proofTHM32}. It gives an answer to Problems A.2 and B raised in Section 1 for the holomorphic Juhl transform $D_{\lambda}$. 
An answer to Problem A.1 (explicit integral formula for holographic transform) will be given in Theorem \ref{thm:CIRM18}.

From a representation-theoretic viewpoint, Theorem \ref{thm:170887} gives quantitative information on the corresponding branching law for the restriction of the (relative) holomorphic discrete series representation $\pi_\lambda^{(n)}$ of $\widetilde G$ to the subgroup $\widetilde G'$, which decomposes the restriction $\pi_\lambda^{(n)}\big\vert_{\widetilde G'}$
 into a multiplicity-free direct Hilbert sum of irreducible representations of the subgroup $\widetilde G'$, see \cite[Thm.~8.3]{K08}:
\begin{equation}\label{eqn:absbraConf}
\pi_\lambda^{(n)}\vert_{\widetilde G'}\simeq {\sum_{\ell\in\N}}^\oplus \pi_{\lambda+\ell}^{(n-1)}.
\end{equation}

\begin{cor}[projection operator]
Suppose $\lambda>n-1$ and $\ell:=\nu-\lambda\in\N$. Then
$$
\frac1{r_\ell(\lambda) c_\ell(\lambda)}D_{\lambda\to\nu}^*D_{\lambda\to\nu}
$$
is the projection operator from the Hilbert space $\mathcal H^2(T_{\Omega(n)})_\lambda$
onto the summand which is unitarily equivalent to the irreducible representation $(\pi_\nu^{(n-1)},\mathcal H^2(T_{\Omega(n-1)})_\nu)$,
see \eqref{eqn:absbraConf}.
\end{cor}

$$
$$


\subsection{Key operators in the proof of Theorem \ref{thm:170887}}
~~~\par
\vskip10pt

Analogously to the case of the tensor product representations (Section \ref{sec:RCT}), we introduce the following continuous operators for the proof of Theorem \ref{thm:170887}:
\begin{alignat*}{4}
D_{\lambda\to\nu}^*& \colon \mathcal H^2(T_{\Omega(n-1)})_\nu && \To && \mathcal
 H^2(T_{\Omega(n)})_\lambda
\qquad
&& \mathrm{adjoint\, of}\,\, D_{\lambda\to\nu},\\
\widehat D_{\lambda\to\nu}&\colon L^2(\Omega(n))_\lambda &&\To&& L^2(\Omega(n-1))_\nu
\qquad &&\mathrm{Fourier\,transform\, of}\,\, D_{\lambda\to\nu},\\
\Phi_\lambda^\nu&\colon L^2(\Omega(n-1))_\nu &&\To&&  L^2(\Omega(n))_\lambda
\qquad &&\mathrm{holographic\, operator}.
\end{alignat*}
See \eqref{eqn:JFourier} and \eqref{eqn:hologSO} below for the definitions of $\widehat D_{\lambda\to\nu}$  and $\Phi_\lambda^\nu$, respectively.
\vskip 7pt

The link between these operators may be summarized in the following diagram:

\begin{equation}\label{diagram:squig2}
\xymatrix{
D_{\lambda\to\nu} 
\ar@{~>}[d]_{\mathrm{Proposition}\, \ref{lem:170770}}
\ar@{~>}[rr]^{\mathrm{Theorem}\,\ref{thm:CIRM18}}
&& D_{\lambda\to\nu}^* 
 \\
\widehat{D}_{\lambda\to\nu}
\ar@{~>}[rr]_{\mathrm{Proposition}\, \ref{lem:170772}}&&  \Phi_{\lambda}^{\nu}
\ar@{~>}[u]_{\mathrm{Proposition}\,\ref{prop:181248}}
}
\end{equation}
\vskip 7pt
Among them, the holographic operator $\Phi_\lambda^\nu$ in the $L^2$-model will play a crucial role in the proof of Theorem \ref{thm:170887}. 
$$
$$

\subsection{Holographic transform in the $L^2$-model on the time-like cone
$\Omega(n)$}
\subsubsection{$L^2$-model of holomorphic discrete series}

For $\lambda>n-1$, 
we denote by $L^2(\Omega)_\lambda\equiv L^2(\Omega,m_\lambda(y)dy)$
 the Hilbert space of square integrable functions
 on the time-like cone $\Omega\equiv \Omega(n)$
 with respect to the measure 
 $m_{\lambda}(y) d y$, 
 where $m_{\lambda}$ is a positive-valued function on $\Omega$
 given by
$$
m_\lambda(y):=Q_{1,n-1}(y)^{\frac n2-\lambda}.  
$$
Let $\langle y,\zeta\rangle=\sum_{j=1}^ny_j\zeta_j$.  
Since the cone $\Omega$ is self-dual in $\R^n$, the Fourier--Laplace transform
$$
\left(\mathcal F_nF\right)(\zeta):=\int_\Omega F(y) e^{i\langle y,\zeta\rangle}dy
$$
is a holomorphic function of $\zeta\in T_\Omega$ if $F\in C_c(\Omega)$.

\begin{fact}
[Faraut--Koranyi {\cite[Thm.~XIII.1.1]{FK}}]
\label{fact:Fnisom}
For $\lambda >n-1$, 
 we set
\begin{equation}\label{eqn:bnlmd}
b_n(\lambda):=\left(2\pi\right)^{\frac{3n}2-1}2^{-2\lambda+n}\Gamma\left(\lambda-\frac n2\right)\Gamma(\lambda-n+1).  
\end{equation}
Then the Fourier--Laplace transform 
$\mathcal F_n \colon C_c(\Omega)
\To \mathcal O(T_\Omega)$
extends to a linear bijection:
\begin{eqnarray}
\label{eqn:Fnisom}
\mathcal F_n \colon L^2(\Omega)_\lambda
{\stackrel{\sim}{\to}} 
\mathcal H^2_\lambda(T_{\Omega(n)}).
\end{eqnarray}
with 
$$
\Vert\mathcal F_n F\Vert^2_{\mathcal H^2_\lambda(T_{\Omega(n)})}=
b_n(\lambda)\Vert F\Vert^2_{L^2(\Omega(n))_\lambda}
\quad
\text{for all $F \in L^2(\Omega)_{\lambda}$.  }
$$
\end{fact}
Via the isomorphism \eqref{eqn:Fnisom}, 
 an irreducible unitary representation of $\widetilde G$
 is defined on $L^2(\Omega)_\lambda$ for $\lambda >n-1$,
 to which we refer as the $L^2$-\emph{model}
 of the (relative) holomorphic discrete series representation $\pi_\lambda^{(n)}$.

Similarly,
 we define a positive-valued function 
 $$
 m_\nu'(y'):=Q_{1,n-2}(y')^{\frac{n-1}2-\nu}
 $$
 on the time-like cone $\Omega'\equiv\Omega(n-1)$ in $\R^{1,n-2}$, and we set
$L^2(\Omega')_\nu:= L^2(\Omega', m_\nu'(y')dy')$
 on which the $L^2$-model
 of the (relative) holomorphic discrete series representation
 $\pi_{\nu}^{(n-1)}$ of $\widetilde G'$ is defined via the unitary map
$$
\mathcal F_{n-1} \colon L^2(\Omega')_\nu
{\stackrel{\sim}{\to}} 
\mathcal H^2_\nu(T_{\Omega(n-1)}) \qquad
\mathrm{for}\quad\nu>n-2
$$
up to rescaling $b_{n-1}(\nu)^{-\frac12}$.

\subsubsection{Gegenbauer polynomial and Juhl's conformally covariant operator}

Let  $a_k(\ell,\alpha)$ be as in \eqref{eqn:alalpha}.
We define a polynomial of two variables by
\begin{equation}\label{eqn:IlCuv}
(I_\ell C_\ell^{\alpha})(u,v):=
\sum_{k=0}^{\left[\frac\ell2\right]}a_k(\ell,\alpha)u^k v^{\ell-2k}.
\end{equation}
For instance, we have
 $(I_0 C_0^{\alpha})(u,v)=1$, 
 $(I_1 C_1^{\alpha})(u,v)=2\alpha v$, 
 and $(I_2 C_2^{\alpha})(u,v)=\alpha(2(\alpha+1)v^2-u)$.  
By definition,
 $(I_\ell C_\ell^{\alpha})(w^2,v)$ is a homogeneous polynomial of degree $\ell$, and
$(I_\ell C_\ell^{\alpha})(1,v)$
coincides with the Gegenbauer polynomial
$C_\ell^{\alpha}(v)$, see \eqref{eqn:gegen} in Appendix. (This is the reason why we adopted
the notation $I_\ell C_\ell^{\alpha}$.)

The F-method \cite[Thm.~6.3]{KP16b} shows
 that the differential operator $\mathcal D_{\ell}^{\alpha}$ 
 (see \eqref{eqn:JuhlonCn}) is expressed as

\begin{equation}\label{eqn:iJuhl}
\mathcal D^{\alpha}_\ell=i^{-\ell}\left(I_\ell C^{\alpha}_\ell\right)
\left(-\Delta_{\C^{1,n-2}}, i\frac\partial{\partial\zeta_n}\right).
\end{equation}


\subsubsection{Construction of discrete summands in the $L^2$-model}

Following the idea of the F-method \cite{K14}, we introduce the \emph{holographic operator}
$\Phi_\lambda^\nu$ as a multiplication operator like in the tensor product case (\emph{cf.}
Definition \ref{def:Phi}). By the simplicity of its formula, the holographic operator
$\Phi_\lambda^\nu$ in the $L^2$-model plays a crucial role in the proof of Theorem \ref{thm:170887}.

Retain the basic setting \eqref{eqn:alphal} where $\ell=\nu-\lambda\in\N$
and $\alpha=\lambda-\frac{n-1}2$. 
For a function $h(y')$ on $\Omega(n-1)$, we define $\left(\Phi_\lambda^\nu h\right)(y)$ on $\Omega(n)$ by
\begin{equation}\label{eqn:hologSO}
\left(\Phi_\lambda^\nu h\right)(y):=Q_{1,n-2}(y')^{-\left(\ell+\frac12\right)}(1-y_n^2)^{\lambda-\frac n2} \left(I_\ell
C^\alpha_\ell\right)(Q_{1,n-2}(y'),-y_n)h(y').
\end{equation}

Then $\Phi_\lambda^\nu$ gives rise to a holographic operator
in the $L^2$-model in the following sense:  
\begin{prop}\label{lem:170772}
Suppose that $\lambda>n-1$ and $\nu=\lambda+\ell$ for some $\ell\in\N$.
\begin{itemize}
\item[(1)] The linear map
$\Phi_{\lambda}^{\nu} \colon L^2(\Omega(n-1))_{\nu}\longrightarrow L^2(\Omega(n))_\lambda$ is an isometry up to scalar:
$$
\Vert 
\Phi_{\lambda}^{\nu}(h)\Vert^2_{L^2(\Omega(n))_\lambda}=c_\ell(\lambda)\Vert h\Vert^2_{L^2(\Omega(n-1))_\nu}\quad\mathrm{for}\,\,\mathrm{all}\,\, h\in L^2(\Omega(n-1))_\nu, 
$$
where we recall that the constant $c_\ell(\lambda)$ is given in \eqref{eqn:c-ell}.  
\item[(2)]
$
\Phi_{\lambda}^{\nu}
$
 intertwines the irreducible unitary representation $\pi_{\nu}^{(n-1)}$ of the subgroup
$\widetilde G'$
 with the restriction $\pi_\lambda^{(n)}\big\vert_{\widetilde G'}$.
\end{itemize}
\end{prop}
We also discuss some further basic properties of the holographic operators
$\Phi_\lambda^\nu$ in Proposition \ref{prop:181248}.

\subsubsection{Proof of Theorem \ref{thm:170887}}\label{subsec:proofTHM32}

Postponing the proof of Proposition \ref{lem:170772} until Section
\ref{sec:357} and Proposition \ref{prop:181248} below until Section \ref{sec:358}
 we complete the proof of Theorem  \ref{thm:170887}.
\begin{prop}\label{prop:normJuhl}
Suppose $\lambda>n-1$ and $\nu=\lambda+\ell$ with $\ell\in\N$. Then the differential operator
$D_{\lambda\to\nu}$ extends to a continuous linear map
$D_{\lambda\to\nu} \colon \mathcal H^2(T_{\Omega(n)})_\lambda\longrightarrow \mathcal H^2(T_{\Omega(n-1)})_\nu$ with the following operator norm:
$$
\Vert D_{\lambda\to\nu}\Vert_{\mathrm{op}}^2=r_\ell(\lambda)c_\ell(\lambda).
$$
\end{prop}
\begin{proof}
It follows from Lemma \ref{lem:opFourier} (1) and  Propositions \ref{lem:170772}
and \ref{prop:181248} that
$$
\Vert D_{\lambda\to\nu}\Vert^2_{\mathrm{op}}=
r_\ell(\lambda) \Vert \widehat D_{\lambda\to\nu}\Vert^2_{\mathrm{op}}
=r_\ell(\lambda)\Vert \Phi_{\lambda}^{\nu}\Vert^2_{\mathrm{op}}
=r_\ell(\lambda)c_\ell(\lambda).
$$
\end{proof}

\begin{proof}[Proof of Theorem \ref{thm:170887}]
By Lemma \ref{lem:Hilbertdeco} and by the (abstract) branching law \eqref{eqn:absbraConf}
for the restriction $\widetilde G\downarrow \widetilde G'$, 
the theorem follows from the expression of the operator norm of the differential operator $D_{\lambda\to\nu}$ given in Proposition \ref{prop:normJuhl}.
\end{proof}

${}$


The rest of this section is devoted to the proof of Proposition \ref{lem:170772}.

\subsubsection{Coordinate change in the $L^2$-model.}

 As in Section \ref{sec:coord-changeSL2} for the tensor product case, we introduce the following coordinates of the time-like cone $\Omega(n)$ in $\R^{1,n-1}$:
\begin{equation}\label{eqn:fibration}
\iota \colon \Omega(n-1)\times (-1,1)\stackrel
{\sim}{\longrightarrow}
\Omega(n), \quad (y',v)\mapsto (y',-\sqrt{Q_{1,n-2}(y')}v),
\end{equation}
which is a bijection because $(y',y_n)\in\Omega(n)$ if and only if 
$y'\in\Omega(n-1)$ and $ y_n^2<Q_{1,n-2}(y')$.

We define a function on $\Omega(n-1) \times (-1,1)$ by
\begin{equation}\label{eqn:Ayv}
M(y',v)\equiv M_{\lambda,\nu}(y',v):= Q_{1,n-2}(y')^{\frac{\ell+1}2}(1-v^2)^{\frac n2-\lambda}.
\end{equation}

Via the isomorphism $\iota$, the ratio of the densities $m_\lambda(y)$ and $m_\nu(y')$, and the holographic
operator $\Phi_\lambda^\nu$ are expressed as follows:
\begin{eqnarray}
\label{eqn:mmAQ-SO}
\frac{m'_\nu(y')}{m_\lambda(y)} &=&M(y',v)^{-1}\;Q_{1,n-2}(y')^{-\frac\ell2},\\
\label{eqn:mmQ}
m_\lambda(y)dy&=&M(y',v)^2\; m'_{\nu}(y')dy'(1-v^2)^{\lambda-\frac n2}dv,\\
\left(\Phi_\lambda^\nu h\right)\circ \iota(y',v)&=&M(y',v)^{-1}\; C_\ell^\alpha(v) h(y'). \label{eqn:holoiota}
\end{eqnarray}

\subsubsection{Fourier transform of the holomorphic Juhl operators}

For $\lambda>n-1$ and $\nu=\lambda+\ell$ ($\ell\in\N$), the holomorphic Juhl operator $D_{\lambda\to\nu}$ gives rise to a continuous operator $\mathcal H^2(T_{\Omega(n)})_\lambda\To 
\mathcal H^2(T_{\Omega(n-1)})_\nu$ between the weighted Bergman spaces \cite[Thm.~5.13]{KP16a}.

 We define a linear map 
$
\widehat D_{\lambda\to\nu}:L^2(\Omega)_\lambda\To L^2(\Omega')_{\nu}
$
 by
\begin{equation}\label{eqn:JFourier}
\widehat D_{\lambda\to\nu}:=\mathcal F_{n-1}^{-1}\circ D_{\lambda\to\nu}\circ\mathcal F_n.
\end{equation}

Then the idea of the F-method \cite{KP16a} implies
 that the ``Fourier transform''
$\widehat D_{\lambda\to\nu}$ of the holomorphic Juhl operator
$D_{\lambda\to\nu}$ is given by a Gegenbauer transform (\emph{cf}. \cite[Chap.~15]{DB})
 along the trajectory in \eqref{eqn:fibration}
 where the parameter $v$  moves:

\begin{prop}\label{lem:170770}
The operator $\widehat D_{\lambda\to\nu}$ is given by the following integral transform:
$$
\left(\widehat D_{\lambda\to\nu} F\right)(y')=
 i^{-\ell} Q_{1,n-2}(y')^{\frac{\ell+1}2}\int_{-1}^1F\circ\iota(y',v) C_\ell^\alpha(v)dv \quad \mathrm{for}
 \quad y'\in\Omega',
$$
where we set $\alpha=\lambda-\frac{n-1}2$.
\end{prop}

\begin{proof}
Let $\alpha=\lambda-\frac{n-1}2$. 
Then it follows from \eqref{eqn:JuhlonCn} and \eqref{eqn:iJuhl} that 
\begin{eqnarray}
\left(i^{\ell}D_{\lambda\to\nu}\circ\mathcal F_n\right)F&=&
\mathrm{Rest}_{\zeta_n=0}\circ \left(I_\ell C_\ell^\alpha\right)
\left( -\Delta_{\C^{1,n-2}},i\frac{\partial}{\partial\zeta_n}\right)\mathcal F_nF
\nonumber\\
&=&\mathrm{Rest}_{\zeta_n=0}\circ\mathcal F_n
\left( I_\ell C_\ell^\alpha(Q_{1,n-2}(y'),-y_n) F\right), \label{eqn:FJuhl}
\end{eqnarray}
for $F\in L^2(\Omega(n))_\lambda$.
Since $(I_\ell C_\ell^\alpha)(u^2,w)=u^\ell C_\ell^\alpha\left(\frac wu\right)$,
 the right-hand side amounts to
\begin{eqnarray*}
&&\int_{\Omega(n-1)}\int_{-1}^1 F\circ\iota(y',v)
Q_{1,n-2}(y')^{\frac{\ell+1}2}
C^\alpha_\ell(v) e^{i\langle y',\zeta'\rangle} dy' dv\\
&=&\mathcal F_{n-1}\left(
Q_{1,n-1}(y')^{\frac{\ell+1}2}\int_{-1}^1F\circ\iota(y',v) C^\alpha_\ell(v)dv\right)
\end{eqnarray*}
via the diffeomorphism \eqref{eqn:fibration}. Since the left-hand side
 of \eqref{eqn:FJuhl} is equal to $i^\ell \mathcal F_{n-1}\circ\widehat D_{\lambda\to\nu}F$
 by the definition \eqref{eqn:JFourier} of $\widehat D_{\lambda\to\nu}$, we proved Proposition
 \ref{lem:170770}.
\end{proof}

\subsubsection{Proof of Proposition \ref{lem:170772}}\label{sec:357}

\begin{proof}[Proof of Proposition \ref{lem:170772}]
Let $\alpha=\lambda-\frac{n-1}2$ as before.
By \eqref{eqn:mmQ} and \eqref{eqn:holoiota} we have 
$$
\Vert\Phi_{\lambda}^{\nu} h\Vert^2_{L^2(\Omega(n))_\lambda}=\Vert h\Vert^2_{L^2(\Omega(n-1))_{\lambda+\ell}}
\int_{-1}^1\left\vert C^{\alpha}_\ell(v)\right\vert^2 (1-v^2)^{\alpha-\frac 12} dv.
$$
Thus the first assertion holds by the definition \eqref{eqn:c-ell} of $c_\ell(\lambda)$.
 
The second assertion follows readily from Proposition \ref{prop:181248}
 below.
\end{proof}

\subsubsection{Holographic operators and the adjoint of symmetry breaking operators}
\label{sec:358}

We have constructed operators $\Phi_\lambda^\nu$ in the
$L^2$-model and $\widehat\Phi_\lambda^\nu$ in the holomorphic model using the F-method. On the other hand, the adjoint of symmetry
 breaking operators between unitary representations yield holographic operators in general (\emph{cf}. Theorem \ref{thm:SBTinj} (1) below).
In our setting, these operators are proportional to each other because the branching law \eqref{eqn:absbraConf} is multiplicity-free. We determine the proportionality constants:
\begin{prop}\label{prop:181248}
Suppose $\lambda> n-1$ and $\ell=\nu-\lambda\in\N$. Then we have
\begin{align*}
\left(\widehat D_{\lambda\to\nu}\right)^* &=i^\ell \Phi_{\lambda}^{\nu}, 
\\
D_{\lambda\to\nu}^* &=i^\ell r_\ell(\lambda) \widehat{\Phi_\lambda^\nu}.
\end{align*}
\end{prop}
\begin{proof}
Take  any $h\in L^2(\Omega(n-1))_\nu$ and $F\in L^2(\Omega(n))_\lambda$ with $\nu-\lambda=\ell\in\N$. 
Since $\alpha=\lambda-\frac{n-1}2$ is real, $C_\ell^\alpha(v)$ is real-valued for $-1< v< 1$, hence we have from Proposition
\ref{lem:170770}
\begin{eqnarray*}
&&\left( h,\widehat D_{\lambda\to\nu}F\right)_{L^2(\Omega(n-1))_\nu}\\&=&i^\ell
\int_{\Omega(n-1)}h(y')\left(\int_{-1}^1 Q_{1,n-2}(y')^{\frac{\ell+1}2}\overline{F\circ\iota(y',v )}
 C_\ell^\alpha(v)dv\right) m'_\nu(y')dy'\\
&=& i^\ell \int_{\Omega(n)}h(y')\overline{F(y)}(I_\ell C_\ell^\alpha)(Q_{1,n-2}(y'),-y_n)m'_\nu(y')dy\\
&=& i^\ell \int_{\Omega(n)}\left(\Phi_{\lambda}^{\nu}h\right)(y)\overline{F(y)} m_\lambda(y)dy\\
&=& i^\ell\left(\Phi_{\lambda}^{\nu}h,F\right)_{L^2(\Omega(n))_\lambda}.
\end{eqnarray*}
Here we have used \eqref{eqn:mmAQ-SO} and \eqref{eqn:holoiota} in the third identity.
Hence the first equality is shown. 
By Lemma \ref{lem:opFourier}, 
 the second equality follows from the first one.
\end{proof}

\subsection{Explicit integral formula for the holographic operator}

In this section, we give an integral formula for the holographic operator $D_{\lambda\to\nu}^*$
in the holomorphic model, giving thus an answer to Problem A.1 in Section 1, see Theorem 
\ref{thm:CIRM18} below.

\subsubsection{Construction of discrete summands in the holomorphic model}

In contrast to the holographic operators $\Psi_{\lambda',\lambda''}^{\lambda'''}$ (Definition \ref{def:HOtensor}) in the tensor product case in Section \ref{sec:RCT}, we do not have a simple integral
expression for a holographic operator like \eqref{eqn:psi} in the present setting. 
Instead, we adopt an alternative approach to construct a holographic operator
 by introducing a {\it{relative reproducing kernel}} $K_{\lambda,\nu}(\zeta,\tau')$ as below.

For $\lambda,\nu\in\C$ with $\ell:=\nu-\lambda\in\N$, we set
\begin{equation}\label{eqn:holoKlmdnu}
K_{\lambda,\nu}(\zeta,\tau'):=\frac{
\left((\zeta_1-\overline \tau_1)^2
-(\zeta_2-\overline \tau_2)^2-\cdots - (\zeta_{n-1}-\overline \tau_{n-1})^2-\zeta_n^2\right)^{-\nu}}
{\zeta_n^{\lambda-\nu}},
\end{equation}
where $\zeta=(\zeta_1,\cdots,\zeta_n)\in T_{\Omega(n)}$ and $\tau'=(\tau_1,\cdots,\tau_{n-1})\in T_{\Omega(n-1)}$.

\begin{rem}
$K_{\lambda,\nu}(\zeta,\tau')$ may be viewed as the holomorphic counterpart of the distribution kernel  of a conformally covariant \emph{integral} symmetry breaking operator for the pair of Riemannian
manifolds $(S^n, S^{n-1})$, see \cite{KS}. See also \eqref{eqn:Klmdnu} for the case $n=2$.
\end{rem}

Let $d\mu_\nu$ be a measure on $T_{\Omega(n-1)}$ given
by
$$d\mu_\nu(\tau'):=\left(\frac i2\right)^{n-1}Q_{1,n-2}(\mathrm{Im}\; \tau')^{\nu-n+1}d\tau'd\overline \tau'.$$

\begin{thm}[holographic operator]\label{thm:CIRM18}
Let $n\geq3$.
 Suppose $\lambda>n-1$ and
 $\nu=\lambda+\ell$ with $\ell\in\N$. 
 Then the integral
 $$
 \int_{T_{\Omega(n-1)}} K_{\lambda,\nu} (\zeta,\tau') g(\tau') d\mu_\nu(\tau')
 $$
 converges for all
  $g\in \mathcal H^2(T_{\Omega(n-1)})_\nu$ and $\zeta\in T_{\Omega(n)}$.
  Moreover, it gives the adjoint  operator $D^*_{\lambda\to\nu}$ up to scalar multiplication:
\begin{eqnarray*}
\left(D^*_{\lambda\to\nu} g\right)(\zeta)=C\int_{T_{\Omega(n-1)}} K_{\lambda,\nu}(\zeta,\tau') g(\tau') d\mu_\nu(\tau'),
\end{eqnarray*}
where the constant $C$ is given by
\begin{equation}\label{eqn:CIRM18}
C= 
\frac{2^{2\lambda-2n+\ell-1}(\lambda-n+1)_{n+\ell-1}(2\lambda-n)_{\ell+1}}{i^{2\lambda+2\ell} \pi^n \ell !}.
\end{equation}
In particular, it
 yields an injective continuous $\widetilde G'$-intertwining operator between weighted Bergman spaces,
$
 \mathcal H^2(T_{\Omega(n-1)})_\nu
\hookrightarrow \mathcal H^2(T_{\Omega(n)})_\lambda.
$
\end{thm}

We first show that Theorem \ref{thm:CIRM18} is derived from the following
Bernstein--Sato type
identity for the holomorphic Juhl operator.

\begin{thm}\label{thm:Cassis}
Let $\mathcal D_\ell^\alpha$ be the differential operator as in \eqref{eqn:iJuhl}.  
We set
\begin{eqnarray}\label{eqn:qnl}
q(n,\ell;\lambda):=\frac{2^{\ell}}{\ell!}\,(2\lambda-n+1)_\ell(\lambda)_\ell.
\end{eqnarray}
Then,
$$
\zeta_n^{-\ell}\,\mathcal D_{\ell}^{\lambda-\frac{n-1}2} Q_{1,n-1}(\zeta)^{-\lambda}= q(n,\ell;\lambda)\, Q_{1,n-1}(\zeta)^{-\lambda-\ell}.
$$
\end{thm}

\begin{rem}\label{rem:Cassis}
Theorem \ref{thm:Cassis} shows that the complex power of the quadratic form 
$Q_{1,n-1}$ satisfies the Bernstein--Sato type identity not only for the power of the Laplacian (see \eqref{eqn:LapQpower} below) but also for another operator closely
related to the holomorphic Juhl operator.
\end{rem}
Postponing the proof of Theorem \ref{thm:Cassis}, we complete the proof of Theorem \ref{thm:CIRM18}. For this, we also
 use the following lemma.
\begin{lem}\label{lem:adker}
Let $D_j$ ($j=1,2$) be complex manifolds, and $\mathcal H_j$ Hilbert spaces contained in $\mathcal O(D_j)$ with
reproducing kernels $K^{(j)}(\cdot,\cdot)$.
Suppose that $\mathcal R \colon \mathcal H_1\to\mathcal H_2$ is a continuous linear map, and
$\mathcal R^* \colon \mathcal H_2\to\mathcal H_1$ is its adjoint operator.
Then,
\begin{itemize}
\item[(1)] $\overline{\mathcal R K^{(1)}(\,\cdot, \zeta)(\tau')}=\left(\mathcal R^*K^{(2)}(\,\cdot,\tau')\right)(\zeta)$ for $\zeta\in D_1, \tau'\in D_2$;
\item[(2)] $\left(\mathcal R^*g\right)(\zeta)=\left( g, \mathcal R K^{(1)}(\,\cdot, \zeta)\right)_{\mathcal H_2}$ for $g\in\mathcal H_2$ and $\zeta\in D_1$.
\end{itemize}
\end{lem}
\begin{proof}
(1)\enspace The first assertion results from the reproducing property of $K^{(j)}(\cdot,\cdot)$ applied to the following identity:
$$
(\mathcal R^* K^{(2)}(\,\cdot,\tau'), K^{(1)}(\,\cdot,\zeta))_{\mathcal H_1}
=( K^{(2)}(\,\cdot,\tau'),\mathcal R K^{(1)}(\,\cdot,\zeta))_{\mathcal H_2}.
$$
\par\noindent
(2)\enspace  The statement is immediate from the following: 
$$\left(\mathcal R^*g\right)(\zeta)= \left(\mathcal R^*g, K^{(1)}(\,\cdot,\zeta)\right)_{\mathcal H_1}= (g,\mathcal R K^{(1)}(\,\cdot,\zeta))_{\mathcal H_2}.
$$
\end{proof}

\begin{proof}[Proof of Theorem \ref{thm:CIRM18}]
Applying Lemma \ref{lem:adker} to the triple
$$
(\mathcal R, \mathcal H_1,\mathcal H_2)=( D_{\lambda\to\nu}, \mathcal H^2(T_{\Omega(n)})_\lambda, \mathcal H^2(T_{\Omega(n-1)})_\nu),
$$
we obtain the following integral expression of the adjoint operator $D_{\lambda\to\nu}^*$:
\begin{equation}\label{eqn:DKint}
\left(D_{\lambda\to\nu}^*g\right)(\zeta)=
\int_{T_{\Omega(n-1)}} g(\tau')\overline{D_{\lambda\to\nu} K_\lambda(\,\cdot,\zeta)(\tau')}
d\mu_\nu(\tau').
\end{equation}
Here, we have viewed the reproducing kernel $K_\lambda(\tau,\zeta)=k_{\lambda,n} Q_{1,n-1}(\tau-\overline{\zeta})^{-\lambda}$ defined in \eqref{eqn:Klmbd} as a function of 
$\tau\in T_{\Omega(n)}$ with parameter $\zeta\in T_{\Omega(n)}$ and applied
the holomorphic Juhl operator $D_{\lambda\to\nu}$. 
Writing $\tau$ as $\tau =(\tau', \tau_n)$, we get from Theorem \ref{thm:Cassis}:
\begin{eqnarray*}
&&D_{\lambda\to\nu} K_\lambda(\tau,\zeta)\\
&=&
k_{\lambda,n}\, q(n,\ell;\lambda)\,
\mathrm{Rest}_{\tau_n=0}\circ
(\tau_n-\overline\zeta_n)^\ell
Q_{1,n-1}(\tau-\overline{\zeta})^{-\lambda-\ell}\\
&=& 
(-1)^\ell{k_{\lambda,n}} q(n,\ell;\lambda) \overline{K_{\lambda,\nu}(\zeta, \tau')}
 \end{eqnarray*}
for $\tau'\in T_{\Omega(n-1)},
$
by the definition \eqref{eqn:holoKlmdnu} of the relative reproducing kernel $K_{\lambda,\nu}(\zeta,\tau')$. Since $q(n,\ell;\lambda)\in\R$ when $\lambda\in\R$, the integral formula of
Theorem 
\ref{thm:CIRM18} is shown with the constant $C=(-1)^\ell\overline{k_{\lambda,n}} q(n,\ell;\lambda)$.
A short computation shows the formula  \eqref{eqn:CIRM18}.
\end{proof}
The rest of this section is devoted to the proof of Theorem \ref{thm:Cassis}.

\subsubsection{Proof of Theorem \ref{thm:Cassis}}

\begin{lem}\label{lem:cirm180318}
Suppose $\ell\in\N$.
Then there exist  $q_j\equiv q_j(n,\ell;\lambda)$ ($0\leq 2j\leq\ell)$ such that
$$
\mathcal D_\ell^\alpha Q_{1,n-1}(\zeta)^{-\lambda}=\sum_{j=0}^{\left[\frac\ell2\right]}
q_j\zeta_n^{\ell-2j} Q_{1,n-1}(\zeta)^{-\lambda-\ell+j}.
$$
\end{lem}

\begin{proof}
It is easy to see that an analogous statement holds for $\left(\frac{\partial}{\partial\zeta_n}\right)^\ell$ instead of $\mathcal D_\ell^\alpha$, namely, there
exist $q_j'\equiv q_j'(n,\ell;\lambda)$ ($0\leq 2j\leq\ell$) such that
\begin{equation}\label{eqn:derQ}
\left(\frac{\partial}{\partial\zeta_n}\right)^\ell
Q_{1,n-1}(\zeta)^{-\lambda}=\sum_{j=0}^{\left[\frac\ell2\right]}
q_j'\zeta_n^{\ell-2j} Q_{1,n-1}(\zeta)^{-\lambda-\ell+j}.
\end{equation}

We rewrite $\mathcal D_\ell^\alpha$ as a polynomial of $\Delta_{\C^{1,n-1}}=\frac{\partial^2}{\partial\zeta_1^2}-\frac{\partial^2}{\partial\zeta_2^2}-\cdots-\frac{\partial^2}{\partial\zeta_n^2}$
 and $\frac{\partial}{\partial\zeta_n}$ by substituting $\Delta_{\C^{1,n-2}}=\Delta_{\C^{1,n-1}}+\frac{\partial^2}{\partial\zeta_n^2}$ into
\eqref{eqn:JuhlonCn}:
\begin{equation}\label{eqn:Dshiftp}
\mathcal D_\ell^\alpha =
\sum_{k=0}^{\left[\frac\ell2\right]} p_k(n,\ell;\alpha)\left(\frac{\partial}{\partial\zeta_n}\right)^{\ell-2k}
\left(\Delta_{\C^{1,n-1}}\right)^k,
\end{equation}
where the first coefficient is given by
\begin{equation}\label{eqn:p0}
p_0(n,\ell;\alpha)=\sum_{k=0}^{\left[\frac\ell2\right]}a_k(\ell,\alpha)=
 C_\ell^\alpha(1)=\frac{(2\alpha)_\ell}{\ell!}.
\end{equation}

An iterated use of the formula
$$
\Delta_{\C^{1,n-1}} Q_{1,n-1}(\zeta)^{-\lambda}=2\lambda(2\lambda-n+2)Q_{1,n-1}(\zeta)^{-\lambda-1},
$$
leads us to
\begin{equation}\label{eqn:LapQpower}
\left(\Delta_{\C^{1,n-1}}\right)^kQ_{1,n-1}(\zeta)^{-\lambda}=s_k(n,\lambda) Q_{1,n-1}(\zeta)^{-\lambda-k},
\end{equation}
for some polynomials $s_k(n,\lambda)$ of $\lambda$ of degree $2k$.
We note that $s_0(n,\lambda)=1$.
 Now the lemma follows from \eqref{eqn:derQ}.
\end{proof}

Clearly the coefficients
$q_j=q_j(n,\ell; \lambda)$ in Lemma \ref{lem:cirm180318}
are unique.
The proof of Theorem \ref{thm:Cassis} is reduced to the following proposition
on these coefficients $q_j(n,\ell; \lambda)$.

\begin{prop}\label{prop:DQtop} 
${}$
\begin{itemize}
\item[(1)] 
{\rm{(the first term)}}. Recall that
$q(n,\ell;\lambda)$ is defined in \eqref{eqn:qnl}. Then,
$$
q_0\left(n,\ell;\lambda\right)=q(n,\ell;\lambda).
$$
\item[(2)] 
{\rm{(vanishing of higher terms)}}.
$$
q_j\left(n,\ell;\lambda\right)=0\quad\mathrm{for}\,\mathrm{all}\,\, j\geq1.
$$
\end{itemize}
\end{prop}

\begin{proof}[Proof of Proposition \ref{prop:DQtop} (1)]
In the expression 
$$
\left(\frac{\partial}{\partial\zeta_n}\right)^{\ell-2k} \left(\Delta_{\C^{1,n-1}}\right)^k 
Q_{1,n-1}(\zeta)^{-\lambda}=s_k(n,\lambda)\left(\frac{\partial}{\partial\zeta_n}\right)^{\ell-2k} 
Q_{1,n-1}(\zeta)^{-\lambda-k},
$$
the term $\zeta_n^\ell \;Q_{1, n-1}(\zeta)^{-\lambda-\ell}$ occurs 
only when $k=0$, and its coefficient is given by
$
s_0(n,\lambda) 2^\ell (\lambda)_\ell=2^\ell(\lambda)_\ell.
$
By \eqref{eqn:Dshiftp}, we get
$$
q_0(n,\ell;\lambda)=p_0(n,\ell;\alpha)\cdot 2^\ell(\lambda)_\ell.
$$
Now the first assertion of Proposition \ref{prop:DQtop}
follows from \eqref{eqn:p0}.
\end{proof}

In order to prove the second assertion of Proposition \ref{prop:DQtop}, we discuss the kernel of the holomorphic Juhl operator
 $D_{\lambda\to\nu} \colon \mathcal O(T_{\Omega(n)})\To
\mathcal O(T_{\Omega(n-1)})$.

\begin{prop}\label{prop:KerJuhl}
Suppose $\lambda-\frac{n-1}2\not\in\{0,-1,-2,\cdots\}$. Then for any $N\in\N$ we have
\begin{multline}
\bigcap_{j=0}^N\mathrm{Ker}\left(D_{\lambda\to\lambda+j}\right)\\=
\left\{f\in\mathcal O(T_{\Omega(n)})\colon \mathrm{Rest}_{\zeta_n=0}\circ \left(\frac{\partial}{\partial\zeta_n}\right)^j f=0\,\,\mathrm{for}\,\mathrm{all}\,0\leq j\leq N\right\}. \nonumber
\end{multline}
\end{prop}

\begin{proof}
By the definition \eqref{eqn:holoJuhl} of $D_{\lambda\to\nu}$,
the right-hand side is clearly contained in the left-hand side. To see the opposite inclusion,
we recall from the definition \eqref{eqn:holoJuhl} that the symmetry breaking operator $D_{\lambda\to\lambda+j}$
is of the form
$$
D_{\lambda\to\lambda+j}=\mathrm{Rest}_{\zeta_n=0}\circ\left(a_0 \left(\frac{\partial}{\partial\zeta_n}\right)^j+
\sum_{k=1}^{\left[\frac j2\right]} a_k
\left(\frac{\partial}{\partial\zeta_n}\right)^{j-2k}
\Delta^k_{\C^{1,n-2}}\right),
$$
where the first coefficient $a_0\equiv a_0\left(j,\alpha\right)$ is given by
$$
 a_0\left(j,\alpha\right)=\frac{2^j}{j!}(\alpha)_\ell \qquad \mathrm{with}\quad\alpha=\lambda-\frac{n-1}2.
$$

We now prove the proposition by induction on $N$. The statement is clear for $N=0$ and $1$ because
$D_{\lambda\to\lambda}= \mathrm{Rest}_{\zeta_n=0}$ and $D_{\lambda\to\lambda+1}= \mathrm{Rest}_{\zeta_n=0}\circ \frac{\partial}{\partial\zeta_n}$.

Suppose that $f\in\mathcal O(T_{\Omega(n)})$
satisfies $D_{\lambda\to\lambda+j}f=0$ for $0\leq j\leq N+1$. By the inductive assumption, we get
$\mathrm{Rest}_{\zeta_n=0}\circ \left(\frac{\partial}{\partial\zeta_n}\right)^jf=0$ ($0\leq j\leq N)$. Since $a_0\equiv a_0(j,\alpha)$ is nonzero for any $j$ by the assumption on $\lambda$, $D_{\lambda\to\lambda+N+1}f=0$ implies  
$\mathrm{Rest}_{\zeta_n=0}\circ \left(\frac{\partial}{\partial\zeta_n}\right)^{N+1}f=0$.
Thus the proposition is proved by induction.
\end{proof}

\begin{proof}[Proof of Proposition \ref{prop:DQtop} (2)]
Let $\nu=\lambda+\ell$. By \eqref{eqn:DKint} and Lemma \ref{lem:cirm180318}, we obtain
\begin{eqnarray}\label{eqn:Dstarsum}
&&\left( D_{\lambda\to\nu}^* g\right)(\zeta)
\\&=&(-1)^\ell \overline{k_{\lambda,n}}\sum_{j=0}^{\left[\frac\ell2\right]}q_j\zeta_n^{\ell-2j}
\int_{T_{\Omega(n-1)}}
g(\tau') Q_{1,n-1}(\zeta-\overline{\tau})^{-\lambda-\ell+j}d\mu_\nu(\tau'), \nonumber
\end{eqnarray}
where we write $\tau = (\tau',0)= (\tau_1',\cdots, \tau_{n-1}',0)$ by abuse of notations.

On the other hand, the composition map
$$
D_{\lambda\to\lambda+j}\circ D_{\lambda\to\lambda+\ell}^* \colon 
\mathcal H^2(T_{\Omega(n-1)})_{\lambda+\ell}\longrightarrow \mathcal H^2(T_{\Omega(n-1)})_{\lambda+j}
$$
is a $G'$-intertwining operator for any $j$. 
Since the $G'$-modules $\mathcal H^2(T_{\Omega(n-1)})_{\lambda+j}$ ($j\in\N$) are irreducible and mutually inequivalent if $\lambda>n-1$, such an intertwining operator
must be zero unless $j=\ell$. Therefore
$$
\mathrm{Image}\left(D_{\lambda\to\lambda+\ell}^*\right)\subset
\bigcap_{j=0}^{\ell-1}\mathrm{Ker}\left(D_{\lambda\to\lambda+j}\right).
$$
We now prove that $q_j\equiv q_j(n,\ell;\lambda)$ vanishes for all $1\leq j\leq\left[\frac\ell2\right]$
by downward induction on $j$. For simplicity, we treat the case where $\ell$ is even, say $\ell=2m$. The case where $\ell$ is odd can be dealt with similarly.

By Proposition \ref{prop:KerJuhl}, we have
$$
\mathrm{Rest}_{\zeta_n=0}\circ D_{\lambda\to\nu}^*g=0 \quad
\text{for all}\quad g\in
\mathcal H^2(T_{\Omega(n-1)})_\nu. 
$$
Then it follows from \eqref{eqn:Dstarsum} that
$$
 (-1)^\ell\, \overline{k_{\lambda,n}} q_m
\int_{T_{\Omega(n-1)}}
g(\tau') Q_{1,n-2}(\zeta'-\overline{\tau})^{-\lambda-m}d\mu_\nu(\tau')=0.
$$
Thus we conclude that $q_m=0$ because
$k_{\lambda,n}\neq0$.

Suppose that we have shown $q_j=0$ for $j=m, m-1,\cdots, m+1-s$ for some $s\geq1$.
 If $s\leq m-1$,
then $2s\leq\ell-1 (=2m-1)$ and we can proceed by 
applying Proposition \ref{prop:KerJuhl} with $N=\ell-1$, hence
$$
\mathrm{Rest}_{\zeta_n=0}\circ \left(\frac{\partial}{\partial\zeta_n}\right)^{2s}
\circ D_{\lambda\to\nu}^*g=0.
$$
By the inductive assumption, we obtain
$$
 (-1)^\ell \overline{k_{\lambda,n}} q_{m-s}
\int_{T_{\Omega(n-1)}}
g(\tau') Q_{1,n-2}(\zeta'-\overline{\tau})^{-\lambda-m-s}d\mu_\nu(\tau')=0
$$
for all $g\in
\mathcal H^2(T_{\Omega(n-1)})_\nu$. Thus we conclude that $q_{m-s}=0$ 
as far as $s\leq m-1$. Hence
 we have shown $q_j=0$ for all $j\geq1$.
 \end{proof}
 Thus the proof of Theorem \ref{thm:Cassis} (hence, also the one of Theorem \ref{thm:CIRM18}) is completed.

\section{Perspectives of symmetry breaking and holographic transforms}
\label{sec:conclusion}

We end this article with discussion
 on a representation-theoretic background of Problems A and B
 in a broader framework.

In Section \ref{sec:concl1},
 we consider these problems from the viewpoint of the branching laws of unitary representations of locally compact groups.  
In Section \ref{sec:concl2},
 we investigate Problems A and B for triples $(G,G',\pi)$
 such that

\begin{itemize}
\item[] $\bullet\,(G,G')$ is a reductive symmetric pair of holomorphic type;
\item[] $\bullet\,\pi$ is a unitary highest weight module
 of $G$ of scalar type, 
\end{itemize}
generalizing the settings for the main results in Sections \ref{sec:RCT} and \ref{sec:HJT}.  
The role of special orthogonal polynomials in these cases
 is clarified in Section  \ref{sec:concl3}.

\subsection{Branching laws, symmetry breaking transform and holographic transform}\label{sec:concl1}

Let $G\supset G'$ be a pair of groups,
 $\pi$ an irreducible $G$-module,
 and $\rho$ an irreducible $G'$-module.  
We recall from Section \ref{sec:Intro}
 that an element in $\mathrm{Hom}_{G'}\left(\pi\vert_{G'},\rho\right)$
 (resp. in
$\mathrm{Hom}_{G'}\left(\rho,\pi\vert_{G'}\right)$) is said to be a symmetry breaking operator (resp. a holographic operator).
We also recall that a symmetry breaking transform (resp. a holographic transform) is a collection of symmetry breaking operators (resp. holographic operators)
 where $(\rho,W)$ runs over a certain set $\Lambda$ of irreducible representations of the subgroup $G'$.

If $\pi$ is a unitary representation of a locally compact group $G$
 on a Hilbert space $V$, 
 then Mautner's theorem guarantees that the restriction ($\pi\vert_{G'},V$)
 is unitarily equivalent to the direct integral of irreducible unitary representations of the subgroup $G'$:
\begin{equation}\label{eqn:BL}
\pi\vert_{G'}\simeq\int_{\widehat{G'}}^\oplus m_\pi(\rho)\rho\, d\mu(\rho),
\end{equation}
where $\widehat{G'}$ is the set of equivalence classes of irreducible unitary representations of $G'$
(\emph{unitary dual}), $\mu$ is a Borel measure on $\widehat{G'}$ endowed with the Fell topology, and $m_\pi \colon \widehat{G'}\To \N\cup\{\infty\}$ is a measurable function (\emph{multiplicity}). 
The irreducible decomposition \eqref{eqn:BL} is called branching law
 of the restriction $\pi\vert_{G'}$, which is unique up to isomorphism if $G'$ is a type $I$ group, in particular, if $G'$ is a real reductive group by a theorem of Harish-Chandra \cite{HC}. 

The (abstract) branching law \eqref{eqn:BL}
 would be enriched through Problems A and \ref{prob:B}
 by geometric realizations of irreducible representations
 and explicit intertwining operators:
\begin{align*}
&&\text{from (LHS) to (RHS)}  \qquad   & \text\it{symmetry\,\, breaking\,\, transform};\\
&&\text{from (RHS) to (LHS)}  \qquad  & \text{\it{holographic\, transform}}.  
\end{align*}

In the unitary case, it is natural to take $\Lambda$ to be the support of the measure $\mu$ in \eqref{eqn:BL}. If $\Lambda$ is a countable set, then the branching law \eqref{eqn:BL} is \emph{discretely decomposable} 
 without continuous spectrum. 
A criterion for the triple $(G,G',\pi)$ to admit a discretely decomposable restriction $\pi\vert_{G'}$ was studied in \cite{kdeco98,kdeco98Inv} when $G\supset G'$ is a pair of real reductive groups.

On the other hand, 
 the multiplicity $m_{\pi}(\rho)$ in \eqref{eqn:BL} is not always finite
 when $\pi$ and $\rho$ are infinite-dimensional.  
A geometric criterion for the pair $(G,G')$ 
 to assure that $\operatorname{Hom}_{G'}(\pi^{\infty}|_{G'}, \rho^{\infty})$
 is finite-dimensional
 for all smooth irreducible $G$-modules $\pi^{\infty}$ and $G'$-modules $\rho^{\infty}$ was established 
 in \cite{xKOfm}.  

\vskip 1pc

If the branching law \eqref{eqn:BL} is discretely decomposable and multiplicity free, 
 then we could expect a simple and detailed study of symmetry breaking transform
 and holographic transform.  
In this case, 
 since the vector space $\mathrm{Hom}_{G'}(\pi\vert_{G'},\rho)$
 is one-dimensional, 
 symmetry breaking operator is unique up to scaling
 for every $\rho$, 
 and the symmetry breaking transform is defined
 as the collection of countably many such operators.

\subsection{Symmetric pairs of holomorphic type}\label{sec:concl2}

In this section, we provide a geometric condition (see Setting \ref{setting:takagi18} below) that assures the branching law $\pi\vert_{G'}$ to be discretely decomposable and multiplicity free. 
In this case, we see that every
symmetry breaking operator is a
 differential operator (\emph{e.g.}  the Rankin--Cohen transforms studied in Section \ref{sec:RCT} and the holomorphic Juhl transforms in Section \ref{sec:HJT}), 
 and that our symmetry breaking transform $D$
is injective, hence giving an affirmative 
answer to Problem A.0 in Section \ref{sec:Intro}.
The main results in Sections \ref{sec:RCT} and \ref{sec:HJT}
 are built on special cases of this general setting.

Let us fix some notations. Let $G$ be a connected reductive Lie group, $\theta$ a Cartan
involution, $K=\{g\in G : \theta g=g\}$, $\mathfrak g=\mathfrak k+\mathfrak p$ the corresponding Cartan decomposition, and $\mathfrak g_\C=\mathfrak k_\C+\mathfrak p_\C$ its complexification.
Assume that there exists a central element $Z$ of $\mathfrak k_\C$ such that
$$
\mathfrak g_\C=\mathfrak k_\C+\mathfrak n_++\mathfrak n_-
$$
is the eigenspace decomposition of $\operatorname{ad}(Z)$
 with eigenvalues $0,1$, and $-1$, respectively.
This assumption is satisfied
 if and only if $G$ is locally isomorphic to a direct product of compact Lie groups 
 (with $Z=0$) and noncompact Lie groups of Hermitian type. 
Then the associated Riemannian symmetric space $X=G/K$
 becomes a Hermitian symmetric space 
 with complex structure induced from the Borel embedding $G/K\subset G_\C/K_\C\exp(\mathfrak n_+)$. 
Take a Cartan subalgebra $\mathfrak t$ of $\mathfrak k$, and write 
$\rho(\mathfrak n_+)$ for half the sum of roots in $\Delta(\mathfrak n_+,
\mathfrak t_\C)$.

\begin{sett}\label{setting:takagi18} 
Let $(G,G')$ be a reductive symmetric pair of holomorphic type,
 that is,
 $X=G/K$ and $Y=G'/K'$ are both Hermitian symmetric spaces and the natural embedding $\iota \colon Y\hookrightarrow X$ is holomorphic. 
Let $\mathcal L=G\times_K\C_{\lambda}$ be
 a $G$-equivariant holomorphic line bundle over $X$
 associated to a unitary character $\C_\lambda$ of $K$, 
 and we set 
$\mathcal H^2(X,\mathcal L):=\left(\mathcal O\cap L^2\right)(X,\mathcal L)$.  
Assume $\lambda$ satisfies the following condition:
\begin{equation}\label{eqn:lmdpos}
\begin{cases}
\langle \lambda,\alpha\rangle=0
 \quad &\forall\alpha\in\Delta
(\mathfrak k_\C,\mathfrak t_\C),
\\
\langle \lambda-\rho(\mathfrak n_+),\alpha\rangle>0 \quad &\forall\alpha\in\Delta(\mathfrak n_+,\mathfrak t_\C).  
\end{cases}
\end{equation}
\end{sett}

 The Hilbert space $\mathcal H^2(X,\mathcal L)$ is naturally identified with a weighted Bergman space, which is nonzero if $\lambda$ satisfies the
 condition \eqref{eqn:lmdpos}.  
We denote by
$\pi$ the representation of $G$ on the Hilbert space $\mathcal H^2(X,\mathcal L)$, which is irreducible and unitary, and is
 called a {\it{holomorphic discrete series representation}} of $G$.  
The list of irreducible symmetric pairs $(G,G')$ 
 of holomorphic type 
 may be found in \cite[Table 3.4.1]{K08}.  
\begin{fact}
[{see \cite[Thm.~B]{K08}}]
In Setting \ref{setting:takagi18}, the restriction $\pi\vert_{G'}$ is discretely decomposable and multiplicity free.
\end{fact}

Any irreducible $G'$-module
 that occurs in the branching law \eqref{eqn:BL}
 for the unitary representation $(\pi,{\mathcal{H}}^2(X,{\mathcal{L}}))$
 is of the form 
 $\mathcal H^2(Y,\mathcal W)$
 for some $G'$-equivariant holomorphic vector bundle $\mathcal W$ over $Y$
 associated to an irreducible finite-dimensional unitary representation $W$
 of $K'$, 
 and such bundles ${\mathcal{W}}$ are classified.
 Thus $\Lambda$ is parametrized by a subset of $\widehat{K'}$, or by a 
 subset of dominant integral weights which can be
  described in terms of the root data (see \cite[Thm.~8.3]{K08}). 
 We write $\rho_\ell$ for the irreducible unitary representation
 of $G'$ corresponding to $\ell\in\Lambda$, and identify $\Lambda$ as a subset of $\widehat{G'}$
  by $\ell\mapsto\rho_\ell$.  
Here is a summary on general results
 about symmetry breaking operators
 in this setting:
\begin{fact}
\label{fact:4.3}
In Setting \ref{setting:takagi18}, let $\mathcal W$ be the $G'$-equivariant holomorphic vector bundle corresponding to $\ell\in\Lambda$.
\begin{itemize}
\item[(1)] Any continuous $G'$-homomorphism $\mathcal O(X,\mathcal L)\To\mathcal O(Y,\mathcal W)$
is given as a holomorphic differential operator, and induces a continuous $G'$-homomorphism between
the Hilbert spaces $\mathcal H^2(X,\mathcal L)\To\mathcal H^2(Y,\mathcal W)$.
\item[(2)] Any continuous $G'$-homomorphism $\mathcal H^2(X,\mathcal L)\To \mathcal H^2(Y,\mathcal W)$ extends to a continuous $G'$-homomorphism $\mathcal O(X,\mathcal L)\To\mathcal O(Y,\mathcal W)$.
\end{itemize}
\end{fact}

\begin{proof}
(1)\enspace The first statement is proved in \cite[Thm.~5.3]{KP16a} (\emph{localness theorem}), and the second one is in \cite[Thm.~5.13]{KP16a}.
\par\noindent
(2)\enspace
By (1) there is a natural injective map
\begin{equation}\label{eqn:OH2}
   \operatorname{Hom}_{G'}(\mathcal O(X,\mathcal L)\big\vert_{G'}, \mathcal O(Y,\mathcal W)) \hookrightarrow 
   \operatorname{Hom}_{G'}(\mathcal H^2(X,\mathcal L)\big\vert_{G'}, \mathcal H^2(Y,\mathcal W)).
\end{equation}
To prove that \eqref{eqn:OH2} is surjective, we observe that the left-hand side of
\eqref{eqn:OH2} is understood by the branching law
 for the generalized Verma module
 \cite[Thm.~5.2]{K12} via the duality theorem \cite[Thm.~A]{KP16a}, 
 whereas the right-hand side of \eqref{eqn:OH2} is given by
the
 branching law
 of the unitary representation $\mathcal H^2(X,\mathcal L)|_{G'}$
 (\cite[Thm.~8.3]{K08}), and that they coincide
 under the condition \eqref{eqn:lmdpos}. Hence \eqref{eqn:OH2} is bijective.
 \end{proof}

In order to clarify the dependence of the parameter $\ell$, 
 we write ${\mathcal{W}}_{\ell}$ 
 for the $G'$-equivariant vector bundle corresponding to $\ell \in \Lambda$ from now.  
Then Fact \ref{fact:4.3} tells that the one-dimensional vector space
\[
   \mathrm{Hom}_{G'}\left( \mathcal O(X,\mathcal L)\big\vert_{G'},\mathcal O(Y,\mathcal W_\ell)\right)
\simeq 
 \mathrm{Hom}_{G'}\left( \mathcal H^2(X,\mathcal L)\big\vert_{G'},\mathcal H^2(Y,\mathcal W_{\ell})\right)
\]
 is spanned by a
\emph{differential} symmetry breaking operator. We fix such a generator $D_\ell$ for every $\ell\in\Lambda$.

Since $D_\ell \colon \mathcal H^2(X,\mathcal L) \to \mathcal H^2(Y,\mathcal W_\ell)$ is a continuous operator
between the Hilbert spaces,
its operator norm $\Vert D_\ell\Vert_{\mathrm{op}}$ is finite and its adjoint
$D_\ell^*$ is a continuous linear operator. Set
$$
 C_\ell:=\Vert D_\ell\Vert_{\mathrm{op}}^2, \quad \Psi_\ell:=\frac1{C_\ell} D_\ell^*.
$$
Let $D=(D_{\ell})_{\ell \in \Lambda}$ be the symmetry breaking transform.  
Then we have the following:
\begin{thm}\label{thm:SBTinj}
Suppose we are in Setting \ref{setting:takagi18}.
\begin{itemize}
\item[(1)] $\Psi_\ell\colon \mathcal H^2(Y,\mathcal W_\ell)\To\mathcal H^2(X,\mathcal L)$ is a holographic operator. Moreover, it is an isometry up to renormalization.

\item[(2)] The symmetry breaking transform $D$ is injective on $\mathcal H^2(X,\mathcal L)$.

\item[(3)]
Any $f\in\mathcal H^2(X,\mathcal L)$ is recovered from its symmetry breaking transform $Df$ by
$$
f=\sum_{\ell\in\Lambda}\Psi_\ell \left( Df\right)_{\ell}.
$$
\item[(4)] The norm $\Vert f\Vert_{\mathcal H^2(X,\mathcal L)}$ is recovered from the sequence of
norms $\Vert \left( Df\right)_{\ell}\Vert_{\mathcal H^2(Y,\mathcal W_\ell)}$ by
$$
\Vert f\Vert_{\mathcal H^2(X,\mathcal L)}^2=
\sum_{\ell\in\Lambda}\frac1{C_\ell}
\Vert \left( Df\right)_{\ell}\Vert_{\mathcal H^2(Y,\mathcal W_\ell)}^2.
$$
\end{itemize}
\end{thm}
\begin{proof}
The unitary representation of $G$ on the Hilbert space $\mathcal H^2(X,\mathcal L)$ is decomposed discretely and multiplicity freely into the Hilbert
direct sum:
\begin{equation}\label{eqn:HKSK}
\mathcal H^2(X,\mathcal L)|_{G'}
 \simeq \sum_{\ell\in\Lambda}{}^{\oplus}\mathcal H^2(Y,\mathcal W_\ell)
\end{equation}
as unitary representations of the subgroup $G'$
by Fact \ref{fact:4.3}.  
\par\noindent
(1)\enspace
 The adjoint operator $D_\ell^*$ is a $G'$-homomorphism because 
 $\mathcal H^2(X,\mathcal L)$ and both $\mathcal H^2(Y,\mathcal W_\ell)$ are unitary representations. The second assertion follows from Schur's lemma
 because $\mathrm{Hom}_{G'}(\mathcal H^2(Y,\mathcal W_\ell), \mathcal H^2(X,\mathcal L)\big\vert_{G'})$ is one-dimensional.
\par\noindent
(2)\enspace
Expand $f\in\mathcal H^2(X,\mathcal L)$ as $f=\sum_{\ell\in\Lambda}f_\ell$
according to the decomposition \eqref{eqn:HKSK}. 
Then $(Df)_\ell$ is a nonzero multiple of $f_\ell$ by Schur's lemma since the decomposition \eqref{eqn:HKSK}
is multiplicity free. Hence, if $Df=0$, then $f_\ell=0$ for all $\ell\in\Lambda$, and therefore $f=0$.
\par\noindent
Statements (3) and (4) are direct consequences of Lemma \ref{lem:Hilbertdeco}.
\end{proof}

\subsection{Role of orthogonal polynomials}\label{sec:concl3} 
In this section we investigate Problems A and B in Setting \ref{setting:takagi18}, and clarify
the role of the F-method and special orthogonal polynomials for the $L^2$-theory of symmetry breaking transforms consisting of holomorphic differential operators.

Suppose we are in  Setting \ref{setting:takagi18}. 
As we have seen in Section \ref{sec:concl2}, 
Theorem \ref{thm:SBTinj} (2) solves Problem A.0, whereas 
Theorem \ref{thm:SBTinj} (1), (3), and (4) give a framework for Problems A.1, A.2, and B, respectively,
 for the symmetry breaking 
transform $D=(D_\ell)_{\ell \in \Lambda}$. 
Thus the solution to Problems A and B is reduced to the following four 
questions of finding explicit description and closed formul{\ae} of
\begin{itemize}
\item the support of $\Lambda$;
\item holomorphic differential operators $D_\ell$;
\item the operator norm $\| D_\ell \|_{\operatorname{op}}$;
\item the adjoint operator $D_\ell^*$.
\end{itemize}

Let us summarize briefly what was known, what has been proved in this article,
and what looks promising.

As we mentioned in Section \ref{sec:concl2},
the explicit description of $\Lambda$, equivalently, the 
branching law \eqref{eqn:HKSK} for the restriction $\pi\vert_{G'}$ in Section \ref{sec:concl1} was proved  in 
\cite[Thm. 8.3]{K08}, which gives a generalization of the Hua--Kostant--Schmid formula in the case when $G'=K$. 
Denote by $\mathrm{rank}_\R G/G'$ the split rank of the reductive symmetric space
$G/G'$.
Then it
turns out
 that $\Lambda$ is a free abelian semigroup generated by 
 $\mathrm{rank}_\R G/G'$ elements, 
 see \cite{K08}.

It is more involved to construct symmetry breaking operators $D_{\ell}$
 explicitly 
than determining $\Lambda$, namely, the
branching law of the restriction $\pi\vert_{G'}$. As of now, an explicit construction of $D_\ell$ for all $\ell\in\Lambda$ with exhaustion theorem is known when $\operatorname{rank}_\R G/G'=1$,
see \cite{KP16b}. There are 
six families of such symmetric pairs $(G,G')$,
 and the resulting symmetry breaking operators include
 the Rankin--Cohen operators and the holomorphic Juhl operators.

In order to obtain the operator norm $\Vert D_\ell\Vert_{\mathrm{op}}$
 of such holomorphic differential operators $D_{\ell}$, we
 have developed the idea of the F-method to connect $\Vert D_\ell\Vert_{\mathrm{op}}$ with the $L^2$-norm
of special polynomials $P_\ell$ in the following two cases in this article.
\vskip7pt
{\small{\centerline{
\begin{tabular}{c|c}
$\qquad D_\ell$ & $\qquad P_\ell$\\
  \hline
  &\\
  Rankin--Cohen operators & Jacobi polynomials \\&\\
  \hline
  &\\
  Juhl operators & Gegenbauer polynomials
\end{tabular}}}}
\vskip7pt
\noindent
The relationship between $D_\ell$ and $P_\ell$ follows from the fact that the $G'$-equivariance condition on the operator $D_\ell$ is transformed into a certain differential equation 
 (\emph{e.g.}\ Jacobi differential equation \eqref{eqn:JacobiDE})
 for the polynomial $P_\ell$. 
It is plausible that this idea would work in the full generality of Setting~\ref{setting:takagi18}.

Concerning the adjoint operator $D_\ell^*$, this article has provided two kinds of integral formul{\ae}, that is, by the line integral (Definition \ref{def:HOtensor}), see
Proposition \ref{prop:PsiRC}, and by the integral over the tube domain (Theorem
\ref{thm:CIRM18}). The former has an advantage that the formula is simple and does not require the unitarity of representations, whilst the latter uses a natural idea of the ``relative reproducing kernel" $K_{\lambda,\nu}(\zeta,\tau')$, see
\eqref{eqn:holoKlmdnu}.


\section{Appendix: Jacobi polynomials and Gegenbauer polynomials}\label{sec:appendixJacob}
\subsection{The Jacobi polynomials}
Suppose $\alpha,\beta\in\C$ and $\ell\in\N$.
The Jacobi polynomial $P^{\alpha,\beta}_\ell(t)$ is a polynomial solution to the
Jacobi differential equation
\begin{equation}\label{eqn:JacobiDE}
\left((1-t^2)\frac{d^2}{dt^2}+(\beta-\alpha-(\alpha+\beta+2)t)\frac d{dt}+\ell(\ell+\alpha+\beta+1)\right)y=0,
\end{equation}
which is normalized  by
 $P^{\alpha,\beta}_\ell(1)=\frac{\Gamma(\alpha+\ell+1)}{\Gamma(\alpha+1)\ell!}=
 \frac{(\alpha+1)_\ell}{\ell!}$. Then it satisfies the Rodrigues formula
\begin{equation}\label{eqn:JRodrigues}
(1-t)^\alpha(1+t)^\beta P_\ell^{\alpha,\beta}(t)=\frac{(-1)^\ell}{2^\ell\ell!}\left(\frac d{dt}\right)^\ell\left((1-t)^{\ell+\alpha}(1+t)^{\ell+\beta} \right).
\end{equation}

The Jacobi polynomial
 $P^{\alpha,\beta}_\ell(t)$ is
nonzero and is
 a polynomial of degree $\ell$ for generic parameters (see \cite[Thm. 11.2]{KP16b} for the precise condition).
Explicitly, one has
\begin{eqnarray}\label{eqn:Jacobi}
 P^{\alpha,\beta}_\ell(t)&=&
\frac{(\alpha+1)_\ell}{\ell!}
 {}_2F_1\left(-\ell,\alpha+\beta+\ell+1;\alpha+1;\frac{1-t}2\right) \nonumber
 \\
 &=&
\sum_{j=0}^\ell
\frac{(\alpha+\beta+\ell+1)_j(\alpha+j+1)_{\ell-j}}{(\ell-j)!j!}\left(\frac{t-1}2\right)^j.
\end{eqnarray}

The first Jacobi polynomials are
\begin{itemize}
\item $P_0^{\alpha,\beta}(t)= 1,$
\item
$P_1^{\alpha,\beta}(t)=\frac{1}{2} (\alpha-\beta+(2+\alpha+\beta) t)$.
\end{itemize}
For real $\alpha,\beta$ with $\alpha,\beta>-1$, the Jacobi polynomials $\left\{P_\ell^{\alpha,\beta}\right\}_{\ell\in\N}$ form an orthogonal
basis in the Hilbert space $L^2\left((-1,1),(1-x)^\alpha(1+x)^\beta dx\right)$ with the following
norm (see \cite[page 301]{AAR} for example):
\begin{equation}\label{eqn:Pnorm}
\int_{-1}^1\left\vert P_\ell^{\alpha,\beta}(x)\right\vert^2(1-x)^\alpha(1+x)^\beta dx=
\frac{2^{\alpha+\beta+1}\Gamma(\ell+\alpha+1)\Gamma(\ell+\beta+1)}{(2\ell+\alpha+\beta+1)\Gamma(\ell+\alpha+\beta+1)\ell!}.
\end{equation}

When $\alpha=\beta$ these polynomials yield Gegenbauer polynomials (see 
\eqref{eqn:PtoC} below), and they further reduce to Legendre polynomials in the case when $\alpha=\beta=0$.

\subsection{The Gegenbauer polynomials}
For $\alpha\in\C$ and $\ell\in\N$, the Gegenbauer polynomial
 (or ultraspherical polynomial) $C^\alpha_\ell(t)$ 
 is defined by 
 \begin{equation}\label{eqn:gegen}
C^\alpha_\ell(t)=\sum_{k=0}^{\left[\frac\ell2\right]} a_k(\ell,\alpha) t^{\ell-2k},
\end{equation}
where $ a_k(\ell,\alpha)$ is given in \eqref{eqn:alalpha}. The Gegenbauer polynomials are special cases of the Jacobi polynomials by
\begin{equation}\label{eqn:PtoC}
\frac{(2\alpha+1)_\ell}{(\alpha+1)_\ell} P_\ell^{\alpha,\alpha}(x)= C_\ell^{\alpha+\frac12}(x),
\end{equation}
and have
 the generating function:
$$
(1-2t r+r^2)^{-\alpha}=\sum_{\ell\in\N} C_\ell^\alpha(t) r^\ell.
$$
We note that $C_\ell^\alpha(1)=\frac{(2\alpha)_\ell}{\ell!}$. If $\alpha>-\frac12$, then the Gegenbauer polynomials $\left\{ C_\ell^\alpha(v)\right\}_{\ell\in\N}$ form an
orthonormal basis in the Hilbert space $L^2\left((-1,1), (1-v^2)^{\alpha-\frac12}dv\right)$ with the
following
 $L^2$-norm (see \cite[7.313]{GR}):
\begin{equation}\label{eqn:L2Gegen}
\int_{-1}^1 \left\vert C_\ell^{\alpha}(v)\right\vert^2(1-v^2)^{\alpha-\frac{1}{2}}\,dv = \frac{\pi 2^{1-2\alpha}\Gamma(\ell+2\alpha)}{\ell!(\ell+\alpha)\Gamma\left(\alpha\right)^2}.
\end{equation}

\footnotesize{ \noindent  Addresses:
T. Kobayashi. Graduate School of Mathematical Sciences and  Kavli IPMU (WPI), The University of
 Tokyo,
3-8-1 Komaba, Meguro, Tokyo, 153-8914 Japan; \texttt{{
 toshi@ms.u-tokyo.ac.jp}}.\vskip5pt

\noindent M. Pevzner.
Laboratoire de Math\'ematiques de Reims, Universit\'e
de Reims-Champagne-Ardenne, FRE 2011 CNRS, F-51687, Reims, France; \texttt{{
 pevzner@univ-reims.fr.}}

\end{document}